\UseRawInputEncoding 
 \documentclass[11pt, a4paper]{article}

\usepackage{a4wide}
\usepackage{amsmath,amsfonts,amssymb,amsthm,enumerate,ulem,mathabx}

\usepackage{esint}

\usepackage{tikz}
\usetikzlibrary{intersections,decorations.markings}
\usepackage{pgfplots}
\usetikzlibrary{pgfplots.fillbetween}
\usepackage[colorlinks=true, pdfstartview=FitV, linkcolor=blue, citecolor=blue, urlcolor=blue,pagebackref=false]{hyperref}
\usepackage{bm}
\newenvironment{changemargin}[2]{\begin{list}{}{%
\setlength{\topsep}{0pt}%
\setlength{\leftmargin}{0pt}%
\setlength{\rightmargin}{0pt}%
\setlength{\listparindent}{\parindent}%
\setlength{\itemindent}{\parindent}%
\setlength{\parsep}{0pt plus 1pt}%
\addtolength{\leftmargin}{#1}%
\addtolength{\rightmargin}{#2}%
}\item }{\end{list}}
\newcommand{\R}{\mathbb{R}}
\newcommand{\Z}{\mathbb{Z}}

\newcommand{\bv}{\bm{v}}

\newcommand{\Div}{{\rm{Div}}}

\providecommand{\R}{\mathbb{R}}

\providecommand{\Z}{\mathbb{Z}}

\providecommand{\eps}{\varepsilon}

\renewcommand{\leq}{\leqslant}
\renewcommand{\geq}{\geqslant}

\renewcommand{\div}{\operatorname{div}}

\newcommand{\dist}{\operatorname{dist}}

\newcommand{\Id}{\operatorname{Id}}
\renewcommand{\S}{\mathcal{S}}
\newcommand{\dd}{\, d}

\DeclareMathOperator{\Bog}{Bog}
\DeclareMathOperator{\Y}{Y}

\DeclareMathOperator{\diam}{diam}
\DeclareMathOperator{\supp}{supp}
\newcommand{\dmin}{d_{\min}}
\DeclareMathOperator{\dv}{\operatorname{div}}
\newcommand{\loc}{\mathrm{loc}}

\definecolor{darkgreen}{rgb}{0,0.5,0}
\definecolor{lightgreen}{rgb}{0,0.8,0}
\definecolor{darkblue}{rgb}{0,0,0.7}
\definecolor{darkred}{rgb}{0.9,0.1,0.1}
\definecolor{lightblue}{rgb}{0,0.51,1}

  \newtheorem{thm}{Theorem}[section]
  \newtheorem{cor}[thm]{Corollary}
  \newtheorem{Lemma}[thm]{Lemma}
  \newtheorem{Proposition}[thm]{Proposition}
  \theoremstyle{definition}
    \newtheorem{Definition}[thm]{Definition}
  \theoremstyle{remark}
  \newtheorem{remark}[thm]{Remark}
\newtheorem{op}[thm]{Open Problem}

\usepackage{graphicx}
\usepackage{enumitem}
\usepackage{float}
\usepackage{color}
\usepackage[mathscr]{eucal}
\usepackage[sectionbib]{chapterbib}

\begin{document}

\pagenumbering{roman}

\newcommand{\Addresses}{{% additional braces for segregating \footnotesize
  \bigskip
  \footnotesize

\noindent  R.~H\"ofer, \textsc{Institut de Mathématiques de Jussieu-Paris Rive Gauche, Université de Paris,
    8~Place Aurélie Nemours,
F75205 Paris Cedex~13, France}\par\nopagebreak
\noindent  \textit{E-mail address:}\texttt{hoefer@imj-prg.fr}

  \medskip

\noindent  C.~Prange, \textsc{Laboratoire de Math\'ematiques AGM, UMR CNRS 8088, Cergy Paris Universit\'e, 2~avenue Adolphe Chauvin, F95302 Cergy-Pontoise Cedex, France}\par\nopagebreak
\noindent  \textit{E-mail address:} \texttt{christophe.prange@cyu.fr}

  \medskip

\noindent  F.~Sueur, \textsc{  Institut de Math\'ematiques de Bordeaux, UMR CNRS 5251,Universit\'e de Bordeaux, 351~cours
de la Lib\'eration, F33405 Talence Cedex, France   $\&$ Institut  Universitaire de France}\par\nopagebreak 
\noindent  \textit{E-mail address:} \texttt{Franck.Sueur@math.u-bordeaux.fr}

}}

\title{Motion of several slender rigid filaments in a Stokes flow}

\author{Richard M. H\"ofer, Christophe Prange, Franck Sueur}

\date{\today}

\maketitle
%%% ----------------------------------------------------------------------
\setcounter{page}{5}

\begin{abstract}

 We investigate the dynamics of several slender rigid bodies moving in a flow driven by the three-dimensional steady Stokes system in presence of a smooth background flow. More precisely we consider the limit where the thickness of these slender rigid bodies tends to zero with a common rate $\epsilon$, while their volumetric mass density is held fixed, so that the bodies shrink into separated massless curves. While for each positive $\epsilon$, the bodies' dynamics are given by the Newton equations and correspond to some coupled second-order ODEs for the positions of the bodies, we prove that the limit equations are decoupled first-order ODEs whose coefficients only depend on the limit curves and on the background flow. These coefficients appear through appropriate renormalized Stokes resistance tensors associated with each limit curve, and through renormalized Fax\'en-type force and torque associated with the limit curves and the background flow. We establish a rate of convergence of the curves  of order  $O( \vert \log \eps \vert^{-1/2})$. 
  We  also determine the limit effect due to the limit curves on the fluid,  in the spirit of the immersed boundary method.
 Both for the convergence of the filament velocities and the fluid velocities we identify an initial exponential relaxation within a  $O(\eps^2 \vert \log \eps \vert)$ time. 
\end{abstract}

\pagenumbering{arabic}\setcounter{page}{1}

\begin{changemargin}{-1cm}{-1cm}

\tableofcontents

\end{changemargin}

%%%%%%%%%%%%%

\section{Introduction}
\label{sec-intro}

In view of various applications in particular in biology (considering DNA) and in oceanography (considering sediment, plankton), the motion of rigid bodies with anisotropic shapes immersed in incompressible flows requests some mathematical analysis. 
In particular in view of modelling and numerics one may wish to consider asymptotic models where one gets rid of the dimensions corresponding to small extent of the rigid bodies. In this paper we consider the case of a  finite number $N$ of slender rigid bodies. We will use the terminology ``filaments". Their radii have the same smallness parameter $\eps$ in $(0 ,1)$ and shrink into curves in $\R^3$ as $\eps \rightarrow 0$. We will call these curves ``centerline curves" or ``filament centerlines" . We will assume: (i) that these filaments are immersed into an incompressible fluid driven by the steady Stokes system in presence of a background flow, (ii) that their dynamics is driven by the Newton equations with  forces acting on the filaments only due to the viscous stress tensor on their boundaries and (iii) that their volumetric mass density is independent of $\eps$, so that their limit centerline curves are massless. 
This leads to a system of second-order ODEs all coupled to each other through the Stokes equations. We refer to Section \ref{sec-setting} for a precise description of this so-called Newton-Stokes system, 
and for a straightforward local-in-time well-posedness result, see Lemma \ref{CL}. In particular this result establishes the existence of smooth solutions as long as there is no collision between filaments.

The main result of this paper, Theorem \ref{main} stated in Section \ref{sec-cv}, is the convergence of this Newton-Stokes system to a limit system describing the dynamics of the centerline  curves. We will also determine the limit effect due to these centerline  curves on the fluid,  in the spirit of the immersed boundary method. 

To give a flavour of our result to the reader we describe below how this limit system looks like. 
 The precise description of the limit dynamics is given in Section \ref{sec-mr-lim}.
To emphasize what concerns limit objects when $\eps\rightarrow 0$
we use the notation $\hat{\cdot}$ for a quantity $\cdot$. On the other hand, 
 to avoid heavy notations, we will only {make} the dependence on $\eps$ {explicit} when it is necessary in order to avoid confusion or to make precise in a quantitative way this dependence.  
Otherwise it {is} understood that the quantities at stake {can} depend on $\eps$ even if there is no corresponding index in the notation.
 The limit model drives the dynamics of the position at time $t$ of a collection of   smooth  curves $ \hat{{\mathcal C}}_i(t)$ without any self-intersection,  for  $ 1\leq i \leq N$, by the following rigid motion 
  \begin{align*} 
  \hat{{\mathcal C}}_i(t) = \hat{h}_i(t) + \hat{Q}_i(t) \bar {\mathcal C}_i .
  \end{align*}
 Here $\bar {\mathcal C}_i $ denotes a reference curve.
 On the other hand to emphasize what concerns limit objects when $\eps\rightarrow 0$ 
we use the notation $\hat{\cdot}$ for a certain quantity $\cdot$.
The vector $\hat{h}_i(t)$ in $\R^3$ and the matrix $\hat{Q}_i(t) $ in $SO(3)$ satisfy some first-order ODEs: 
  \begin{align} \label{in1}
 &\hat{h}_i'(t) = \hat{\textrm{v}}_i (t) ,  \quad
    \hat{Q}_i' (t) = (\hat{\omega}_i(t) \wedge \cdot) \hat{Q}_i (t),
     \,\\     \label{in2}
     &\text{and } \, 
  	(\hat{\textrm{v}}_i(t) , \hat{\omega}_i (t)) = F[\hat{{\mathcal C}}_i(t), u^\flat (t,\cdot) \vert_{\hat{{\mathcal C}}_i(t)}],
  \end{align}
where the notation $ (\hat{\omega}_i(t) \wedge \cdot)$
is used for the skew-symmetric matrix canonically associated with the wedge product by the vector $\hat{\omega}_i(t)$ in $\R^3$,  $u^\flat (t,\cdot) \vert_{\hat{{\mathcal C}}_i(t)}$ denotes the trace on $\hat{{\mathcal C}}_i(t)$ of  a smooth
 background flow velocity $u^\flat (t,\cdot)$. Moreover,  $F[\cdot,\cdot]$ is a universal operator acting on smooth simple curves and smooth incompressible vector fields.  This operator is given explicitly in \eqref{debaze}. 
Let us highlight that,  to ease   the  reading, in the left hand side of \eqref{in2}, and several more times below, 
we  identify $(v,w)$ with the corresponding column vector.
We emphasize that the limit dynamics \eqref{in1}-\eqref{in2} is a system of uncoupled first-order ODEs which reflects that both the mass of the filaments and their perturbation on the fluid tends to zero as $\eps$ converges to $ 0$.

  The main novelty of our work is to study the coupled dynamics of a collection of slender filaments. As far as we know, previous analytic works on filaments in Stokes flows are focused on static problems with only a single filament.
  Related to these works, the first part of our analysis establishes approximations for the forces and torques acting on the filaments as well as on the fluid perturbation caused by the filaments. 
  We show that explicit force distributions on the particle centerlines are sufficient to capture these quantities to leading order. This part of our analysis, carried out in Section \ref{sec:singularity.method}, is related to so-called slender body theory, and we believe that our results there, Theorem \ref{lem:slender.several} and Corollary \ref{prop:asymptotic.resistance} are of independent interest.  The second part of our analysis consists in the study of a system of singularly perturbed ODEs relying on a modulated energy argument. 
  We refer to Sections \ref{sec:strategy} and \ref{sec:structure} for a more extensive outline of the key elements and the structure of the proof of the main result,
  to Section \ref{compa-bib} for a discussion of some related results and to Section \ref{sec-mr-op} for some open problems.

\section{Setting of the problem}
\label{sec-setting}

This section is devoted to the description of the setting of the problem.

  \subsection{Geometry of the filaments}
  For each index $ 1\leq i \leq N$ 
  we consider a filament $\mathcal S_i$ which can be closed or non-closed.
 For $\eps \in (0 ,1)$, the filament is given in terms of a reference filament  $\bar{\mathcal S}_i$ which  is described by a centerline and a shape function for the cross section as follows. 
  For a non-closed filament, the centerline 
   $\bar {\mathcal C_i}$ is assumed to be a curve of length $L_i >0$, parametrized by arc length without self-intersections by a smooth function $\gamma_i \colon [0,L_i] \to \R^3$.
  We assume, for each index $ 1\leq i \leq N$, that the curve  $\bar {\mathcal C_i}$
  is not a straight line, i.e. $\gamma_i'' \neq 0$.
  The shape function is a smooth map $\Psi_i \colon [0,L_i] \times B_1(0) \to \R^2$, such that $\Psi_i(s,0) = 0$ and 
  $\Psi_i(s,\cdot)$ is a diffeomorphism to its image for all $s \in [0,L_i]$. Here, $B_1(0)$ denotes the open unit   ball in $\R^2$. Moreover, let $R_i \colon [0,L_i] \to SO(3)$ be a smooth function such that $R_i e_3 = \gamma_i'$, where $e_3=(0,0,1)$.
  Then, we define
\begin{align} \label{def:S^eps}
 {\bar{ \mathcal S}_i^\eps}= \bar{ \mathcal S}_i := \{ \gamma_i(s) + \eps  R_i(s)(\Psi_i(s,B_1(0)) \times \{0\}) : 0 \leq s \leq L_i \}. 
\end{align}

In the case of a closed filament, the definition is analogous but we replace the interval $[0,L_i]$ by $\R/L_i \Z$ for $\gamma_i$, $R_i$ and $ \Psi_i$.

\begin{remark} \label{rem:nonclosed.smooth}
    Note that the non-closed filaments are not smooth but only Lipschitz due to corners at their ends. With minor modifications of some arguments, our analysis also applies to smooth non-closed filaments, which could be defined as 
    \begin{align} 
 {\bar{ \mathcal S}_i^\eps} = \bar{ \mathcal S}_i := \{ \gamma_i(s) + \eps  a_\eps(s) R_i(s)(\Psi_i(s,B_1(0)) \times \{0\}) : 0 \leq s \leq L_i \}, 
\end{align}
where the additional function $a_\eps(s) : [0,L_i] \to \R$ is given by
\begin{align}
    a_\eps(s) = \begin{cases}
        \sqrt{\frac s \eps} &\quad \text { for } s \in [0,\eps], \\
        1 &\quad \text { for } s \in [\eps, L_i - \eps], \\
        \sqrt{\frac{L_i - s}{\eps}} & \quad \text{ for } s \in [L_i-\eps,L_i].
    \end{cases}
\end{align}
\end{remark}

We assume that the reference centerlines are centered at the origin in the sense that
\begin{align} \label{eq:center.mass.limit}
    \int_{\bar {\mathcal C}_i} x  \dd \mathcal H^1 = 0 ,
\end{align}
where $\dd \mathcal{H}^1$ is  the one-dimensional Hausdorff measure. 
We emphasize that the center of mass of the reference filaments $\mathcal S_i$, for  $ 1\leq i \leq N$, depends on $\eps$. For simplicity, we assume that the mass density is constant in each of the filaments. Then, their centers of mass are given as the following barycenters
\begin{align}
    \bar h_{i,\eps} := \fint_{\bar {\S}_i} x \dd x. 
\end{align}
By \eqref{eq:center.mass.limit}, we have
\begin{align} \label{eq:center.mass.small}
    |\bar h_{i,\eps}| \leq C \eps,
\end{align}
where the constant $C$ depends only on the functions specifying the reference filament, i.e.  $L_i$, $\gamma_i$, $R_i$, and $\Psi_i$.

We are interested in the limit of the dynamics (specified below) as $\eps \to 0$ for given, $\eps$-independent initial data for the centerlines of the filaments. More precisely, 
we fix $\hat h_i(0) \in \R^3$, $\hat Q_i(0) \in SO(3)$ such that
\begin{align}
\label{dimathC}
  \hat{\mathcal  C}_i (0) =  \mathcal C_i (0) := \hat h_i(0) + \hat Q_i(0) \bar {\mathcal C}_i ,
\end{align}
are the positions of the centerlines at time $0$ for all $\eps >0$.

Then, the center of mass $h_{i,\eps}(t)$ and the orientation  $Q_{i,\eps}(t)$ of the filament at time $t$ have initial data
\begin{align} \label{eq:initial.positions}
    h_{i,\eps}(0) = \hat h_i(0) + {\hat Q_{i}(0)} \bar h_{i,\eps} \quad \text{ and } \quad  Q_{i,\eps}(0) = \hat Q_i(0),
\end{align}
and we denote the filament at time $t$ with parameter $\eps$ by
  \begin{align} \label{def:Stourne}
  \mathcal S_i(t) :=  \mathcal S_{i,\eps}(t) = h_{i,\eps}(t) + {Q_{i,\eps}}(t) (\bar {\mathcal S}_i - \bar h_{i,\eps}) ,
  \end{align}
  and similarly for the centerline $\mathcal C_i(t)$: 
  \begin{align*} 
  \mathcal C_i(t) :=  \mathcal C_{i,\eps}(t) = h_{i,\eps}(t) + {Q_{i,\eps}}(t) (\bar {\mathcal C}_i - \bar h_{i,\eps}) .
  \end{align*}

\begin{figure}  

\centering

\begin{tikzpicture}

\draw[line width=9pt, blue!30] (0,0.44) to[out=0, in=240] (1, 0.3) 
to[out=60, in=300] (2, 1.52)
to[out=120, in=0] (1.53, 3)
to[out=180, in=40] (0.44, 2.03)
to[out=220, in=0] (-0.63, 2.61) 
to[out=180, in=60] (-1.63, 1.02) 
to[out=-120, in=180] (0, 0.44) ;

\draw (0,0.4) to[out=0, in=240] (1, 0.3) 
to[out=60, in=300] (1.97, 1.5)
to[out=120, in=0] (1.5, 3)
to[out=180, in=40] (0.4, 2)
to[out=220, in=0] (-0.6, 2.6) 
to[out=180, in=60] (-1.6, 1) 
to[out=-120, in=180] (0, 0.4) ;

\draw[red] plot [smooth cycle, tension=0.6] coordinates {(4.4,0.4) (5,0.2) (5.8,0.4) (6.5,0.5)(6.4,1.1)  (5.9,1.2) (5.4,1.4) (5,0.9) (4.6,0.8) };
\draw[fill] (5.2,0.6) circle [radius=0.025];

\draw[red] plot [smooth cycle, tension=0.6] coordinates {(4.4,0.4) (5,0.2) (5.8,0.4) (6.5,0.5)(6.4,1.1)  (5.9,1.2) (5.4,1.4) (5,0.9) (4.6,0.8) };

\begin{scope}[scale=0.17]
\draw[shift={(6,9)}, red] plot [smooth cycle, tension=0.6] coordinates {(4.4,0.4) (5,0.22) (5.8,0.4) (6.5,0.48)(6.4,0.97)  (5.9,1.07) (5.4,1.2) (5,0.8) (4.6,0.72) };
\end{scope}

\draw[dashed] (1.9,1.55)--(4.5,0.3);
\draw[dashed] (1.9,1.75)--(5.4,1.4);
\draw[dash pattern={on 1pt off 1pt}] (-0.9,0.55)--(-1,1);
\draw[dash pattern={on 1pt off 1pt}] (0.2,1.42)--(0.3,1.94);

\node at (0.2,1) [anchor=south] {{$\bar{\mathcal C}_i$}};
\node at (-1.3,1.2) [anchor=west, blue] {{$\bar {\mathcal S}_i$}};

\node at (5.6,0.5) [anchor=south] {{$\gamma_i(s)$}};

;\end{tikzpicture}

  \caption{A closed reference filament}

\end{figure}
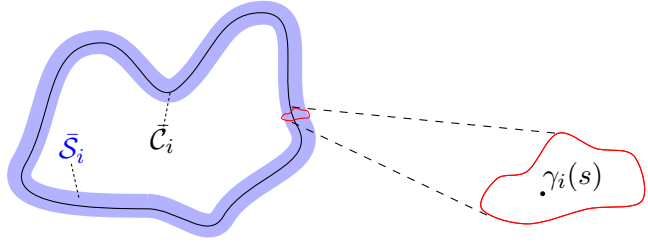

 \subsection{Kinematics  of the filaments}
For any $t \geq 0$, we denote by  $\omega_i(t)$ in $\R^3$ the unique angular velocity
of the $i$-th filament such that  
\begin{equation}
  \label{satane}
 Q_{i,\eps}' (t) Q_{i,\eps}^{T} (t)  = (\omega_i(t) \wedge \cdot) ,
 \end{equation}
where $Q_{i,\eps}^{T}$ denotes the transpose matrix of $Q_{i,\eps}$ and $ (\omega_i(t) \wedge \cdot)$ denotes the skew-symmetric matrix canonically associated with the wedge product by vector $\omega_i(t)$. 
We also set
\begin{equation}
  \label{det-v-rm}
 \textrm{v}_i(t) := h_{i,\eps}'(t)  . 
\end{equation}
Accordingly, the solid velocities are given by
\begin{equation}
  \label{vit-sol}
v^{{\mathcal S}_i}(t,x) := \textrm{v}_i(t)  + \omega_i(t) \wedge (x-h_{i,\eps}(t)),
\end{equation}
for all $x\in\mathcal S_i(t)$.

We highlight  that all these quantities depend implicitly on $\eps$ which we usually omit in the notation except for the quantities $h_{i,\eps},Q_{i,\eps}$. For these, we will always write the $\eps$ to avoid confusion with functions depending on variables $h_i,Q_i$ that will appear later.

 \subsection{Inertia  of the filaments}
We assume that the filaments' volumetric  density is fixed, and we 
denote  by  $\eps^2 m_{i} >0$ the mass of $\mathcal S_i$ and by
 $\eps^2 {\mathcal J}_{i} (t)$
 the  inertial matrix at time $t \geq 0$, so that $m_{i}$ and ${\mathcal J}_{i}$ are of order one with respect to $\eps$.
 Moreover  the matrix
 ${\mathcal J}_{i}$ is  positive definite, uniformly in $\eps$ (this only fails if $\gamma_i$ was a straight line, which has been explicitly excluded)
 and evolves in time 
 according to Sylvester's law:
\begin{align} \label{eq:Sylvester}
{\mathcal J}_{i} (t) = Q_{i,\eps}(t)  \mathcal{J}_{0,i} Q_{i,\eps}^{T} (t) ,
\end{align}
where $ \mathcal{J}_{0,i} $ denotes the initial value $\mathcal{J}_{0,i}  := {\mathcal J}_{i} (0)$.

  \subsection{Ambient fluid}
We assume that, for any $t\geq0$, 
 the open set 
$$\mathcal{F}(t) := \R^3 \setminus  \cup_i \mathcal S_i (t) ,$$
 is occupied by a  fluid whose velocity $u$ and pressure $p$  are given as the sums
\begin{equation}
  \label{vit-tot}
u := u^\flat + u^\mathfrak{p}  \quad   \text{ and } \quad  p := p^\flat + p^\mathfrak{p} ,
\end{equation}
where
 \begin{gather} \nonumber 
  (u^\flat ,p^\flat ) \in C \big( [0,  {+ \infty});{\dot H^1 ( \R^3) {\times L^2(\R^3)} }) \cap  \big( W^{2,\infty}((0,+\infty)\times \R^3) \times W^{1,\infty}((0,+\infty)\times \R^3)\big) 
  \\ \text{  satisfying }  \div u^\flat = 0   ,  \label{regover}
\end{gather}
 is the background flow, and $(u^\mathfrak{p} , p^\mathfrak{p})$ 
is the perturbation flow due to the filaments,  whose evolution is assumed to be  driven by the steady Stokes equations:
\begin{subequations}
\label{intrau-tot}
\begin{gather}
\label{intrau1}
\displaystyle  - \Delta u^\mathfrak{p}   + \nabla p^\mathfrak{p} =0 \quad  \text{ and } \quad \operatorname{div} u^\mathfrak{p}   = 0  \quad \text{in }\mathcal{F}(t), \\
\label{intrau1bis}
u^\mathfrak{p} = v^{{\mathcal S}_i}  - u^\flat \quad \text{in}\ \    \mathcal{S}_i  (t), \ \ 1 \leq i \leq N. 
\end{gather}
\end{subequations}

 \subsection{Dynamics of the filaments}
The filaments  are assumed to be only accelerated, for any $t\geq0$, by the force exerted by the fluid  on their
boundaries $\partial  \mathcal S_i (t)$  according to  the Newton equations: 
\begin{subequations}
 \label{intrauN}
\begin{gather}
 \label{intrau2}
 \eps^2 m_i \textrm{v}_i' (t) =  -\int_{\partial  \mathcal S_i (t)} \Sigma(u,p)  n \,d\mathcal{H}^2 , 
 \\  \label{intrau2bis}
 \eps^2 (\mathcal{J}_i \omega_i)' (t) = -\int_{\partial  \mathcal S_i (t)} (x-h_{i,\eps}(t)) \wedge \Sigma(u,p)  n \,d\mathcal{H}^2 ,
\end{gather}
\end{subequations}
where $d\mathcal{H}^2$ is  the two-dimensional Hausdorff measure and $n$ denotes  the unit normal vector on $\partial  \mathcal S_i (t) $ 
 pointing outside the fluid domain $\mathcal{F}(t) $ and
\begin{eqnarray}
\label{tensor}
 \Sigma(u,p) := 2 D(u) - p \Id ,
\end{eqnarray}
 where $D(u)$ is the deformation tensor defined by 
\begin{eqnarray}
\label{defo}
 D(u) :=  \frac{1}{2} ( \partial_{j} u_{i} +  \partial_{i} u_{j}  )_{1 \leqslant i,j \leqslant  3} .
\end{eqnarray}

 \subsection{The whole Newton-Stokes  system at a glance}
Gathering  \eqref{def:Stourne}, \eqref{satane},   \eqref{det-v-rm}, \eqref{vit-sol}, \eqref{eq:Sylvester},   \eqref{vit-tot}, \eqref{intrau-tot} and  \eqref{intrauN}
we arrive at the following 
Newton-Stokes  system.

 For $1 \leq i \leq N$,
\begin{subequations}
 \label{N-S}
  \begin{align} 
  h_{i,\eps}' (t)&= \textrm{v}_i (t) , \label{e.hi}
  \\ Q_{i,\eps}' (t)    &= (\omega_i(t) \wedge \cdot) Q_{i,\eps} (t),\label{e.Qi}
  \\ \eps^2 \, m_i \,  \textrm{v}_i' (t) &= - \int_{\partial  \mathcal S_i (t)} \Sigma(u^\flat + u^\mathfrak{p},p^\flat + p^\mathfrak{p})  n \,d\mathcal{H}^2 , \label{eq:v'}
 \\   \eps^2 (\mathcal{J}_i \omega_i)' (t) &= - \int_{\partial  \mathcal S_i (t)} (x-h_{i,\eps} (t)) \wedge \Sigma(u^\flat + u^\mathfrak{p},p^\flat + p^\mathfrak{p})  n \,d\mathcal{H}^2 , \label{eq:omega'}
 \end{align}
 \begin{gather} \text{ where } {\mathcal J}_{i} (t) = Q_{i,\eps}(t)  \mathcal{J}_{0,i} Q_{i,\eps}^{T} (t) \, 
\text{ and } \, \mathcal S_i(t) = h_{i,\eps}(t) + Q_{i,\eps}(t) (\bar {\mathcal S}_i - \bar h_{i,\eps}) \label{eq:J(t)},
\end{gather}
and
    \begin{gather} \label{ssys1}
    \displaystyle  - \Delta u^\mathfrak{p}   + \nabla p^\mathfrak{p} =0 \quad  \text{ and } \quad \operatorname{div} u^\mathfrak{p}   = 0  \quad \text{in }\mathcal{F}(t), \\ \label{bc-u}
u^\mathfrak{p} = v^{{\mathcal S}_i}  - u^\flat \quad \text{for}\ \  x\in  \mathcal{S}_i  (t) ,
\quad \text{ for  } \, 1 \leq i \leq N,
\\  \text{ where } v^{{\mathcal S}_i}(t,x) := \textrm{v}_i(t)  + \omega_i(t) \wedge (x-h_{i,\eps}(t))   \,    \text{ for } \,  x \in \mathcal S_i (t). \label{eq:v^S}
\end{gather}
\end{subequations}
A reformulation of the Newton equations  \eqref{eq:v'}-\eqref{eq:omega'} into a 
compact form, involving in particular the so-called Stokes resistance matrices,   will be  given in Section \ref{refor}.

 \subsection{A local-in-time well-posedness result}
Despite its apparent complexity,  the system \eqref{N-S} can be considered as a system of second-order quasilinear ODEs on the $6N$ degrees of freedom of the rigid bodies, the fluid state being  
given by an auxiliary steady Stokes system for which time only appears as a parameter. 
Moreover the coefficients of this ODE, although
their coefficients are given in a rather non  explicit way, 
are smooth as long as the filaments $\mathcal S_i (t)$ remain separated; this follows from standard results on the regularity with respect to shape changes for which we refer for example to \cite{simon,BNDRDM1,BNDRDM2}.
Therefore it 
follows from the Cauchy-Lipschitz theorem that,
starting from separated positions with arbitrary velocities, 
we have 
the following local-in-time well-posedness result.
\begin{Proposition} \label{CL}
For each $\eps$ in $(0,1)$, 
given some initial  disjoint positions and some initial velocities of the filaments, 
there is $T^{\max}_\eps \in (0,+\infty]$ and a unique smooth solution to    \eqref{N-S} on $[0,T^{\max}_\eps)$. Moreover, if $ T^{\max}_\eps 
<+\infty$, then 
\begin{align} \label{explo}
	\lim_{t \to T_\eps^{\max}} \min_{i \neq j} \dist(\mathcal S_i(t),\mathcal S_j(t)) = 0.
\end{align}
\end{Proposition}
\begin{proof}
As mentioned above, the existence, locally in time, of a  smooth solution to \eqref{N-S} is a straightforward consequence of the classical regularity properties of  the Stokes system and of the Cauchy-Lipschitz theorem. It remains to prove 
the last statement regarding the lifetime of these solutions. 
To this end,  we multiply, for  $ 1\leq i \leq N$,
the equation  \eqref{eq:v'} by 
 $\textrm{v}_i $ and the equation  
  \eqref{eq:omega'} by $\omega_i$.   
  By summing the resulting identities, and recalling \eqref{eq:v^S}, we get
  \begin{align} \nonumber
\eps^2 \sum_{1\leq i \leq N}  \Big(   \frac12 m_i \textrm{v}_i^2
 +   \mathcal{J}_i \omega_i \cdot \omega_i \Big)'
  =\ &- \sum_{1\leq i \leq N} \int_{\partial  \mathcal S_i (t)}  v^{{\mathcal S}_i}
 \cdot \Sigma(u^\mathfrak{p} + u^\flat,p^\mathfrak{p} + p^\flat)  n \,d\mathcal{H}^2 
 \\ \nonumber
 =\ &- \sum_{1\leq i \leq N} \int_{\partial  \mathcal S_i (t)}   u^\mathfrak{p}
 \cdot  \Sigma(u^\mathfrak{p},p^\mathfrak{p})  n \,d\mathcal{H}^2 \\  \nonumber
\quad &- \sum_{1\leq i \leq N} \int_{\partial  \mathcal S_i (t)}   u^\flat
 \cdot  \Sigma(u^\mathfrak{p},p^\mathfrak{p})  n \,d\mathcal{H}^2 
\\ \quad &- \sum_{1\leq i \leq N} \int_{\partial  \mathcal S_i (t)}   v^{{\mathcal S}_i}
 \cdot \Sigma(u^\flat,p^\flat)   n \,d\mathcal{H}^2 ,
 \label{labpalab}
 \end{align}
 thanks to   \eqref{bc-u}. 
 Hence, integrating by parts 
 in ${\mathcal F}$ for the two first terms in the right hand side of \eqref{labpalab} 
 taking into account \eqref{ssys1}, and integrating by parts in $\cup_i {\mathcal S}_i$ for the last term  in the right hand side of \eqref{labpalab}, we arrive at
  \begin{gather} \label{nrj}
  \eps^2 \sum_{1\leq i \leq N}  \Big(   \frac12 m_i \textrm{v}_i^2
 +   \mathcal{J}_i \omega_i \cdot \omega_i \Big)' + \int_{\mathcal F}  D u^\mathfrak{p} : D u^\mathfrak{p}
 =\   - \int_{\mathcal F}  D u^\flat : D u^\mathfrak{p} 
 \\ \nonumber \quad -   \sum_{1 \leq i \leq N} \int_{\mathcal S_i} (-\Delta u^\flat + \nabla p^\flat) \cdot  v^{{\mathcal S}_i}. 
\end{gather}
 This identity holds true 
 as long as there is no collision. 
Then, by the Cauchy-Schwarz inequality, Young's inequality for products and a Gronwall argument, we deduce that the function
$$ 
t \mapsto  \eps^2 \sum_{1\leq i \leq N}  \Big(   \frac12 m_i \textrm{v}_i^2
 +   \mathcal{J}_i \omega_i \cdot \omega_i \Big) (t) + \int_0^t \int_{\mathcal F}  D u^\mathfrak{p} : D u^\mathfrak{p},
 $$
 remains bounded as long as there is no collision. Then it follows from classical blowup criteria for ODEs that the solution can be continued as long as there is no collision. In particular, if the maximal lifetime $ T^{\max}_\eps $ of the smooth solution to    \eqref{N-S}  satisfies $ T^{\max}_\eps < +\infty$, then 
 \eqref{explo} holds true. 
\end{proof}
\begin{remark}
\label{rk-nrj}
 It is worth to observe that the energy identity \eqref{nrj} used in the proof above,  alone, is not sufficient to obtain bounds on the filament velocities which are uniform with respect to $\eps$ as the $O(1)$ energy transfer with the background flow, see the right hand side of \eqref{nrj}, that may \textit{a priori} lead to high velocities due to the factor $\eps^2$ associated with the filaments' inertia. 
\end{remark}
%

%%%
\section{Main results}
\label{sec-mr}

This section is devoted to the statements of the main results of the paper. 
More precisely, the main result of this paper, that is 
 the convergence of the Newton-Stokes system to a limit system as the thickness parameter $\eps$ goes to $0$,  is given in Section \ref{sec-cv}, in particular in Theorem \ref{main}.
To state this result, a few notations have to be introduced, which is the subject to the Subsections \ref{sec-gn}, \ref{sec-mr-lim} and \ref{sec-afb}.
In Section \ref{sec:strategy} we expose the strategy of the proof of Theorem \ref{main} by considering a toy model. 
The organization of the proof of Theorem \ref{main}  is detailed in Section 
\ref{sec:structure}. 
We will also draw some comparisons with the existing literature on close issues, see Subsection \ref{compa-bib},  and we will finally mention a few open problems, see Section \ref{sec-mr-op}.

\subsection{A few general notations}
\label{sec-gn}

First we introduce, for $1 \leq i \leq N$, 
the vector fields:
\begin{equation} \label{t1.6}
v_{i,\alpha}{[h_i]} (x) := \left\{\begin{array}{ll} 
e_\alpha & \text{if} \ \alpha =1,2,3 ,\\ \relax
  e_{\alpha-3} \wedge (x - h_i) & \text{if} \ \alpha=4,5,6 ,
\end{array}\right.
\end{equation}
where $e_\alpha $, for $\alpha=1,2,3$, denotes the $\alpha$-th unit vector of the canonical basis of $ \R^{3}$.
These vector fields are elementary rigid velocities with respect to the $i$-th filament.

We define, for $p$ in $\R^3$, the $3 \times 3$ matrix
\begin{equation} \label{ki}
k(p) := 8\pi\left(\Id - \frac12 p \otimes  p\right).
\end{equation}
One may wonder right away why such a matrix $k$  could occur in the discussion in the present setting. It is actually related to the Stokes kernel $S$, defined for $x$ in $\R^3 \setminus \{ 0 \}$, by 
\begin{equation}\label{e.defS}
S(x) = \frac 1 {8 \pi |x|} \left( \Id + \frac x {|x|} \otimes \frac x {|x|} \right) 
	=   \frac 1 {|x|}  S_0\left( \frac x {|x|} \right) ,
\end{equation}
where 
for $p$ in the euclidean unit  sphere $S^2$,
\begin{align}
	\label{notsu}
	S_0(p) := \frac 1 {8 \pi}(\Id + p \otimes p).
\end{align}
Then, by the Sherman-Morrison formula, we observe that
for any $p$ in $S^2 $,
 the matrix $S_0(p)$  is invertible and its inverse is precisely $k(p)$ defined above. 
As a matter of fact,  the use of the identity 
\begin{align} \label{eqqe}
\forall p\in S^2, \quad  S_0(p) k(p) = \Id ,
 \end{align}
is crucial in our analysis below, see \eqref{here}.
\ \par \
Next we associate with a smooth oriented curve $\mathcal{C}$ without  self-intersections and with two vector fields $ v$ and $\tilde v$ defined on $\mathcal{C}$ with values in $\R^3$, the following real-valued functional: 
	\begin{align} 
	\label{lau}
		I_{\mathcal C} [v,\tilde{v}]   &:=   \frac12
	 \int_{\mathcal{C}} k(\tau) v \cdot \tilde{v}   \, \dd \mathcal{H}^1   ,
	\end{align}
where we recall that $\dd \mathcal{H}^1$ is  the one-dimensional Hausdorff measure  and $\tau$ denotes the unit tangent vector field along $\mathcal{C}$. Since the matrix $k$ is symmetric, the operator $I_{\mathcal C} [\cdot,\cdot]$ is bilinear symmetric.

\subsection{Limit dynamics}
\label{sec-mr-lim}

Now {we define} the objects which occur in the limit dynamics of the filaments when the thickness parameter $\eps$ converges to $0$. To do so, we will define several functions depending on the filament positions denoted by $(h_i,Q_i)$, for  $1 \leq i \leq N$. 

For $1 \leqslant \alpha \leqslant 6 $ and for  $1 \leq i,j \leq N$, we set 
	\begin{align} \label{hatK}
	 \hat{\mathcal K}_{i,\alpha,j,\beta}  ({h_i,Q_i}) &:=  
	  \delta_{ij} \, 
	I_{\mathcal{C}_i  (h_i,Q_i)}  \big[v_{i,\alpha}[h_i],  v_{i,\beta}[h_i]\big]  ,
	\end{align}
	where $\delta_{ij}$ is the Kronecker symbol, and, 
 for $ \alpha =1,2,3 $, for  $1 \leq i \leq N$, for $t\geq 0$, 
\begin{align}
\label{hatfab}
 \hat{F}^\flat_{i,\alpha} ({t,h_i,Q_i}) := 
 	I_{\mathcal{C}_i  (h_i,Q_i)}  [ v_{i,\alpha}[h_i] ,
u^\flat(t)  ]  
\quad \text{ and } \quad \\ 
 \hat{T}^\flat_{i,\alpha} ({t,h_i,Q_i})   := 
 I_{\mathcal{C}_i  (h_i,Q_i)}  [   v_{i,\alpha+3}[h_i] ,
  u^\flat (t)   ]  .
\end{align}
Let us emphasize that the matrices 
$ \hat{\mathcal K}_{i,\alpha,i,\beta} $ and the vectors $\hat{F}^\flat_{i,\alpha} $ and $\hat{T}^\flat_{i,\alpha}$ do not depend on $\eps$ but on the positions, considered here as variables, of the filament centerlines denoted by $\mathcal{C}_i  (h_i,Q_i)$ and defined by 
\begin{align*}
  \mathcal{C}_i  (h_i,Q_i) :=  h_i +  Q_i \bar {\mathcal C}_i .
\end{align*}
In addition, the vectors $\hat{F}^\flat_{i,\alpha} $ and $\hat{T}^\flat_{i,\alpha}$ depend explicitly on time through the time dependence of $u^\flat$.
We refer to them respectively as the 
 Stokes resistance matrix associated with the filament centerline $\mathcal{C}_i $, and 
 the   Fax\'en force and torque  associated with the filament centerline $\mathcal{C}_i$ and with 
 the background flow $u^\flat$. 
 
 For  $1 \leq i \leq N$, let us consider the $6 \times 6$ matrices  diagonal blocks
\begin{align}  \label{hatKblock}
\hat{\mathcal K}_{i,i}{=\hat{\mathcal K}_{i,i}(h_i,Q_i)} := (\hat{\mathcal K}_{i,\alpha,i,\beta})_{1 \leq  \alpha,\beta \leq 6} ,
\end{align}
and, for  $1 \leq i \leq N$, 
the vectors of $\R^6$: 
\begin{align} \label{def:f.hat.flat.frac}
\hat{\mathfrak{f}}_i^\flat {=\hat{\mathfrak{f}}_i^\flat(t,h_i,Q_i) }:= ((\hat{F}^\flat_{i,\alpha})_{1 \leq  \alpha\leq 3} , (\hat{T}^\flat_{i,\alpha})_{1 \leq  \alpha\leq 3}).
\end{align}
By a change of coordinates, it is easy to see that $\hat {\mathcal K}_{i,i}$ 
satisfies
\begin{align} \label{eq:Sylvester.K}
     \hat{\mathcal K}_{i,i}(h_i,Q_i)
     = \begin{pmatrix} Q_i & 0 \\ 0 & Q_i \end{pmatrix}   \hat{\mathcal K}_{i,i}(0, \Id) \begin{pmatrix} Q_i^T & 0 \\ 0 & Q_i^T \end{pmatrix}. 
\end{align}
\begin{Lemma} \label{hat-inverse}
For any $(h_i,Q_i)$ in $\R^3 \times SO(3)$, for $1 \leq i \leq N$,  the matrix $\hat{\mathcal K}_{i,i}(h_i,Q_i)$ is symmetric positive definite. 
\end{Lemma}
\begin{proof}
By \eqref{eq:Sylvester.K}, it suffices to consider the case where $(h_i,Q_i) = (0,\Id)$, and we will omit to write this variable.
Since for any {$p\in S^2$}, the matrix $k(p)$ is positive symmetric and satisfies $k(p) \geq {4\pi}
\Id$,  we deduce that  for all $(v,\omega) \in \R^3 \times \R^3$
	\begin{align} \label{eq:K.hat.coercive}
		\hat{\mathcal K}_{i,i} \begin{pmatrix} v\\ \omega \end{pmatrix} \cdot (v,\omega) \geq 2\pi
		\int_{\mathcal{C}_i} | v + \omega \wedge x|^2 \dd \mathcal{H}^1  (x).
		%> 0.
	\end{align}
Indeed, for $\omega \neq 0$, the integrand is non-constant since $\mathcal{C}_i$ is not a straight line. Thus, the integral is positive for $(v,\omega) \neq 0$.
\end{proof}
With these tools in hands we can now explicitly present the system which will prove to be the zero-thickness limit of the Newton-Stokes system  \eqref{N-S}. This system drives the 
 dynamics of 
the positions at time $t$ of the centerline curves
$\hat{{\mathcal C}}_i(t)$,  for  $ 1\leq i \leq N$, by the  rigid motions 
  \begin{align} \label{def:Ctourne}
  \hat{{\mathcal C}}_i(t) = \hat h_i(t) + \hat Q_i(t) \bar {\mathcal C}_i .
  \end{align}
  Here the vector ${\hat h_i(t)}\in\R^3$ and the matrix ${\hat Q_i(t)} \in SO(3)$ satisfy the following
  first-order ODEs: 
  \begin{gather} \label{debaze}
 {\hat h}_i'(t) = 
 \hat{\textrm{v}}_i (t) , 
 \quad  {{\hat Q}_i' (t) = ({\hat \omega}_i(t) \wedge \cdot){\hat Q}_i^{T} (t) } ,
    \\
  \label{fav}  	(\hat{\textrm{v}}_i(t) ,{\hat \omega}_i (t)) =  	\hat{\mathcal K}_{i,i}^{-1} (\hat h_i(t), \hat Q_i(t))	\, \hat{\mathfrak{f}}_i^\flat (t,\hat h_i(t), \hat Q_i(t)) . 
      \end{gather}
For  $ 1\leq i \leq N$, the right hand side of \eqref{fav} only depends on ${\hat h}_i(t)$ and ${\hat Q}_i (t)$, not on the positions of the other centerline curves corresponding to $j\neq i$.

On the other hand neither the matrices $\hat{\mathcal K}_{i,i}$, nor their inverses, usually referred to as mobility matrices, are  diagonal, not even by $3 \times 3$ blocks. This coupling between translation/rotation velocities and force/torque is typical of the case of rigid bodies with shape anisotropies. {It} is usually called the Jeffery effect, see \cite{Jeffery,Junk-Illner,VM}.

Finally, since the coefficients of \eqref{fav} are smooth and globally Lipschitz (this follows immediately from \eqref{eq:Sylvester.K}, \eqref{def:f.hat.flat.frac}, \eqref{lau} and the smoothness assumption for $u^\flat$), 
the Cauchy-Lipschitz theorem applies again and guarantees 
the following global-in-time well-posedness result.
\begin{Proposition} \label{CL-limit}
Given some initial  disjoint positions  of the centerline curves, given a smooth background flow $u^\flat$ satisfying  \eqref{regover}, 
there is a unique smooth global-in-time solution to  \eqref{def:Ctourne}-\eqref{debaze}-\eqref{fav} on $[0,+\infty)$. 
\end{Proposition}
Although each of these decoupled ODEs admits a unique smooth global-in-time solution, it could be that some of the positions of the  centerline curves which they define collide in finite time.
\begin{Definition} \label{collision-limit}
Let us denote by  
 $\hat T $ in $(0,+\infty]$ the time of the first  collision in the limit dynamics, that is the first time for which at least two of the 
 centerline curves $\hat{\mathcal C_i}(t)$ defined by \eqref{def:Ctourne}-\eqref{debaze}-\eqref{fav} have a non-empty intersection, with the convention that $\hat T = +\infty$ if there is no such collision. More precisely, we define
 \begin{align}
    \hat d_{\min}(t) &:= \min_{i \neq j} \dist(\hat {\mathcal C}_i(t),  \hat {\mathcal C}_j(t)), \label{def:d.min.hat} \\
    \hat T &:= \inf\{ t \geq 0 : \hat d_{\min}(t) = 0 \} \label{def:hat.T}.
 \end{align}
\end{Definition}

\subsection{Asymptotic fluid behaviour}
\label{sec-afb}

Regarding the fluid behaviour  when the thickness parameter $\eps$ converges to $0$, it is only a matter to understand the behaviour of 
 the perturbation flow  $(u^\mathfrak{p} , p^\mathfrak{p})$ due to the filaments, since on the other hand the background flow $(u^\flat ,p^\flat ) $ is fixed. 
Precisely, the steady Stokes system in presence of several thin filaments has been the object of several studies usually referred to as the 
slender body theory or as the immersed boundary method. 
 It can also be viewed as a Stokesian counterpart of the issue of Newtonian capacity, see \cite{FNN}. 
 
 To capture the leading term of  the perturbation flow  $(u^\mathfrak{p} , p^\mathfrak{p})$  as $\eps$ converges to $0$, the key idea is to consider 
  the Stokes system in the full space $\R^3$ with an appropriate source term given as Dirac masses along the centerline curves $\mathcal C_i$. 
The intensity of these Dirac masses is related to the bilinear operator  $I_{\mathcal C} [\cdot,\cdot]$ defined in \eqref{lau} in the following way. 
Let $v$  a divergence-free vector field in  $W^{1,\infty}(  \cup_{1 \leq i \leq N}  \mathcal S_i )$. 
We define 
 $ \mu_{\mathcal C_i} [v]$ as the vector measure, supported on ${\mathcal{C}}_i$,  defined by
\begin{align} \label{def:mu(v)}
<	 \mu_{\mathcal C_i} [v], \phi >	 &:=   I_{{\mathcal{C}}_i} [v,\phi]  , \quad \text{ for any }  \phi \in C_c (\R^3 ; \R^3) .
\end{align}
Moreover we define the vector field
\begin{align} \label{def:U}
	U_{\mathcal C_i}[v]:=  S \ast  \mu_{\mathcal C_i} [v], 
\end{align}
where the symbol $\ast$ stands for the convolution in $\R^3$, $S$ is the Stokes kernel defined by \eqref{e.defS}.
This reads 
	\begin{align} \label{eq:U.formal}
U_{\mathcal C_i}[v](x)	= \frac12 \int_{\mathcal{C}_i}  S(x -y) k(\tau(y))  v( y) \, \dd \mathcal{H}^1  (y) ,
	\end{align}
for any $x\in\R^3\setminus\mathcal C_i$, 
 where we recall that $\tau$ is the unit tangent vector defined below \eqref{lau}.

Let us recall that the counterpart of the Stokes kernel $S$ for the 
pressure is the vector $P(x) $, defined for $x$ in $\R^3 \setminus \{ 0 \}$, by 
\begin{align} \label{noyauP}
	P(x) = \frac x {4 \pi |x|^3}  .
\end{align}
Notice that  
\begin{align}
\label{fonda}
	-\Delta S + \nabla P = \delta_{0} \Id \quad \text{ and } \quad  \dv S = 0 ,
\end{align}
in the sense of distributions, where the differential operators are applied column-wise.
Then we associate with the operator $U_{\mathcal C_i}$ the following counterpart for the pressure
\begin{align} \label{def:P}
	P_{\mathcal C_i}[v](x) := P \ast  \mu_{\mathcal C_i} [v],  
\end{align}
which reads {for any $x\in\R^3\setminus\mathcal C_i$},
\begin{align*} 
	P_{\mathcal C_i}[v](x) = {\frac 1 2}\int_{\mathcal{C}_i}  P(x -y) \cdot k(\tau(y))  v(y) \, \dd \mathcal{H}^1  (y)	.
\end{align*}
Thus it follows from 
\eqref{fonda} that, in the sense of distributions {in} 
the variable $x$, 
\begin{align}
\label{singm}
	-\Delta U_{\mathcal C_i}[v] + \nabla P_{\mathcal C_i}[v] =  \mu_{\mathcal C_i} [v]\quad \text{ and } \quad   \div  U_{\mathcal C_i}[v]   = 0 .
\end{align}
Our main result below, see Theorem \ref{main}, establishes that the vector field 
\begin{align}
\label{hatup}
{\hat u}^\mathfrak{p}(t,\cdot) := \sum_{1 \leq i \leq N} U_{\hat{\mathcal C}_i}[   {\hat v}^{{\mathcal S}_i}(t,\cdot)  - u^\flat (t,\cdot)  ],
\end{align}
where 
\begin{align}
\label{hatS}
{\hat v}^{{\mathcal S}_i}  (t,x) :=  {\hat h}_i' (t)  +  {\hat Q}_i' (t)     {\hat Q}_i (t)^T (x- {\hat h}_i(t)) ,
\end{align}
is the leading part of the perturbation flow ${u}^\mathfrak{p}$ 	 up to a renormalization factor $|\log \eps|^{-1}$. 
%
%%%%%%%%%%%%%%%%%%%%%%

\subsection{Convergence result}
\label{sec-cv}

The main result of this paper is the following theorem{, which contains two points: (i) an estimate of} the time of the first collision and {(ii) the convergence of} the dynamics of the filaments and {of a renormalized fluid} perturbation velocity as the thickness parameter $\eps$ of the filaments converges to zero. The following statement aims at {providing} a {simple} description of our results while some complementary more technical elements will be discussed below. 

\begin{thm}
\label{main}
We consider some initial  disjoint positions of the  centerline curves, a smooth background flow
$u^\flat$ satisfying  \eqref{regover}, and the solutions $({\hat h}_i , {\hat Q}_i)_{1 \leq i \leq N}$ given by Proposition \ref{CL-limit}. Let $\hat T $ in $(0,+\infty]$ the time of the first  collision associated with this solution as defined in Definition \ref{collision-limit}.
 For each $\eps$ in $(0,1)$, 
 the  initial  positions of the filaments of thickness parameter $\eps$ are deduced from the ones for the  centerline curves by \eqref{def:S^eps} and  
\eqref{def:Stourne}. Let $\kappa$ in $(0,1)$ such that for any 	$\eps$ in $(0,\kappa)$ the initial positions of the filaments are disjoints. 
Let us consider some initial rigid velocities, all independent of $\eps$ in $(0,\kappa)$, for the $N$ filaments. 
 We denote, for $\eps$ in $(0,\kappa)$, the corresponding solutions
 $(h_{i,\eps} , Q_{i,\eps} )_{1 \leq i \leq N}$ to the  Newton-Stokes system  \eqref{N-S} up to the time $T^{\max}_\eps$ as given by Proposition \ref{CL}. 

Then on the one hand
\begin{equation}
\label{temps}
 \liminf_{\eps \to 0} T^{\max}_\eps \geq \hat T, 
\end{equation}
and on the other hand, for any $1 \leq i \leq N$, 
 for all $T < \hat T$ there exists $C$  depending only on $u^\flat$,
 the filaments 
 ${\bar{ \mathcal S}_i^\kappa}$ with thickness $\kappa$, 
 $\inf_{t \in [0,T]} \hat d_{\min}(t)$ (see Definition \ref{collision-limit}) and the initial velocities, and there exists $\eps_0 > 0$ depending in addition on $T$ 
 such that for all $\eps $ in $(0, \eps_0)$,
 \begin{align} \label{est:main.positions}
 \|(h_{\eps, i} , Q_{\eps, i} )    - ({\hat h}_i , {\hat Q}_i)\|_{L^\infty(0,T)}
 \leq C \left(\eps +   |\log \eps|^{-1/2} T\right) e^{C T}.
 \end{align}
The perturbation flow $u^\mathfrak{p}$ due to the filaments, extended by the filament velocity inside each filament, 
satisfies the following estimates: for any compact subset $K$ of $\R^3$ and for any $p $ in $[1, 2)$, for all $t < T$, for all $\eps $ in $(0, \eps_0)$,
\begin{equation}
\label{conv-flui-precised}
	\|    u^\mathfrak{p} (t,\cdot)  - |\log \eps|^{-1} \, {\hat u}^\mathfrak{p} (t,\cdot)  \|_{L^p ( K) }  
	\leq 
	 C |\log \eps|^{-1} \left( 
	 \, e^{-  \frac{C t}{\eps^2 |\log \eps|}} +  |\log \eps|^{-\frac12} e^{C t} \right).
\end{equation}
where ${\hat u}^\mathfrak{p} $ is given by \eqref{hatup} and \eqref{hatS}.
	\end{thm}

\ \par \ 
A few comments on Theorem \ref{main} are in order. 

\ \par \ 
Let us start with saying that Theorem \ref{main} establishes the convergence of the   original system \eqref{N-S} for the filaments to the reduced model \eqref{fav} for their centerline curves, as the thickness parameter $\eps$ converges to zero. Let us highlight that  the system \eqref{N-S} is a coupled system made of the Newton equations associated with all the filaments whereas the limit equation \eqref{fav} for each filament is decoupled from the others. Such a phenomenon   enters the scope 
of the theme of \textit{hydrodynamic decoupling}. Here it states that the main effect on each limit centerline curve is due to the background flow and not from the other filaments.  The limit equations  \eqref{fav} have the advantage in view of  applications to  only involve the geometry of the centerline curves rather than the one of the whole filaments.

\ \par \ 
The system \eqref{N-S} is a second-order  system whereas the limit equations \eqref{fav} are {first-order equations}. Therefore one initial data has to be dropped for the limit system \eqref{fav}. Unless 
the initial data for the system \eqref{N-S} satisfies the compatibility conditions  $( \textrm{v}_i (0) ,\omega_i (0))  =  (\hat{\textrm{v}}_i(0) ,{\hat \omega}_i (0))$, for all $1 \leq i \leq N$, with $(\hat{\textrm{v}}_i(t) ,{\hat \omega}_i (t))$ given by \eqref{fav}, 
the velocities dynamics exhibit {an} initial layer, which prevents  
uniform convergence of the filament velocities down to the initial time.   
Indeed a byproduct of our analysis is that 
we are able to describe the nature of the initial stage: it is an exponential relaxation within a time interval of order 
  $O(\eps^2 \log \eps)${. During that} time, 
 a transition of the amplitudes of the filament velocities occurs which is of order $O(1)$. After this initial stage, the dynamics of the filaments is adapted to the first-order dynamics of the limit system and the convergence occurs at least with a rate $O( \vert \log \eps \vert^{-\frac12})$. 
   More precisely, {for any $T < \hat T$ there is $\eps_0 > 0$ and $C>0$ as in Theorem \ref{main}} such that for any $\eps $ in $(0,\eps_0)$, for any $1 \leq i \leq N$, and for any $t$ in $[0,T]$,
\begin{align}
 \left| ({\textrm{v}}_i ,\omega_{i})(t)  -  ({\hat {\textrm{v}}}_i ,  \hat \omega_i  ) (t) \right| 	& \leq   \left|  (\textrm{v}_i  ,\omega_i ) (0)- 
  (\hat{\textrm{v}}_i,{\hat \omega}_i ) (0)  \right|  \, e^{-  \frac{C t}{\eps^2 |\log \eps|}} + C  |\log \eps|^{-\frac12} e^{Ct}.  \label{allezesti}
\end{align}

\ \par \ 
Our analysis allows us to give an  even better approximation of the solutions to \eqref{N-S} when $\eps$ goes to $0$,
by a family of velocities, indexed by $\eps$, given by a quasi-static balance similar to \eqref{fav}, but with the Stokes resistance matrices and  the   Fax\'en force and torque 
associated with the whole set of filaments rather than their sole centerlines, see  Theorem \ref{refinedqs}.

\ \par \ 
Regarding the  fluid part of the system,  Theorem \ref{main} establishes that 
 after an initial relaxation stage the perturbation of the fluid velocity  ${u}^\mathfrak{p}$ is well-approximated in $L^p_\loc$ by $|\log \eps|^{-1} \hat {u}^\mathfrak{p}$ which is explicitly given (see \eqref{hatup}) in terms of the limit dynamics of the filament centerlines only. The estimate 
 \eqref{conv-flui-precised} should be interpreted in the sense that, firstly, the fluid perturbation is of order $|\log \eps|^{-1}$ in $L^p_\loc$, $p<2$ which corresponds to the Stokes resistance of the filaments. Secondly, the  perturbation, rescaled to order $1$, is well approximated by $ \hat {u}^\mathfrak{p}$ up to an error which corresponds to the sum of the errors of the positions \eqref{est:main.positions} and of the velocities \eqref{allezesti} of the filament centerlines. 
 As we will see, it is possible to improve the estimate to $L^p_\loc$, $p< 6$ on the expense of the rate of convergence. More precisely, for $2 \leq p < 6$ and for all $\delta > 0$,
 \begin{align}
\label{conv-flui-precised.comment}
	&\|    u^\mathfrak{p}(t,\cdot)  - |\log \eps|^{-1} \, {\hat u}^\mathfrak{p} (t,\cdot)  \|_{L^p ( K) }   \\ \nonumber
	&\leq 
	 C |\log \eps|^{-1} \big(     \sum_{1 \leq i \leq N} \left| (\textrm{v}_i ,\omega_i ) (0,\cdot) -  (\hat{\textrm{v}}_i ,{\hat \omega}_i )(0,\cdot) 
	 \right|  \, e^{-  \frac{C t}{\eps^2 |\log \eps|}} 
	+  |\log \eps|^{-\frac 1 2( \frac 3 p - \frac 1 2 - \delta)} e^{C t} \big).
\end{align}
\ \par \ 
 We observe that, at any time, the leading part
${\hat u}^\mathfrak{p}$  of the perturbation flow given by 
 \eqref{hatup} satisfies the modified Stokes equation in the sense of distributions in
the variable $x$, in $\R^3$,
\begin{align*}
	-\Delta {\hat u}^\mathfrak{p}  + \nabla {\hat p}^\mathfrak{p} = \sum_{1 \leq i \leq N}  \hat \mu^\mathfrak{p}_i  \quad \text{ and } \quad   \div {\hat u}^\mathfrak{p}  = 0 ,
\end{align*}
where for $1 \leq i \leq N$,  the term $\mu^\mathfrak{p}_i $ is the vector measure given by 
$$
 \hat \mu^\mathfrak{p}_i  := 
\mu_{\hat{\mathcal C}_i}[   {\hat v}^{{\mathcal S}_i}  - u^\flat   ],
$$
where $ {\hat v}^{{\mathcal S}_i} $ is given by \eqref{hatS}
and
${\hat p}^\mathfrak{p} := P   \ast  \hat \mu^\mathfrak{p}_i  ,$
with $ P $ given by \eqref{noyauP}. 
Moreover,  for any $1 \leq i \leq N$, it follows from the definition of 
the vector fields $v_{i,\alpha}$, for $\alpha =1,2,3$, in \eqref{t1.6},  that the total mass of the measure $\mu^\mathfrak{p}_i $ is
$$
 \int_{\hat{\mathcal C}_i} d \hat \mu^\mathfrak{p}_i   
 = 
 (\hat{\mathcal K}_{i,\alpha,i,\beta} (\hat h_i , \hat Q_i))_{1 \leq  \alpha \leq 3 ; 1 \leq  \beta \leq 6} \, 
  	(\hat{\textrm{v}}_i ,{\hat \omega}_i ) - (\hat{F}^\flat_{i,\alpha})_{1 \leq  \alpha\leq 3} ,
 $$
where we also recall the definitions \eqref{hatK} and \eqref{hatfab}. 
The right hand side above is precisely the leading part of the force due to the fluid on the $i$-th filament, up to the 
renormalization factor $|\log \eps|^{-1}$ and  to the sign, so that its vanishing is precisely the part of \eqref{fav} which concerns the force.  
This is reminiscent of Newton's third law of motion (a.k.a. the action-reaction principle). 

We emphasize that the perturbation flow ${u}^\mathfrak{p}$ is not well approximated by $|\log \eps|^{-1} \, {\hat u}^\mathfrak{p} $ in 
 $\dot H^1$. On the one hand, the perturbation in $\dot H^1$ is actually of order $|\log \eps|^{-1/2}$ instead of $|\log \eps|^{-1}$ (since the Stokes resistance $|\log \eps|^{-1}$ corresponds to the square of the $\dot H^1$-norm).
 On the other hand, the $\dot H^1$-norm turns out to be concentrated in a region of order $\eps$ around the filaments. Since the errors of the positions compared to the limit system is much larger (of order $|\log \eps|^{-1/2}$),  $|\log \eps|^{-1} \, {\hat u}^\mathfrak{p} $ is not a good approximation in $\dot H^1$.
 However, we will show, see Proposition \ref{pro:u^p.U},  that  ${u}^\mathfrak{p}$ is well approximated in $\dot H^1$ by 
 \begin{align}
 |\log \eps|^{-1} \,    \sum_{1 \leq i \leq N} U_{\mathcal C_i} [v^{\mathcal S_i}(t,\cdot) - u^\flat(t,\cdot) ].
 \end{align}
This estimate is actually an important ingredient in the proof of our main result.

\ \par \ 
\subsection{Strategy of the proof of Theorem \ref{main}}  \label{sec:strategy}
Let us give here a glimpse of some elements of the proof of Theorem \ref{main}, whose detailed proof is the purpose of the rest of the paper.
Let us focus first on the way we  deal with the $\eps$-dependence in the filaments dynamics, letting aside for a while the role played by the Stokes system.  
\begin{itemize}
\item A first ingredient is a reformulation of the Newton equations into a second-order ODE for the $6N$ degrees of freedom of the filaments, see \eqref{compressed}. This singularly perturbed ODE looks like the following toy-model:
\begin{equation}
\label{ecosse}
\eps^2 \, q'' = -  |\log \eps|^{-1}  (k^\eps(q) q' - f^\eps(q)) + r^\eps,
\end{equation}
where the scalar
unknown $q$ stands for the variables encoding the positions of the filaments (with a mute dependence on $\eps$), $k^\eps $ are positive $q$-Lipschitz functions uniformly with respect to $q$ and $\eps$, $f^\eps$ are  $q$-Lipschitz functions uniformly with respect to $q$ and $\eps$, and $r_\eps$ is a {remainder}   
with nice estimates. 
\item A second ingredient is a modulated energy argument which consists in estimating the dynamics of 
\begin{equation*}
\frac12 (q' - V^\eps(q))^2 \quad \text{ with } \quad V^\eps(q) := (k^\eps(q))^{-1} \, f^\eps(q).
\end{equation*}
This leads to 
\begin{equation}
\label{irlande}
q' = V^\eps(q)  + \tilde r^\eps,
\end{equation}
with $\tilde r^\eps$ satisfying some relevant estimates. 
\item A third ingredient is to prove that, roughly speaking,
\begin{equation}
\label{galles}
\forall \overline q, \quad V^\eps(\overline q) \rightarrow \hat V(\overline q) \quad \text{as } \quad \eps \rightarrow0 . 
\end{equation}
\item This finally allows to compare $q$ and the solution $\hat q$ of the limit ODE: 
$$\hat q' = \hat V(\hat q).$$
\end{itemize}
Of course this protocol relies on a detailed analysis of the asymptotic behaviour {on the fluid part}, to obtain the behaviour with respect to $\eps$  of the coefficients  in \eqref{ecosse} %, 
and to prove \eqref{galles}. 
This analysis uses  properties of the Stokes system in the presence of several   filaments in the zero-thickness limit,  for which time only plays the role of a parameter through the  positions of the filaments. 
We will therefore devote a separate section to this issue first, {see} Section \ref{sec:singularity.method}. 
This analysis will also allow to obtain the part of Theorem \ref{main} which concerns the  asymptotic behaviour {of the fluid}.

 \begin{remark} \label{rk-modu}
  The idea of using a modulated energy to deal with singular ODEs is rather ubiquitous in nature; let us mention the paper \cite{BG} for a spectacular use in the context of the analysis of 
 the motion of a charged particle in a slowly varying electromagnetic field when the particle mass converges to zero. An important difference with the case of the equation  \eqref{compressed} is that 
  in \cite{BG} the term without derivative is  a gyroscopic term,  rather than a damping term, so that the modulation  provides a center-guide along which the exact solution oscillates. 
  \end{remark}

\ \par \ 
\subsection{Organization of the proof of Theorem \ref{main}} 
\label{sec:structure}

 In Section \ref{sec:singularity.method} we analyze the asymptotic behaviour of the solution of the  steady Stokes system in presence of several thin  filaments with Dirichlet data at the interface between the fluid and the filaments.  A well-known approximation consists in replacing the presence of the slender filaments by appropriate source terms which are measures supported on the  filament centerlines in the  steady Stokes system set in the whole space $\R^3$.  The precise definition of this approximation is given in Section \ref{sec:singularity.method}  together with an error estimate of the difference between this approximation and the exact solution in the natural energy space, see Theorem  \ref{lem:slender.several}. This analysis holds for any given configuration of the filaments as long as there is no intersection of two or more filaments. 

 In Section \ref{shape-derivative},  we bound the shape derivatives of the Dirichlet energy of solutions of the  steady Stokes system in presence of several thin  filaments.  

Section \ref{gathered-section} is devoted to the proof of the part of Theorem \ref{main}  which concerns the asymptotic behaviour of the filament centerlines.

On the other hand the  part of Theorem \ref{main} which concerns the fluid 
asymptotic behaviour is proven in Section \ref{sectiondassaut}.

\subsection{Comparison with  the literature} 
\label{compa-bib}

It is well known,  see for example the classical textbooks\cite{HW,Pozrikidis}, 
that the solution to the steady Stokes system in the exterior of  bodies  can be written in terms of boundary integral operators over the surfaces of the bodies.
The purpose of the slender body theory is to approximate this solution in the case where the bodies are 
  thin filaments by replacing the integral operators over the surfaces by  integral operators over the filament centerlines. 
  This idea dates back to  Hancock \cite{Hancock}, Cox \cite{Cox}, Batchelor  \cite{Batchelor}, Keller and Rubinow  \cite{KR}, Johnson \cite{Johnson} and had a regain of interest with the {numerical} work  by Peskin  \cite{Peskin}; see also the more recent papers  \cite{Ohm-numerics, Tornberg}.
  More precisely in the slender body theory, in the case where one considers the steady Stokes equations in the exterior of a single $\eps$-thick filament 
$ {\bar{ \mathcal S}_i^\eps}$, as defined in 
 \eqref{def:S^eps}, with some boundary data $v$ on $ \partial {\bar{ \mathcal S}_i^\eps}$,  one substitutes to the exact solution $u$ of this exterior problem, the  solution $u_f$ to the steady Stokes equations in the full space $\R^3$ with as source term  Dirac masses along the centerline curve $\bar{\mathcal{C}}_i$ of $ {\bar{ \mathcal S}_i^\eps}$, that is a  measure $ \mu_{f}$ defined by 
\begin{align} \label{def:muf}
<	 \mu_{f} , \phi >	 &:=   \int_{ \bar{\mathcal{C}}_i} f  \cdot \phi   \, \dd \mathcal{H}^1     , \quad \text{ for any }  \phi \in C_c (\R^3 ; \R^3) ,
\end{align}
where the (vector) density  $f$ has to be chosen in a relevant way. Indeed $u_f$  is given by $u_f := S \ast  \mu_{f} $, where the symbol $\ast$ stands for the convolution in $\R^3$ and $S$ is the Stokes kernel defined by \eqref{e.defS}, which satisfies the steady Stokes equations, with zero source,  in the exterior of the  centerline curve $\bar{\mathcal{C}}_i$ and   \textit{a fortiori } in the exterior of 
the $\eps$-thick filament  $ {\bar{ \mathcal S}_i^\eps}$.  Therefore, when comparing $u$ and $u_f$, the key point is that the trace of $u_{f} $ on $ \partial {\bar{ \mathcal S}_i^\eps}$, which is a linear integral operator acting on the density $f$, matches with $v$. 
However it has been shown  in \cite{Goetz}  that this operator is actually not invertible. 
On the other hand, as already observed in \cite{Cox}, the leading order part of the integral operator is completely local, and gives rise to the correspondence $ f(y) = \frac12 k(\tau(y)) v(y)$ as in \eqref{eq:U.formal}. 
To our knowledge, we provide here for the first time rigorous quantitative error estimates for the zero order slender body approximation given by this correspondence. Neglecting higher-order terms has the advantage of an explicit approximation but restricts to errors of order $|\log \eps|$. However, it seems that in the case of non-circular cross-section errors of this order are unavoidable anyway, if one only relies on approximations through force densities on the centerline. On the other hand, in the case of a  filament with circular cross sections, one  may consider refined approximations by adding to Dirac masses along the centerline curve some other higher-order singularities, in particular the so-called doublets which correspond to $\Delta S$. In this case,  invertible regularizations of the integral operator mentioned above 
have been studied in \cite{Mori-Ohm,Mori-Ohm-Accuracy}.
Let us also mention the recent papers \cite{Mori-Ohm-Spirn,Mori-Ohm-free-ends} 
 which  provide rigorous justifications of the slender body theory 
in the case where the density of force on the centerline curve of a single filament with circular cross sections
is prescribed. 

In \cite{Gonzalez}, Gonzalez has tackled the zero-radius limit of the quasi-static motion of  a single massless filament. His result establishes a limit balance similar to our result, however only under an extra  assumption on the asymptotic behaviour of the density of forces acting on the filament.
His conditional result relies on an different approach than ours, that is on the  boundary integral formulation of the Stokes equations. 

\smallskip

The aforementioned papers are mostly concerned with the quasi-static Stokes problem in the exterior of a given filament  and not with the time  evolution of the filament. 
On the other hand, in \cite{Junk-Illner} a rigid body of arbitrary shape is considered, moving in a viscous incompressible flow driven by the unsteady incompressible  Navier-Stokes  equations. The authors provide a formal derivation of the motion  in the limit where the  size of the body converges to $0$ and the mass density is fixed. This asymptotic analysis relies on the assumption that  the fluid is undisturbed by the particle at the main order and that  the rotation of the rigid body is $O(1)$, while it results from the analysis that at the leading order the particle behaves as a passive tracer in the fluid. 
In \cite{VM} the authors have extended the analysis to other inertia regimes.

\smallskip

Readers familiar with  the vortex filament conjecture for Euler Flows may be tempted to draw a comparison with the present work.
This conjecture concerns the 3D incompressible Euler equations in the case where the initial vorticity is concentrated along a smooth curve. 
It is believed, see for instance \cite{BV,Miot}, that the curve evolves in time by binormal curvature flow, to leading order. Therefore 
two huge differences in this problematic, compared to the present setting, are that{: (i)} it concerns a single phase problem, rather than a diphasic system where fluid and rigid bodies are considered, and {(ii)} the dynamics of the curve is way more intricate since it can deform in time, which corresponds to an infinite number of degrees of freedom. 
An important step toward this conjecture has recently been {achieved} by Jerrard and Seis in \cite{JS} where it is shown that under the assumption that the vorticity remains concentrated along a smooth curve when time proceeds, then this curve  approximatively evolves  by binormal curvature flow. 
Despite these important differences, the mathematical analysis shares some common features,  for example in the way to deal with singular line integral. 
In this respect, it is  interesting to compare  Lemma \ref{lem:dirac} with \cite[Section 4.5]{JS}.

\smallskip

Let us also mention  another possible comparison to a setting where the fluid is also assumed to be driven by the incompressible Euler equations: the work   \cite{GS}  where the zero radius limit of the dynamics of several solids in a 2D perfect incompressible fluid is studied. In particular it shares with the present setting the feature to deal with the case where the inertia of some rigid bodies converges to zero in the limit  so that their limit dynamics is a first-order equation rather than a second-order equation. 
 Accordingly the proofs both use  some modulated energy arguments, compare \cite[Section 7]{GS} and Section \ref{section-modu} below. However the forces which drive the limit dynamics are rather different in both settings, on the one hand they are gyroscopic type forces in the case of  \cite{GS}, similarly to the setting evoked in Remark  \ref{rk-modu},  and on the other hand they are viscous drag type forces in the present paper. Another difference is that in   \cite{GS}  the limit dynamics of the particles  are still coupled in the limit and they influence the fluid, as point vortices. On the other hand we deal here with some 3D rigid bodies shrinking to 1D limit rigid bodies instead of 2D rigid bodies shrinking to point particles.

\subsection{A few possible extensions as open problems}
\label{sec-mr-op}

In this subsection, we state a few open problems regarding some extensions of the analysis performed in this paper. 
\begin{op}
We let aside the particular case of rod-like filaments whose centerlines are line segments, which seems to require additional work due to the degeneracy of the limit Stokes resistance matrix $\hat {\mathcal K}$, for which Lemma  \ref{hat-inverse} does not hold true. Indeed, the resistance to rotations around the orientation of the rod like filaments scales like $\eps^2$ rather than $|\log \eps|^{-1}$.
In the case where the cross sections of the filaments are circular, some decoupling of the dynamics occurs and one can substitute an orientation vector $\xi $ in $S^2$ to the orientation matrix $Q_i $ in $SO(3)$ in order to describe the filaments' rotations. In such a case, it seems possible  to adjust our arguments in order to obtain a result similar to Theorem \ref{main}. However, in the case where the cross sections are not circular, the analysis seems more delicate.
\end{op}
\begin{op}
In view of the quantitative convergence result obtained in Theorem \ref{main}, a 
natural issue is to obtain, in the general case as in the case of line segments,  higher-order asymptotic expansions of the dynamics  with respect to $\eps$.
In particular it would be interesting to analyze  the influence of the cross sections on the dynamics. 
As it can be seen from the toy-model \eqref{ecosse}, and from the  compressed form  of the Newton equations given in \eqref{compressed}
where we highlight that the coefficients are related to the fluid state and 
depend on $\eps$,  establishing such asymptotic expansions  in time requires to prove some precise asymptotic description of the fluid state. 
In this direction it would be interesting to investigate if the analysis performed in \cite[Chapter 12.2]{MNP}, which overcomes the difficulties related to the boundary layers associated with non circular cross-sections  in the case of the Laplace equations with a circular centerline could be adapted to the present setting. 
Let us also mention that the influence of small scales in the cross sections can also be encoded by a different choice of the boundary conditions at the interface between the fluid phase and the solid phase. 
In this paper we concentrate on the case of the no-slip condition at the interface, but some other conditions could be considered as well, such as the Navier slip  
conditions, see \cite{HP20} and the references therein. Hence, it would be 
interesting to investigate whether or not the results of Theorem \ref{main} can be adapted to other boundary conditions. 
Moreover, one may wonder how a change of shape of the centerline curve influences  the dynamics, and the convergence of the dynamics, as the thickness parameter $\eps$ goes to zero. Another natural issue to consider is whether the asymptotic description can be extended up to a collision. For a similar  issue in a close setting  let us mention the papers \cite{BNDRDM1,BNDRDM2}. 
\end{op}
\begin{op} \label{op-cloud}
It would be interesting to investigate the case where the number $N$ of   filaments  goes to $+\infty$, while the thickness parameter  $\eps$
and the length $\ell$ of the filaments go to $0$ with $\eps << \ell << 1$, so that at the limit the phase corresponding to the rigid filaments is then a cloud of point particles. A first question is to identify the limit dynamics of these particles. Moreover one may identify a case where the density of these particles is  sufficient to create a collective effect at the main order on the fluid. This would extend the investigations on the Brinkman force for arbitrary shapes done in \cite{FNN,HMS}  from a case where anisotropy corresponds to a finite ratio  to the case of an infinite ratio. 
\end{op}
\subsection{A few more notations}
For $E \subset \R^d$ and $r>0$ we denote 
\begin{align}
    B_r(E) = \{ x \in \R^d : \dist(x,E) < r \}.
\end{align}
We use the convention that in our estimates the constant $C$ might change from line to line and might depend on the background velocity, on the number $N$ of filaments and on the functions specifying the reference filaments, i.e. $\Psi_i$, $\gamma_i$,  $R_i$, $1 \leq i \leq N$. 
We will always specify other dependencies and will make any dependence of $C$ on $\eps$ explicit.

We also point out that the following convention is used throughout the paper: the letter $u$ stands always for the fluid velocity, while the letter $v$ stands for the solid velocities.

There are several smallness requirements on $\eps$ throughout the paper,   typically denoted by  $\eps < \eps_0$. Similarly as for the constant $C$ we will for simplicity 
allow $\eps_0$ to change its value throughout the proofs of our results. Notice that we will usually take $\eps_0$  smaller than $\kappa$, where $\kappa$ is defined in Theorem \ref{main}.

%%%%%%%%%%%%%%%%%%%%%%%%%%%%%%%%%%%%

\section{Immersed boundary method for the steady Stokes system in presence of several thin  filaments}
\label{sec:singularity.method}

This section is devoted to the asymptotic behaviour, in the limit where the thickness $\eps$ of the filaments 
$(\mathcal S_j)_j$ converges to zero, of the solution  $u $ in $\dot H^1(\R^3)$ to the problem
\begin{align} \label{eq:Stokes.data.filaments}
	\begin{aligned}
		- \Delta u + \nabla p &= 0  \quad \text{ and }\quad \div u = 0 \quad \text{in } \mathcal F \\
		u(x) &= v(x)  \quad \text{ in }  \mathcal S_i, \\
		u(x) &= 0  \quad \text{ in } \mathcal S_j, \quad \text{ for} ~ j \neq i ,
	\end{aligned}
\end{align}
where $1 \leq i \leq N$ is given, as well as the data  $v$ which is assumed to satisfy
\begin{align}
\label{ass-v}
	v \in W^{1,\infty}( \mathcal S_i) \, \text{ satisfying }  \int_{ \mathcal S_i} \div v  = 0.
	\end{align}
	Here these filaments are supposed to  be given and fixed in terms of the reference filaments $\bar {\mathcal S}_j$ and some translations and rotations $h_j, Q_j$ as in \eqref{def:Stourne} but without any time dependence. The quantities $h_j, Q_j$ are supposed to be given in such a way that the filaments $(\mathcal S_j)_j$ do not overlap or touch. In fact all the results in this section that concern several filaments will be stated under the assumption that the minimal distance between the filament centerlines 
\begin{align} \label{dmin}
	\dmin := \min_{i \neq j} \dist(\mathcal C_i,\mathcal C_j) ,
\end{align}
is bounded from below and under a smallness condition on $\eps$. Together, this implies a lower bound on the distance between the filaments $\mathcal S_i$.

 To approximate the solution $u$ to \eqref{eq:Stokes.data.filaments}
 we rely on the auxiliary velocity field  $U_{\mathcal C_i}[v]$ given in   \eqref{def:U}. 
 This velocity field  solves the Stokes system in the full space $\R^3$ with an appropriate source term given as {Dirac masses} along the limit curve $\mathcal C_i$. 
It follows from \eqref{eq:U.formal}, from the decay  of {the kernel} $S$ defined by \eqref{e.defS} (and its derivative), and from the boundedness of $k$ from \eqref{ki} that for all $x \in \R^3 \setminus \mathcal C_i$
\begin{align}
	|U_{\mathcal C_i}[v](x)| \leq  C \|v\|_{L^\infty} \min \left\{ \log \left( 1 + \frac 1 {\dist(x,  \mathcal C_i)} \right), \frac 1 {\dist(x, \mathcal C_i)} \right \}, \label{eq:decay.U} \\
	|\nabla U_{\mathcal C_i}[v](x)| \leq  C \|v\|_{L^\infty} \min \left\{  \frac 1 {\dist(x, \mathcal C_i)},  \frac 1 {\left(\dist(x, \mathcal C_i)\right)^2} \right \}. \label{eq:decay.nabla.U} 
\end{align}

The following result establishes that $U_{\mathcal C_i}[v ]$ is the leading part of 	the solution $u$ to \eqref{eq:Stokes.data.filaments}, up to a renormalization factor 
$|\log \eps|^{-1}$ as long as the filaments are sufficiently separated in terms of $\dmin$ given by \eqref{dmin}.
\begin{thm} \label{lem:slender.several}
For all $d> 0$ there exists $\eps_0(d)>0$ and $C(d)>0$, for all filament configuration with $d_{\min} \geq d$ and for all $\eps $ in $(0, \eps_0)$, for any $v$ satisfying \eqref{ass-v} we have the following result. 
	 The solution $u$ to \eqref{eq:Stokes.data.filaments} satisfies 
	\begin{align} \label{est:u_H^1}
		\|u\|_{\dot H^1(\R^3)}  
		\leq C | \log \eps |^{-1/2} \|v\|_{W^{1,\infty} ( \mathcal S_i)}.
	\end{align}
Moreover,
	\begin{align} \label{est:u-U_i}
		\|u -  |\log \eps|^{-1} U_{\mathcal C_i}[v ]\|_{\dot H^1(\R^3 \setminus \mathcal S_i)}  
		\leq C | \log \eps |^{-1} \|v\|_{W^{1,\infty} ( \mathcal S_i)}, \\ \label{est:u-U_i.L^p}
		\|u -  |\log \eps|^{-1} U_{\mathcal C_i}[v ]\|_{W^{1,q}(K)}  
		\leq C | \log \eps |^{-3/2} \|v\|_{W^{1,\infty} ( \mathcal S_i)},
	\end{align}
	for any $q$ in $[ 1, 3/2)$ and any compact $K \subset \R^3$, where $C$ in \eqref{est:u-U_i.L^p} depends in addition on $q$ and $K$.
\end{thm}
To prove Theorem \ref{lem:slender.several}, we will proceed in several steps. 
First, in Subsection \ref{sec:pointwise}, we will establish pointwise estimates of $U_{\mathcal C_i}[v]$. 
Then in Subsection \ref{zeveryproof}, we will deduce uniform estimates in $\dot H^1(\R^3)$ based on Helmholtz' minimum dissipation theorem, see Theorem \ref{Helmholtz}. 
This enables to tackle the very proof of Theorem \ref{lem:slender.several} in Subsection \ref{zeveryproof}. 

Theorem \ref{lem:slender.several}  will be used in Section \ref{sectiondassaut} to prove the  part of Theorem \ref{main}  devoted to the asymptotic behavior of the fluid, once  the  asymptotic behavior of the dynamics of the filaments is obtained.

Theorem \ref{lem:slender.several} is also useful to establish 
 approximation results of the force {exerted by the fluid on the filaments.} 
 To cover the different uses which we will need, we first show a rather general result, where we make use of the elementary rigid velocities   $v_{i,\alpha}$   defined in \eqref{t1.6}.  We associate with these fields, for $1 \leq \alpha\leq 6$ and $1 \leq i \leq N$,   the unique solutions $ V_{i,\alpha}$ in $\dot H^1(\R^3)$  to 
\begin{subequations}
 \label{sti}
\begin{gather}
\label{sti1}
 -\Delta  V_{i,\alpha}  + \nabla  P_{i,\alpha}    = 0    \quad \text{ and } \quad  \div V_{i,\alpha} = 0  ,    \quad  \text{ in }   \mathcal F , 
\\ \label{sti3}
   V_{i,\alpha}  =  \delta_{i,j}  v_{i,\alpha}   , \quad \text{in }  \mathcal S_j  . 
\end{gather}
\end{subequations}
The vector fields $ V_{i,\alpha} $ are smooth, decay as $1/ \vert x \vert$ at infinity, their first-order derivatives and the associated pressures $P_{i,\alpha}$ decay  as $1/ \vert x \vert^2$.

The next result concerns the approximation of force and torque.
\begin{cor} \label{prop:asymptotic.resistance}
	For all $d > 0$ there exists a constant $C = C(d) > 0$ such that for all $\eps $ in $(0, \eps_0(d))$, for all filament configuration with $\dmin \geq d$, for all divergence-free functions $v \in W^{1,\infty}(\cup_{j=1}^N \mathcal S_j)$ and all $1 \leq i \leq N$,  for $1 \leqslant \alpha \leqslant 6 $, 
	\begin{align}
	\begin{aligned}
	   &\left| \int_{\cup_{j=1}^N \partial \mathcal S_j}  (\Sigma (V_{i,\alpha}, P_{i,\alpha}) n) \cdot v \, d\mathcal{H}^2 -  |\log \eps|^{-1} 
	   I_{\mathcal{C}_i }  [v_{i,\alpha}, v]  
	    \right| \\
	    & \quad
\leq C |\log \eps|^{- 3/2} \|v\|_{W^{1,\infty}{(\cup_{j=1}^N \mathcal S_j)}}.
    \end{aligned} \label{eq:Faxen.force}
	\end{align}
\end{cor}
The proof of Corollary \ref{prop:asymptotic.resistance} will be given in Subsection \ref{sec-prop:asymptotic.resistance}.
\ \par \

A first particular useful application of Corollary  \ref{prop:asymptotic.resistance} corresponds to the case where $v=\delta_{i,j}v_{j,\beta}$ for  $1 \leq \beta \leq 6$ and $1 \leq i,j \leq N$. It entails that for $1 \leq \alpha,\beta \leq 6$ and $1 \leq i,j \leq N$, the quantity 
\begin{equation} \label{resis}
 \mathcal K_{i,\alpha,j,\beta} := \int_{  \partial  \mathcal S_j}  (\Sigma ( V_{i,\alpha}, P_{i,\alpha}) n) \cdot v_{j,\beta} \, d\mathcal{H}^2 ,
 \end{equation}
 satisfies 
	\begin{align}
		\left| \mathcal K_{i,\alpha,j,\beta}  -  
		|\log \eps|^{-1}   \hat{\mathcal K}_{i,\alpha,j,\beta} 
		 \right| &\leq C | \log \eps |^{-3/2}. \label{eq:approx.K}
	\end{align}
Recall that the limit Stokes resistance matrices	$ \hat{\mathcal K}_{i,\alpha,j,\beta}$ are defined in \eqref{hatK} and are considered here as  being associated with a fixed position of the centerline curves.

We will denote by  $\mathcal K$ the $6N \times 6N$ matrix  whose coefficients are  these quantities $\mathcal K_{i,\alpha,j,\beta}$,  for $1 \leq \alpha,\beta \leq 6$ and $1 \leq i,j \leq N$.  
Recall that the matrix $\mathcal K$ is referred to as   the  steady Stokes  resistance tensor, that it depends on all the positions $h_i$ and orientations $Q_i$ and is symmetric positive definite, as a consequence of integrations by parts, energy and uniqueness properties of the exterior steady Stokes system. Let us refer for example  to \cite[Chapter $2$]{Kim-Karilla}, \cite[Chapter $5$]{Galdi},  \cite[Chapter $2$ and $3$]{Lady}.

Moreover it follows immediately from \eqref{eq:approx.K} and the coercivity of $\hat{\mathcal K}$ that we observed in \eqref{eq:K.hat.coercive} (recall that $\hat{\mathcal K}$ is block-diagonal) that 
 	\begin{align}
		  \mathcal K &\geq  \frac1C |\log \eps|^{-1} \Id \quad   \text{and }  \quad   | \mathcal K^{-1} | 
	   \leq C | \log \eps | .
		  \label{eq:K.coercive}
		  \end{align}
		\ \par \

Another particular use of  Corollary  \ref{prop:asymptotic.resistance}
 is the case where $v=u^\flat$.
It will provide some estimates on the so-called Fax\'en forces and torques defined by 
\begin{equation}\label{clour}
 \mathfrak{f}^\flat  :=  ( (F^\flat_{i},T^\flat_{i} ))_{1 \leq i \leq N}.
\end{equation}
where
\begin{equation}
\label{fab3}
F^\flat_{i}  := ( F^\flat_{i,\alpha}  )_{\alpha=1,2,3} \quad   \text{ and }  \quad  T^\flat_{i}  := (T^\flat_{i,\alpha} )_{\alpha=1,2,3} ,
\end{equation}
with, for $\alpha=1,2,3$, $1\leq i \leq N$, 
\begin{align}
\label{fab1}
F^\flat_{i,\alpha}  := 
 \int_{  \cup_{j=1}^N \partial  \mathcal   S_j } (\Sigma (V_{i,\alpha}, P_{i,\alpha}) n) \cdot u^\flat \, d\mathcal{H}^2 ,
\\ \label{fab2}
T^\flat_{i,\alpha}  := 
 \int_{  \cup_{j=1}^N \partial  \mathcal   S_j }  ( \Sigma (V_{i, \alpha +3 }, P_{i, \alpha+3}) n) \cdot u^\flat \, d\mathcal{H}^2 .
\end{align}
Indeed 
 	applying 	\eqref{eq:Faxen.force}  to $v=u^\flat$, we arrive at
		\begin{align}  \label{ofint}
	   | \mathfrak{f}^\flat -| \log \eps |^{-1} \hat{\mathfrak f}^\flat| 
	   & \leq C | \log \eps |^{-3/2} \|u^\flat\|_{W^{1,\infty}}.
	\end{align}
Above 
$$ {\hat{\mathfrak{f}}%_i
^\flat }
:= (\hat{\mathfrak{f}}_i^\flat)_{ 1 \leq i \leq N} , $$ 
where we recall that, for $1 \leq i \leq N$, the vector $\hat{\mathfrak{f}}^\flat_{i} $ gathering the limit  Fax\'en forces and torques is defined in \eqref{def:f.hat.flat.frac} and is here considered as {being} associated with a fixed position of the centerline curves.

\subsection{Modified centerlines for the non-closed filaments}
\label{sec:non-closed}

To simplify the proof of Theorem \ref{lem:slender.several},
we introduce slightly modified centerline curves in the case of non-closed filaments, i.e. the case when $\gamma_i$ is not periodic. In this case, we cut an $\eps$ layer at both endpoints. More precisely, we  define 
\begin{align}
\bar {\mathcal C}_i^\eps := \gamma_i([\eps,L_i - \eps]),
\end{align} 
and correspondingly, we write $\mathcal C_i^\eps$ for the curve which is obtained from $\bar {\mathcal C}_i^\eps$ through translation and rotation.
This cut-off version satisfies 
\begin{align} \label{eq:distance.cut-off.centerline}
    \dist(\mathcal C_i^\eps,\partial \mathcal S_i) \geq c \eps
\end{align}
for all $\eps < \eps_0$ and some $c > 0$ independent of $\eps$
We remark that $\mathcal C_i^\eps$ resembles the so-called effective centerline in \cite{Mori-Ohm-free-ends}.
 
In the case of a closed filament, i.e. when $\gamma_i$ is periodic, \eqref{eq:distance.cut-off.centerline} is automatically satisfied for $\mathcal C_i^\eps := \mathcal C_i$.

We show the following lemma, which allows us to prove Theorem \ref{lem:slender.several} by replacing 
$U_{\mathcal C_i}$ in \eqref{est:u-U_i} by
$U_{\mathcal C_i^\eps}$.

\begin{Lemma} \label{lem:C^eps}
    Let $v$ as in \eqref{ass-v}.
Then there exists $C > 0$ such that for all $\eps $ in $(0, \eps_0)$,
\begin{align}\label{est:U-U_i.H^1}
\|U_{\mathcal C_i^\eps}[v ] -   U_{\mathcal C_i}[v ]\|_{\dot H^1(\R^3 \setminus \mathcal S_i)}  
		\leq C \sqrt{\eps} \|v\|_{L^{\infty} ( \mathcal S_i)}.
\end{align}
\end{Lemma}
\begin{proof}
By linearity, $U_{\mathcal C_i}[v ] - U_{\mathcal C_i^\eps}[v ]   = U_{\mathcal C_i \setminus \mathcal C_i^\eps}[v ] $.
Similarly to the pointwise estimate \eqref{eq:decay.nabla.U}, we observe that for all $x\in\R^3\setminus\mathcal C_i$,
\begin{align}
    |\nabla U_{\mathcal C_i \setminus \mathcal C_i^\eps}[v ](x)| \leq  \|v\|_{L^{\infty} ( \mathcal S_i)} \min \left\{  \frac 1 {\dist(x,\mathcal C_i \setminus \mathcal C_i^\eps)},  \frac \eps {\left(\dist(x, \mathcal C_i \setminus \mathcal C_i^\eps)\right)^2} \right \}. \label{est:U.C_i.C_i^eps}  .
\end{align}
Denote by $\partial \mathcal C_i$ the two endpoints of the curve $\mathcal C_i$. Then, using that for all $x \in \R^3 \setminus \mathcal S_i$
\begin{align}
{\dist(x,\mathcal C_i \setminus \mathcal C_i^\eps)} \geq c  \dist(x,\partial \mathcal C_i)
\end{align}
estimate \eqref{est:U.C_i.C_i^eps} yields \eqref{est:U-U_i.H^1}.
\end{proof}

\subsection{Pointwise estimates}
\label{sec:pointwise}

This section is devoted to the analysis of the behavior of $U_{\mathcal C_i^\eps}[v]$ on the boundary of the filament $\mathcal S_i$.

We introduce  $\xi_i$ as the orthogonal projection from $\partial   \mathcal S_i$ to $\mathcal C_i$ which is well-defined for $\eps$ sufficiently small since, by assumption,  $\gamma_i$ has no self-intersections. 
 Moreover, we denote by $\partial \mathcal C_i^\eps$ the boundary of the $1$-dimensional manifold $\mathcal C_i^\eps \subset \R^3$, which is empty if $\gamma_i$ is a closed curve (as we defined $\mathcal C_i^\eps = \mathcal C_i$ in this case)  and contains precisely $2$ points otherwise.

\begin{Lemma} \label{lem:dirac}
Let $v$ as in \eqref{ass-v}.
Then there exists $C > 0$ such that for all $\eps $ in $(0, \eps_0)$,
\begin{gather} \label{eq:U_i.boundary}
	|U_{\mathcal C_i}[v] -  |\log \eps |\,  v \circ \xi_i |(x)
	 \leq C \|v\|_{W^{1,\infty} {(\mathcal S_i)}} \left(1 + |\log(\dist(x,\partial \mathcal C_i^\eps))| \right),
\end{gather}
on  $\partial \mathcal S_i$. 
Moreover, 
\begin{align} \label{eq:norm.U_i}
	\|U_{\mathcal C_i}[v]\|_{\dot H^1(\R^3 \setminus \mathcal S_i)} \leq  C |\log \eps|^{\frac 1 2} \|v\|_{L^\infty(\mathcal S_i)} ,
		\end{align}
and
\begin{align}
\label{eq:norm.U_i.subcritical}
	\|U_{\mathcal C_i}[v]\|_{W^{1,p}(K)} \leq C_{K,p} \|v\|_{L^\infty{(\mathcal S_i)}},
\end{align}
for all $p $ in $[1,2)$ and all compact $K \subset \R^3$.
Furthermore, for all $d>0$, there exists  $C$ depending on $d$ such that 
	\begin{equation}  \label{eq:norm.U_i-bis}
	\,	\|U_{\mathcal C_i}[v]\|_{{W^{1,\infty}}( \R^3 \setminus B_d(\mathcal C_i))} \leq  C  \|v\|_{L^\infty{(\mathcal S_i)}} \, \text{ for} ~ j \neq i .
	\end{equation}
\end{Lemma}
\begin{proof}
 The estimates in \eqref{eq:norm.U_i}, \eqref{eq:norm.U_i.subcritical} {and \eqref{eq:norm.U_i-bis}} are  direct consequences of the pointwise estimates \eqref{eq:decay.U} {and} \eqref{eq:decay.nabla.U}.

	It remains to prove \eqref{eq:U_i.boundary}.
	We will assume that the filament is non-closed. The case of a closed filament is slightly easier.
     {Since the statement concerns only a single filament, we might  assume that $\mathcal C_i = \bar {\mathcal C_i}$ , and in particular $\mathcal C_i^\eps = \gamma_i([\eps,L_i - \eps])$.
Then,  we introduce the function $\bar s$  from $\partial   \mathcal S_i$ to $[0,L_i]$ 
 by the formula  $\gamma_i (\bar s (x) ) = \xi_i(x)$ for any $x$ in $\partial   \mathcal S_i$.}
 	In the rest of the proof, we will drop the index $i$.
 	
 	new{Fix $x \in \partial \mathcal S$. Note that $\bar s (x) \in [0,\eps] \cup [L - \eps , L]$ implies
 	$\dist(x, \partial \mathcal C^\eps) \leq  C \eps$. In this case,   \eqref{eq:U_i.boundary} follows immediately \eqref{eq:decay.U} applied to $\mathcal C^\eps$ and \eqref{eq:distance.cut-off.centerline}.
 	Therefore, we might assume in the following that $\bar s (x) \in (\eps,L - \eps)$.}
 
	Let us denote the straight line
	approximation at $\bar s = \bar s(x)$ by 
	$$\gamma_0(s) := \gamma(\bar s) + (s - \bar s) \gamma'(\bar s).$$ 
		We observe that for all $s \in [\eps,L - \eps]$ we have 
	\begin{gather}
	\label{ob1}
	|\gamma(s) - \gamma_0(s)| \leq C |s - \bar s|^2, \quad
	 |\gamma'(s) - \gamma'_0(s)| \leq C |s - \bar s|, \quad \\
	 	\label{ob2}
	|x-\gamma(s)| \geq c |s - \bar s| \quad \text{ and } \quad |x-\gamma_0(s)| \geq |s - \bar s|.
	\end{gather}

	Then, we split the integral
	\begin{align} 
\nonumber	U_{\mathcal C_i}[v](x) &= \frac12 \int_\eps^{L-\eps} \Big( S(x - \gamma(s)) k(\gamma'(s)) v(\gamma (s))  \dd s
	\label{indec}
		\\  &= \frac12 \big( U_1 (x) + U_2 (x)\big) ,
			\end{align}
		 where 
	\begin{align}
	U_1 (x)  &:= \int_\eps^{L-\eps} S(x - \gamma_0(s)) k(\gamma'(\bar s)) v(\gamma (\bar s)) \dd s, \\
 U_2 (x)&:= \int_\eps^{L-\eps} \Big( S(x - \gamma(s)) k(\gamma'(s)) v(\gamma (s)) - S(x - \gamma_0(s)) k(\gamma'(\bar s)) v(\gamma (\bar s))\Big) \dd s.
	\end{align}
	We decompose $U_1 (x)$ further by observing the following. By definition of $S$, 
	\begin{align}
		S(x - \gamma_0(s)) &= 
			{\frac1{8\pi}}\frac 1 {|x - \gamma_0(s)|} \Big(
	\Id +A(x - \gamma_0(s))\Big), 
	\end{align}
		where, for $p$ in $\R^3 \setminus \{0\}$, 
			$$A(p) := \frac p {|p|} \otimes \frac p {|p|}.$$
			Moreover, 
	\begin{align}
	\Id	 = {8\pi}	S_0	(\gamma'(\bar s)) -  \gamma'(\bar s) \otimes \gamma'(\bar s) = 	{8\pi}S_0	(\gamma'(\bar s)) - A((\bar s - s) \gamma'(\bar s)), 
	\end{align}
		where $S_0$ is defined in \eqref{notsu}, 
		so that
	\begin{align}
		S(x - \gamma_0(s)) =	{\frac1{8\pi}}\frac 1 {|x - \gamma_0(s)|} \Big(
	{8\pi}S_0	(\gamma'(\bar s))  +A(x - \gamma_0(s)) - A((\bar s - s) \gamma'(\bar s))\Big). 
	\end{align}
	This leads to 
	\begin{align}
	\label{redec}
	U_1 (x) = U_{1,a} (x) + U_{1,b} (x), 
		\end{align}	
	with
	\begin{gather*}
		U_{1,a} (x) := \Big(\int_\eps^{L-\eps} \frac 1 {|x - \gamma_0(s)|}  \dd s   \Big) S_0	(\gamma'(\bar s)) k(\gamma'(\bar s)) v(\gamma (\bar s)),\\
		U_{1,b} (x)  := 
%\\ \quad 	 \quad 
{\frac1{8\pi}}\Big(\int_\eps^{L-\eps}  \frac {A(x - \gamma_0(s)) - A((\bar s - s) \gamma'(\bar s)} {|x - \gamma_0(s)|} 
	 \dd s \Big)	k(\gamma'(\bar s)) v(\gamma (\bar s)).
	\end{gather*}	
	Then, thanks to the identity \eqref{eqqe}, to the fact that  $\gamma (\bar s)$ is 
  the orthogonal projection of $x$ on 
  $\mathcal C_i$, to the change of variables $z=(s-\bar s)\gamma'(\bar s)$ and to the fact that $\gamma$ is a parametrization by arc length, we obtain 
	\begin{align}
	\label{here}
		U_{1,a} (x) &=
		 \Big(\int_\eps^{L-\eps} \frac 1 {|x - \gamma_0(s)|}  \dd s   \Big) v(\gamma (\bar s))
		\\ \nonumber
		&= \Big(\int_{- \bar s + \eps}^{L - \bar s - \eps} \frac{1}{( |x - \gamma( \bar s)|^2 + z^2)^{1/2}} \dd z  \Big) v(\gamma (\bar s))
		\\  \nonumber
		 &= (\sinh^{-1}( |x - \gamma( \bar s)|^{-1} (L- \bar s - \eps)) - \sinh^{-1}(- |x - \gamma( \bar s - \eps)|^{-1} \bar s)) v(\gamma (\bar s)).
\end{align}	 
Using that $\sinh^{-1}(z) = \log(z + \sqrt{1 + z^2})$, for any real $z$,  and that $c \eps \leq |x - \gamma( \bar s)| \leq C \eps$,  elementary but tedious estimates show
that
\begin{align}
\label{ledeux}
	\big|U_{1,a} (x) -  2 |\log \eps | v(\gamma (\bar s))\big| \leq C |v(\gamma (\bar s))| { \left(1 +  |\log(\dist(x,\partial \mathcal C^\eps))|  \right)},
\end{align}
{where we used that $\partial \mathcal C^\eps = \{\gamma(\eps),\gamma(L - \eps)\}$.}

To estimate $U_{1,b} (x)$, we use
\begin{align}
	|A(p) - A(q)| \leq C \min\left\{ 1, |p-q| \max \left\{\frac 1 {|p|}, \frac 1 {|q|} \right\} \right\}.
\end{align} 
Thus, $	|U_{1,b} (x)| $ is bounded by 

\begin{gather*} 
 C {|v(\gamma (\bar s))|} \int_\eps^{L-\eps} \frac 1 {|x - \gamma_0(s)|}  \min\left\{ 1, |x - \gamma(\bar s)| \max \left\{\frac 1 {|x - \gamma_0(s)|}, \frac 1 {|s-\bar s|} \right\} \right\} \dd s ,
\end{gather*}
and therefore by 
\begin{gather*} 
	  C{|v(\gamma (\bar s))|} \Big( \int_{\bar s - \eps}^{\bar s + \eps} \frac 1 {(\eps^2 + (s - \bar s)^2)^{1/2}}  \dd s +   \int_{[\eps,L  -\eps] \setminus [\bar s - \eps, \bar s + \eps]} \frac \eps {(s - \bar s)^2} \dd s \Big),
	  \end{gather*}
	  so that finally 
	\begin{gather} 
	\label{betec}
|U_{1,b} (x)|	\leq C {|v(\gamma (\bar s))|} .
\end{gather}

	We now estimate $U_2 (x)$.
	Using that 
	\begin{align}
	|S(x_1) - S(x_2)| \leq C |x_1 - x_2| \max \left\{\frac 1 {|x_1|^2}, \frac 1 {|x_2|^2} \right\} ,
\end{align} 
we deduce that 
	\begin{align*}
		|S(x - \gamma(s)) k(\gamma'(s)) v(\gamma (s)) - S(x - \gamma_0(s)) k(\gamma'(\bar s)) v(\gamma (\bar s))|
		 \leq C{\|v\|_{W^{1,\infty} (\mathcal S_i)}}.
	\end{align*}
	Therefore, 
	\begin{align}
	\label{plubetec}
	|U_2 (x) |\leq C{\|v\|_{W^{1,\infty} (\mathcal S_i)}}.
	\end{align}
	Combining \eqref{indec}, \eqref{redec}, \eqref{ledeux}, \eqref{betec} and \eqref{plubetec}, we arrive at \eqref{eq:U_i.boundary}. 
\end{proof}

\subsection{Bogovski\u{i} and extension operators  with some uniformity with respect to the domain} \label{sec:H^1.estimates}
 
 A second ingredient of the proof of Corollary  \ref{prop:asymptotic.resistance} is the use of some 
	Bogovski\u{i} operators with some natural norms that are bounded uniformly with respect to the domain. 
	
   We will make use of  a statement regarding Bogovski\u{i} operators associated with John domains. Roughly speaking, an open bounded domain $\Omega$  is 
	a John domain with respect to a point $x_0$  if each point $y$ in $\Omega$ can be reached by a Lipschitz curve beginning at $x_0$ and contained in 
	$\Omega$ in such a way that, for every point $x$ in the curve, the distance from $x$ to $y$ is proportional to the distance from $x$ to the boundary of $\Omega$.
   This class strictly contains the Lipschitz domains. Notice nevertheless that  
   external cusps are not allowed.
   Let us {now} give the precise 
  definition of John domains following the definition of \cite{AcostaDuranMuschietti06}.
    
    \begin{Definition}
        Let $\Omega \subset \R^n$ be an open bounded domain. Then, $\Omega$ is called a John domain with constant $Z > 0$ if there exists $x_0 \in \Omega$ 
        such that for all $x \in \Omega$ there is an $Z$-Lipschitz map $\rho \colon [0,|x-x_0|] \to \Omega$ such that $\rho(0) = x$, $\rho(|x-x_0|) = x_0$ and  for all $t \in [0,|x-x_0|]$, 
        \begin{equation}
            \label{dist-bd}
            \dist(\rho(t),\partial \Omega) \geq t/Z.
        \end{equation}
    \end{Definition}
    
  The filaments $\mathcal S_i$ are Lipschitz domains and therefore any smooth neighborhood of $\mathcal S_i$ is  a John domain. In the next Lemma, we prove that suitable neighborhoods are John domains uniformly in $\eps$.

    \begin{Lemma} \label{nioul}
        Let $d>0$. Then, there exists $\eps_0 >0$ such that for all $\eps < \eps_0$, $B_d(\mathcal C_i) \setminus \mathcal S_i$ is a John domain with a constant $Z$ independent of $\eps$.
    \end{Lemma}
    \begin{proof}
       We first consider the case of a closed filaments. The necessary adaptations for non-closed filaments will be discussed at the end of the proof.
    
        Fix $c < d$ such that the  projection $\xi_i \colon B_{c}(\mathcal C_i) \to \mathcal C_i$  as in Lemma \ref{lem:dirac} is well defined. We will consider $\eps_0 $ in $(0, c)$.
    
        Choose any reference point $x_0 \in B_d(\mathcal C_i) \setminus  B_{c}(\mathcal C_i)$.
        To construct a curve from $x  \in B_d(\mathcal C_i) \setminus \mathcal S_i$ to $x_0$, we proceed in three steps{: we construct} three curves $\rho_1$, $\rho_2$ and $\rho_3$  that {once pasted} together connect $x$ to $x_0$. We parametrize the curves by arclength and reparametrize at the end to obtain a curve $\rho \colon [0,|y-x_0|] \to B_d(\mathcal C_i) \setminus \mathcal S_i$ connecting $x$ to $x_0$. In the following construction, the constants $\eta, C>0$ will be independent of $\eps$. The construction is visualized in Figure \ref{fig:John}.
    
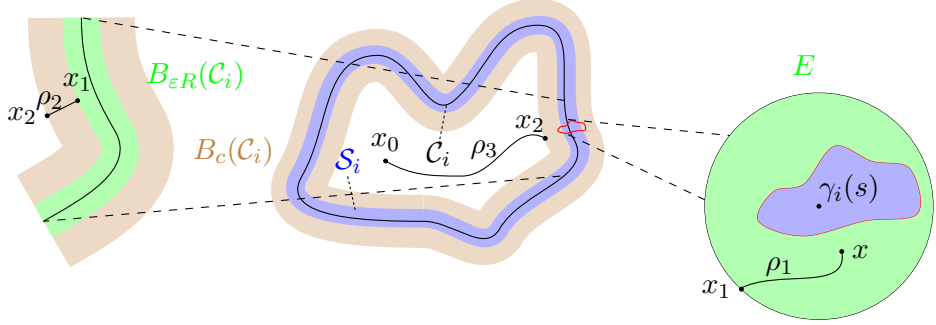
\begin{figure}

      \centering
        
\begin{tikzpicture}

\draw[line width=20pt, brown!30] (0,0.4) to[out=0, in=240] (1, 0.3) 
to[out=60, in=300] (1.97, 1.5)
to[out=120, in=0] (1.5, 3)
to[out=180, in=40] (0.4, 2)
to[out=220, in=0] (-0.6, 2.6) 
to[out=180, in=60] (-1.6, 1) 
to[out=-120, in=180] (0, 0.4) ;

\draw[line width=7pt, blue!30] (0,0.44) to[out=0, in=240] (1, 0.3) 
to[out=60, in=300] (2, 1.52)
to[out=120, in=0] (1.53, 3)
to[out=180, in=40] (0.44, 2.03)
to[out=220, in=0] (-0.63, 2.61) 
to[out=180, in=60] (-1.63, 1.02) 
to[out=-120, in=180] (0, 0.44) ;

\draw (0,0.4) to[out=0, in=240] (1, 0.3) 
to[out=60, in=300] (1.97, 1.5)
to[out=120, in=0] (1.5, 3)
to[out=180, in=40] (0.4, 2)
to[out=220, in=0] (-0.6, 2.6) 
to[out=180, in=60] (-1.6, 1) 
to[out=-120, in=180] (0, 0.4) ;

\begin{scope}[scale=2]
\draw[shift={(-4,-0.7)}, line width=40pt, brown!30] (1.5, 0.9) 
to[out=30, in=300] (1.97, 1.5)
to[out=120, in=-90] (1.75, 2.25)
;

\draw[shift={(-4,-0.7)},line width=14pt, green!30] (1.5, 0.9) 
to[out=30, in=300] (1.97, 1.5)
to[out=120, in=-90] (1.75, 2.25)
;

\draw[shift={(-4,-0.7)}] (1.5, 0.9) 
to[out=30, in=300] (1.97, 1.5)
to[out=120, in=-90] (1.75, 2.25)
;
\end{scope}

\begin{scope}[scale=0.17]
\draw[shift={(6,9)}, red] plot [smooth cycle, tension=0.6] coordinates {(4.4,0.4) (5,0.22) (5.8,0.4) (6.5,0.48)(6.4,0.97)  (5.9,1.07) (5.4,1.2) (5,0.8) (4.6,0.72) };
\end{scope}

\draw[dashed] (1.9,1.55)--(4.5,0.3);
\draw[dashed] (1.9,1.75)--(5.4,1.4);

\draw[dashed] (1.8,1)--(-5,0.4);
\draw[dashed] (1.86,2)--(-4.5,3.1);

\draw[dash pattern={on 1pt off 1pt}] (-0.9,0.55)--(-1,1);
\draw[dash pattern={on 1pt off 1pt}] (0.2,1.42)--(0.3,1.94);

\draw[red, name path=A, fill=blue!30] plot [smooth cycle, tension=0.6] coordinates {(4.4,0.4) (5,0.2) (5.8,0.4) (6.5,0.5)(6.4,1.1)  (5.9,1.2) (5.4,1.4) (5,0.9) (4.6,0.8) };
\draw[fill] (5.2,0.6) circle [radius=0.025];

\draw[name path=B] (5.2,0.6) circle [radius=1.5];
\tikzfillbetween[of=A and B]{green!30};

\node at (5.6,0.5) [anchor=south] {{$\gamma_i(s)$}};
\node at (0.2,1) [anchor=south] {{$\mathcal C_i$}};
\node at (-1.3,1.2) [anchor=west, blue] {{$\mathcal S_i$}};
\node at (-2.5,1) [anchor=south, brown] {{$B_c(\mathcal C_i)$}};
\node at (-3,2) [anchor=south, green] {{$B_{\varepsilon R}(\mathcal C_i)$}};
\node at (5,2.2) [anchor=south, green] {{$E$}};

\draw[fill] (5.5,0) circle [radius=0.025];
\node at (5.5,0) [anchor=west] {{$x$}};
\draw (5.5,0) to[out=-80, in=30] (4.15,-0.5);
\draw[fill] (4.18,-0.5) circle [radius=0.025];
\node at (4.18,-0.5) [anchor=east] {{$x_1$}};
\node at (5,-0.2) [anchor=east] {{$\rho_1$}};

\draw[fill] (-4.55,2) circle [radius=0.025];
\node at (-4.55,2) [anchor=south] {{$x_1$}};
\draw (-4.55,2) to (-4.95,1.8);
\draw[fill] (-4.95,1.8) circle [radius=0.025];
\node at (-4.95,1.8) [anchor=east] {{$x_2$}};
\node at (-4.6,2) [anchor=east] {{$\rho_2$}};

\draw[fill] (1.6,1.5) circle [radius=0.025];
\node at (1.4,1.45) [anchor=south] {{$x_2$}};
\draw (1.6,1.5) to[out=150,in=0 ] (0.5,1) to[out=180, in=-30] (-0.5,1.2);
\draw[fill] (-0.5,1.2) circle [radius=0.025];
\node at (-0.5,1.2) [anchor=south] {{$x_0$}};
\node at (0.8,1.1) [anchor=south] {{$\rho_3$}};

;\end{tikzpicture}

\caption{Illustration of the construction of the curve $\rho$ from $x$ to $x_0$ via $x_1$ and $x_2$} \label{fig:John} 

\end{figure}  

         Let $R := 2 \|\Psi_i\|_\infty$, where we recall from \eqref{def:S^eps} that $\Psi_i$ specifies the cross section of the reference filament. In particular, we have $\mathcal S_i \subset \overline{B_{\eps R/2}(\mathcal C_i)}$.
        
        First, if $x \in B_{\eps R}(\mathcal C_i)$, we construct a Lipschitz curve $\rho_1 \colon [0, T_1]$ from $x$ to $x_1 \in  \partial B_{\eps R}(\mathcal C_i)$ which satisfies $T_1 \leq C \eps R$ and $\dist(\rho_1(t),\partial \mathcal S_i) \geq \eta t$.
       If $x \not \in B_{\eps R}(\mathcal C_i)$, we set $x_1 = x$.         To construct $\rho_1$, we consider $\xi_i(x)$ the projection of $x$ on $\mathcal C_i$ and $\bar s$ such that $\gamma_i (\bar s) = \xi_i (x)$. Let $A$ be the plane through $\xi_i(x)$ perpendicular to $\gamma_i'(\bar s)$. Note that $x \in A$.
        Then, for $\eps$ sufficiently small, $E = A \cap (B_{\eps R}(\mathcal C_i) \setminus \mathcal S_i)$ is a smooth two dimensional domain, and $\hat E := \frac 1 \eps E$ is independent of $\eps$ and depends smoothly on $\xi_i(x)$. Therefore, we may construct $\hat \rho_1$ in $\hat E$ in order to obtain an appropriate curve $\rho_1$ by rescaling.

        Second, we construct a Lipschitz curve $\rho_2 \colon [0,T_2] \to  B_d(\mathcal 
        C_i) \setminus \mathcal S_i$ from $x_1$ to some $x_2 \in B_d(\mathcal C_i) \setminus B_c(\mathcal C_i)$ 
        such that $T_2 = (c - (\eps R + \dist(x_1, B_{\eps R}(\mathcal C_i)))_+$ and
        \begin{align} \label{eq:rho_2}
             \dist(\rho_2(t), \mathcal S_i) \geq    \dist(\rho_2(t), B_{\eps R}(\mathcal C_i)) + \frac 1 2 \eps R = t + \dist(x_1, B_{\eps R}(\mathcal C_i)) + \frac 1 2 \eps R.
        \end{align}
         To construct $\rho_2$, we just move along the gradient of $\dist(\cdot,\partial B_{\eps R}(\mathcal C_i))$. The gradient 
         of $\dist(\cdot,\partial B_{\eps R}(\mathcal C_i))$ coincides with the gradient of $\dist(\cdot,\mathcal C_i)$ outside of $B_{\eps R}(\mathcal C_i)$. The gradient is well defined through our choice of $c$, and \eqref{eq:rho_2} holds.

        Third, we construct a Lipschitz curve $\rho_3  \colon [0,T_3] \to  B_d(\mathcal C_i) \setminus B_c(\mathcal C_i)$ from $x_2$ to $x_0$ 
        such that $|T_3|\leq C|x_2 - x_0|$. 
        The existence of $\rho_3$ is straightforward since  $B_d(\mathcal C_i) \setminus B_c(\mathcal C_i)$ is a Lipschitz domain independent of $\eps$ which contains $x_0$.
        In particular observe that the part of the condition \eqref{dist-bd} which concerns the distance to the external boundary $\partial B_d(\mathcal C_i) $ of the domain 
        $B_d(\mathcal C_i) \setminus \mathcal S_i$ is clear.

      \medskip
        
        Now, we glue the three curves together and rescale to obtain $\rho \colon [0,|x-x_0|] \to B_d(\mathcal C_i) \setminus \mathcal S_i$ which has a Lipschitz constant
        \begin{align}
            Z_0 = \frac{T}{|x-x_0|} =: \frac{ T_1 + T_2 +T_3}{|x-x_0|}.
        \end{align}
        Moreover, $\rho$ satisfies
    \begin{align} \label{eq:rho}
            \dist(\rho(t),\partial \mathcal S_i) \geq \begin{cases}
            \frac{\eta t |x-x_0|}{T} &\quad  0 \leq t \leq \frac{T_1 |x-x_0|}{T} \\
            c\left(t - \frac{T_1 |x-x_0|}{T}\right) + \frac 1 2 \eps R
            &\quad  \frac{T_1 |x-x_0|}{T} \leq t \leq \frac{(T_1 + T_2) |x-x_0|}{T} \\
            c &\quad  \frac{(T_1 + T_2) |x-x_0|}{T} \leq t \leq  |x-x_0|.
            \end{cases}
        \end{align}
       Note that the additional constant $c$ in the second line above arises because the distance to $\partial \mathcal S_i$ is considered instead of the distance to $\mathcal C_i$.
        We claim that
        \begin{align} \label{eq:claim.T}
            T \leq C |x- x_0|.
        \end{align}
        which implies that $Z_0$ is bounded independently of $\eps$ and $x$.
        
        To prove the claim note that
        \begin{align}
            T \leq  C \eps R   \mathbf{1}_{x \in B_{\eps R}(\mathcal C_i)} +  (c - (\eps R + \dist(x_1, B_{\eps R}(\mathcal C_i)))_+ +  C|x_2 - x_0|,
        \end{align}
        where the notation $  \mathbf{1}$ is used for the indicator function of the set written as an index. 
        Consider first the case $x \in B_{\eps R}(\mathcal C_i)$. Then, for $\eps$ sufficiently small, $\eps R \leq c \leq 2 |x - x_0|$. Moreover, $x_1 \in  \partial B_{\eps R}(\mathcal C_i)$. 
        Thus
        \begin{align}
        (c - (\eps R + \dist(x_1, B_{\eps R}(\mathcal C_i)))_+ = |x_1 - x_2| \leq c \leq 2 |x - x_0|,
        \end{align}
        and finally
        \begin{align}
            |x_2 - x_0| \leq |x-x_0| + |x - x_1| - |x_1 - x_2| \leq |x- x_0| + 2 c  \leq C |x-x_0|
        \end{align}
        such that we conclude \eqref{eq:claim.T}.

        If $x \not \in  B_{\eps R}(\mathcal C_i)$ but $x \in B_c(\mathcal C_i)$, then we use that $x_2$ is the orthogonal projection of $x=x_1$ to $\partial B_c(\mathcal C_i)$. Thus
        \begin{align}
        T_2 = |x - x_2| = \dist(x,\partial B_c(\mathcal C_i)) \leq |x - x_0|
        \end{align}
        and we deduce again \eqref{eq:claim.T} by the triangle inequality.
        In the case  $x \not \in B_c(\mathcal C_i)$, the claim is also trivially satisfied since $T_1 = T_2 = 0$.

        It remains to verify that $ \dist(\rho(t),\partial \mathcal S_i) \geq t/Z$, for some $Z$ independent of $x$ and $\eps$.
        By \eqref{eq:rho}, we see that this is satisfied on the first and on the third part of the curve. Recalling $|T_1| \leq C \eps R$, the same holds for the second part of the curve. This finishes the proof.
        
        In the case of a non-closed filament, the proof works almost the same. The only necessary change is due to the fact that the plane $A$ as defined in the construction of $\rho_1$ above does not always contain $x$. Indeed $x \not \in A$ if the segment $[x,\xi(x)]$ is not perpendicular to $\gamma_i (\bar s)$ which can only happen if $\xi (x) \in \partial \mathcal C_i$. Fix such an $x \in B_c(\mathcal C_i)$ and let $P(x)$ be the projection of $x$ to $\partial \mathcal S_i$. Then, since the faces of $\partial S_i$, i.e. the surfaces $\{y \in \partial S_i : \xi(y) \in \partial \mathcal C_i\}$, are flat, the straight curve 
        \begin{align}
            \rho_{1,2}(t):= x + t \frac{x - P(x)}{|x - P(x)|}
        \text{  satisfies } 
            \dist(\rho_{1,2}(t), \partial \mathcal S_i) \geq t ,
        \end{align}
    for all $t \leq T_2$ which we again take to be the time where $\rho_{1,2}(t) \in \partial B_c(\mathcal C_i)$. From there, we can continue with the curve $\rho_3$ as above.
    \end{proof}

Now the statement that we will use is the following particular case of {\cite[Theorem 4.1]{AcostaDuranMuschietti06}}, where
$L^2_0$ denotes the space of $L^2$-functions with vanishing mean.
	\begin{thm}[{\cite[Theorem 4.1]{AcostaDuranMuschietti06}}] \label{John}
	 Let $\Omega \subset \R^3$ be a John domain with constant $Z$. Then there exists a bounded linear operator 
	 $\Bog \colon L^2_0(\Omega)  \to H^{1}_0(\Omega)$
	 such that for all $f \in L^2_0(\Omega)$
	 \begin{align}
	 	\div \Bog f = f,
	\end{align}	 
and the operator norm
	    $\|\Bog\|_{L^2_0(\Omega)  \to H^{1}_0(\Omega)}$ depends only on $Z$ and $\diam(\Omega)$.
	\end{thm}

	Relying on this result, we prove the following lemmas about extending functions defined on $\partial \mathcal S_j$.
	\begin{Lemma}
	\label{Claim} Let $d > 0$. Then, there exists $\eps_0(d) >0$ and $C(d)>0$ such that for all $\eps $ in $(0, \eps_0)$ the following holds. Let $\chi$ in $H^1({\partial} \mathcal S_j)$ satisfying 
	$$\int_{\partial S_j} \chi \cdot n \dd \mathcal{H}^2 = 0.$$
	Then, there exists a divergence-free function $\psi \in \dot H^1(\R^3 \setminus \mathcal S_j)$ such that $ \psi = \chi$ on $ \partial \mathcal S_j$,  $\supp \psi \subset B_d(\mathcal C_j)$ and 
	\begin{align}
		\|\psi\|^2_{\dot H^1 {(\R^3\setminus\mathcal S_j)}} \leq \frac C \eps \| \chi\|^2_{L^2(\partial \mathcal S_j)} + C \eps \|\nabla \chi\|^2_{L^2(\partial \mathcal S_j)}.
	\end{align}
	\end{Lemma}
	\begin{proof}
We consider first the case of a closed filament $\mathcal S_j$. The necessary adaptations for a non-closed filament will be discussed at the end of the proof.
 Let $c>0$ and $T_\eps = B_{\eps c}(\mathcal S_j) \setminus \mathcal S_j$.
 		Denote by $P_\eps$ the projection from $T_\eps$ to $\partial\mathcal S_j$. Then, there exists $(c,\eps_0)$ (depending only on $\mathcal C_j$) such that for all $\eps < \eps_0$, $P_\eps$ is well-defined, smooth and $|\nabla P_\eps| \leq C$. By further reducing $\eps_0$ (depending on $d$), we ensure that $T_\eps \subset B_d(\mathcal C_j)$.
	 Let $\theta_\eps$ be a smooth cutoff function supported in $T_\eps$ such that $\theta_\eps = 1$ on $\partial \mathcal S_j$ and
	 $|\nabla \theta_\eps| \leq \frac C \eps$. 
	 
	Consider the function
	 \begin{align}
	 	\psi(x) = \theta_\eps(x) \chi(P_\eps (x)) - \Bog(\div(\theta_\eps(\cdot) \chi(P_\eps(\cdot)))(x),
	 \end{align}
	 where $\Bog$ denotes  suitable Bogovski\u{i} operators on $B_d(\mathcal C_j) \setminus \mathcal S_{j}$, provided by Theorem \ref{John}.
%	%
	Due to Lemma \ref{nioul} such operators $\Bog$ exist with an operator norm independent of $\eps$,  for $\eps$ sufficiently small.
	 We readily check the condition 
	 \begin{align}
	 	\int_{B_d(\mathcal C_j) \setminus \mathcal S_j} 
	 	\div(\theta_\eps(x) \chi(P_\eps(x)) = \int_{\partial\mathcal S_j} \chi \cdot n \dd \mathcal{H}^2 = 0.
	 \end{align}
		Therefore, we have
		\begin{align}
			\| \psi\|_{\dot H^1(\R^3 \setminus \mathcal S_j)}^2
			&\leq C \| \nabla (\theta_\eps \chi \circ P_\eps) \|_{L^2(\R^3 \setminus   \mathcal S_j)}^2  \\
			&\leq   C {\eps^{-2}\, }  \|\chi \circ P_\eps\|_{L^2(T_\eps)}^2 + C \|\nabla \chi \circ P_\eps\|_{L^2(T_\eps)}^2.
		\end{align}
		Finally, we observe that, by a change of coordinates and Fubini's principle, 
		\begin{align}
			\|\chi \circ P_\eps\|_{L^2(T_\eps)}^2 \leq \eps  \|\chi \|_{L^2(\partial\mathcal S_j)}^2,
		\end{align}
		and analogously for the gradient. This implies the result.
	 For the change of coordinates we used that for each $r \in (0, \eps c)$, $P_\eps$ is a diffeomorphism (uniformly in $r$ and $\eps$) from $\partial B_r(\mathcal S_j)$ to $\partial \mathcal S_j$ for $c, \eps_0$ sufficiently small. 
		
		This is not the case for a non-closed filament $\mathcal S_j$ which is only Lipschitz. Thus, we need to slightly modify the definition of $P_\eps$. To this end, we first define $\tilde P_\eps$ as the projection from $T_\eps$ to \begin{align}
		    \mathcal Z_j := \{ x \in \mathcal S_j \colon \dist(x,\partial \mathcal S_j) \geq c \eps \}.
		\end{align} 
		After possibly reducing $c$ and  $\eps_0$, $\mathcal Z_j$ satisfies an exterior sphere condition with $R \geq 4 c \eps$, which makes  this projection well-defined and also $P_\eps \colon T_\eps \to \partial \mathcal S_j$,
		\begin{align} \label{def:P_eps}
		    P_\eps(x) := \tilde P_\eps (x) + c \eps \frac{x - \tilde P_\eps (x)}{|x -\tilde P_\eps (x)|}.
		\end{align}
		Then, $P_\eps$ is again a diffeomorphism (uniformly in $r$ and $\eps$) from $\partial B_r(\mathcal S_j)$ to $\partial \mathcal S_j$. 
		Indeed, the exterior sphere condition yields for all $x,y \in T_\eps$
		\begin{align} \label{est:upper.proj.0}
		    |\tilde P_\eps(x) - \tilde P_\eps(y)| \leq 2 |x-y|.
		\end{align}
		Thus, for all $x,y \in \partial B_r(\mathcal S_j)$
		\begin{align*}
		 |P_\eps(x) - P_\eps(y)| \leq 2 |x-y| +  c \eps \frac{|x - y| +  |\tilde P_\eps (x) - \tilde P_\eps (y)|}{r + c \eps} \leq 5 |x-y|.
		\end{align*}

        It remains to check that 
        \begin{align} \label{est:lower.proj}
         |P_\eps(x) - P_\eps(y)| \geq c_0 |x-y|.
        \end{align}
        To this end, we distinguish two cases: in the first case, $|\tilde P_\eps(x) - \tilde P_\eps(y)| \geq |x - y|/8$. In this case \eqref{est:lower.proj} follows from \eqref{est:upper.proj.0} applied to $P_\eps(x), P_\eps(y)$ instead of $x,y$ and using that $\tilde P_\eps(P_\eps(x)) = \tilde P_\eps(x)$.
        In the opposite case, $|\tilde P_\eps(x) - \tilde P_\eps(y)| \leq |x - y|/8$
        Then, using $r \leq c \eps$, \eqref{def:P_eps}  implies
        \begin{align*}
             |P_\eps(x) - P_\eps(y)| \geq \frac {|x-y|}2 - 2|\tilde P_\eps(x) - \tilde P_\eps(y)| \geq \frac {|x-y|}4. 
        \end{align*}
	 \end{proof}

\subsection{Proof of Theorem \ref{lem:slender.several}} \label{zeveryproof}

 Let us first focus on  proving 
 \eqref{est:u-U_i}, observing that the estimate \eqref{est:u_H^1} will follow from \eqref{est:u-U_i}, \eqref{eq:norm.U_i} and  \eqref{eq:norm.U_i-bis}. 
 
 By Lemma \ref{lem:C^eps} it suffices to show  \eqref{est:u-U_i} with $ U_{\mathcal C_i}$ replaced by $ U_{\mathcal C_i^\eps}$, namely
 \begin{align} \label{est:u-U_i^eps}
 	\|u -  |\log \eps|^{-1} U_{\mathcal C_i^\eps}[v ]\|_{\dot H^1(\R^3 \setminus \mathcal S_i)}  
		\leq C | \log \eps |^{-1} \|v\|_{W^{1,\infty} ( \mathcal S_i)},
 \end{align}
 Let
 $$u_r := u - |\log \eps|^{-1} U_{{\mathcal C}_{i}^\eps} [v] .$$
Using \eqref{eq:norm.U_i-bis}  for inside $\mathcal S_j$, for $j\neq i$, and the definition of  $u$, we observe that, to prove  the estimate \eqref{est:u-U_i^eps}, 
it suffices to show that 
\begin{align}
\|u_r \|_{\dot H^1(\mathcal F)} \leq C | \log \eps |^{-1} \|v\|_{W^{1,\infty}{(\mathcal S_i)}}.
\end{align}
We apply Lemma \ref{Claim} with $\chi=u_r$, which satisfies the condition
	\begin{align}
\int_{\partial\mathcal S_j} u_r \cdot n \dd \mathcal{H}^2 =		\int_{ \mathcal S_j} \div u_r  = 0, 
	\end{align}
for $1 \leq j \leq N$. 

By \eqref{eq:U_i.boundary}, since
$$\int_{\partial\mathcal S_i} |\log(\dist(x,\partial \mathcal C_i^\eps))|^2 \dd \mathcal H^2 \leq  C \eps, $$
we have
	\begin{align}
		\| u_r \|^2_{L^2(\partial \mathcal S_i)} &\leq 
		 C \eps  |\log \eps|^{-2}\|v\|_{W^{1,\infty}{(\mathcal S_i)}}^2.
\end{align}
	Moreover from the pointwise estimate \eqref{eq:decay.nabla.U} applied to $U_{\mathcal C_i^\eps}$ and recalling \eqref{eq:distance.cut-off.centerline},  we find 
	\begin{align}
		 \|\nabla_\tau u_r\|^2_{L^2(\partial \mathcal S_i)} &\leq 
		\frac   C  \eps |\log \eps|^{-2}\|v\|_{W^{1,\infty}{(\mathcal S_i)}}^2,
	\end{align}
	where $\nabla_\tau$ denotes the tangential part of the gradient.
	Finally, for $j \neq i$, we observe that \eqref{eq:norm.U_i-bis} implies 
	\begin{align}
		\| u_r \|^2_{L^2(\partial \mathcal S_j)} &\leq 
		 C \eps  |\log \eps|^{-2}\|v\|_{W^{1,\infty}{(\mathcal S_i)}}^2, \\
		 		 \|\nabla_\tau u_r\|^2_{L^2(\partial \mathcal S_i)} &\leq 
	  C  \eps |\log \eps|^{-2}\|v\|_{W^{1,\infty}{(\mathcal S_j)}}^2,
\end{align}	
for some constant depending on $d$.
	
	Thus, Lemma \ref{Claim} yields a divergence-free function $\tilde u_r = \sum \psi_i \in H^1(\mathcal F)$ with $\tilde u_r = u_r$ on $\partial \mathcal F$ and 
	\begin{align}
		\|\tilde u_r \|_{\dot H^1(\mathcal F)} \leq  C |\log \eps |^{-1} \|v\|_{W^{1,\infty}}.
	\end{align}
	We conclude by recalling that $u_r$ solves the Stokes equations in $\mathcal F$. Thus, it minimizes the $\dot H^1$-norm among all divergence-free functions which satisfy the same boundary conditions according to the  Helmholtz minimum dissipation theorem which we now recall.

\begin{thm}
\label{Helmholtz}
If  $u $ in $\dot H^1({\mathcal F})$ satisfies
\begin{align*}
		- \Delta u + \nabla p = 0 , \quad \text{ and }\quad \div u = 0 \quad \text{in } \mathcal F ,
\end{align*}
and $\tilde u $ in $\dot H^1({\mathcal F})$ satisfies
\begin{align*}
	\begin{aligned}
		 \div \tilde u = 0 \quad \text{in } \mathcal F \quad \text{ and }\quad
		\tilde u = u  \quad \text{ on }  \partial  \mathcal F  ,
	\end{aligned}
\end{align*}
then
\begin{align*}
	\|u\|_{\dot H^1({\mathcal F})}  
		\leq \|\tilde u\|_{\dot H^1({\mathcal F})} .
\end{align*}
\end{thm}
\begin{proof}
Using $\tilde u - u$ as a test function in the weak formulation of the PDE for $u$, we find $(\tilde u - u,u)_{\dot H^1({\mathcal F})} =0$
and thus $\|\tilde u\|^2_{\dot H^1({\mathcal F})} = \|\tilde u - u\|^2_{\dot H^1({\mathcal F})} + \| u\|^2_{\dot H^1({\mathcal F})}$.
\end{proof}

For the proof of Theorem \ref{lem:slender.several}, it remains to show \eqref{est:u-U_i.L^p}. This estimate follows directly from \eqref{est:u-U_i}, \eqref{eq:norm.U_i.subcritical} and the following lemma.

	\begin{Lemma}\label{lem:improved}
For all $d> 0$ and $p < 3/2$  there exists $\eps_0(d)>0$ and $C_p(d)>0$ such that for all $\eps < \eps_0$ and all $\dmin \geq d$ the following holds. Let $q \in [p,\infty]$ and let $v \in \dot H^1(\mathcal F) \cap W^{1,q}_{\loc}(\R^3)$ be divergence-free and solve the homogeneous Stokes equations in $\mathcal F$, that is 
\begin{align*}
		- \Delta v + \nabla p = 0 \quad \text{ and } \quad  \div v = 0 \quad \text{in } \mathcal F .
\end{align*}
Then,
\begin{align}
    \|v\|_{W^{1,p}_\loc}  \leq C_p |\log \eps|^{-1/2} \|\nabla v\|_{L^2(\mathcal F)} + \eps^{\frac 2 {q'} - \frac{2}{p'}} \|\nabla v\|_{L^q(\cup_i \S_i)}.
\end{align}
\end{Lemma}

\begin{proof}

Let $K$ be compact, $g \in L^{p'}$, $\supp g \subset K$ and let $w$ solve 
\begin{align*}
- \Delta w + \nabla \pi = \dv g	 \quad \text{ and } \quad  \div w = 0 \quad \text{in } \R^3 .
\end{align*}
Then, by standard regularity theory and Sobolev embedding with $1/{p'} = 1/r  - 1/3$
\begin{align}
    \|\nabla w\|_{L^{p'}(\R^3)} + \|w\|_{L^\infty(\R^3)} \leq  C_p (\|\nabla w\|_{L^{p'}(\R^3)} + \|\nabla w\|_{L^{r}(\R^3)}) \leq   C_{K,p}\|g\|_{L^{p'}(\R^3)},
\end{align}
where  we used in the last step that $g \in L^r(\R^3)$ due to its compact support.
Thus the desired estimate for $\|\nabla v\|_{L^p(K)}$ will follow by duality once we have shown
\begin{align} \label{est:nabla.v.g}
&\int_{\R^3} \nabla v \cdot g 
= 2 \int_{\R^3}  D w : D v \\
&\leq C_p \left(|\log \eps|^{-1/2} \|\nabla v\|_{L^2(\mathcal F)} + {\eps^{\frac 2 {q'} - \frac{2}{p'}} \|\nabla v\|_{L^q(\cup_i \S_i)}}\right) 
\left(\|\nabla w\|_{L^{p'}(\R^3)} + \|w\|_{L^\infty(\R^3)} \right).
\end{align}
The estimate for $\|v\|_{L^p(K)}$ follows along the same lines by considering the problem $- \Delta w' + \nabla \pi' = g$ and the fact that 
\begin{align*} 
\int v \cdot g =  2 \int_{\R^3}  D w' : D v . 
\end{align*}
Note that the regularity of $w'$ is even better than the regularity of $w$.

To show \eqref{est:nabla.v.g}, we split the left hand side into
\begin{align}
         2 \int_{\R^3}  D w : D v  = 2 \int_{\cup \mathcal S_i}  D w : D v + 2 \int_{\mathcal F} D w : D v =:I_1 + I_2
\end{align}
By H\"older's inequality, we estimate 
\begin{align}
    I_1 \leq C   {\eps^{\frac 2 {q'} - \frac{2}{p'}} \|\nabla v\|_{L^q(\cup_i \S_i)}} \|\nabla w\|_{L^{p'}},
\end{align}
where we used $|\cup \mathcal S_i| \leq C \eps^2$.
Moreover, by some integrations by parts, we have that 
\begin{align}
            I_2 = \sum_{1 \leq i \leq N} \int_{\partial \mathcal S_i} \Sigma(v,p) n \cdot w =  2 \int_{\R^3 \setminus \mathcal F} D \varphi : D v
    \leq  2 \|v\|_{\dot H^1(\mathcal F)} \|\varphi\|_{\dot H^1(\mathcal F)}
\end{align}
for all divergence-free functions $\varphi \in \dot H^1(\R^3)$ with $\varphi = w$ in $S_i$.

Therefore, it remains to show that such a function $\varphi$ exists which satisfies
\begin{align} \label{est:varphi}
    \|\varphi\|_{\dot H^1} \leq C |\log \eps|^{-1/2} \left(\|\nabla w\|_{L^{p'}(\R^3)} + \|w\|_{L^\infty(\R^3)} \right).
\end{align}
The construction of such a function $\varphi$ is similar to the construction in the proof of Lemma \ref{Claim}. However, since we have better control of the function  which we want to extend, 
we will see that using a different cut-off function yields better estimates than in  Lemma \ref{Claim}.
More precisely, we consider $R\geq 0$ (independent of $\eps$)  and $\eps_0$ sufficiently small, such that $\mathcal S_i \subset B_{\eps R}(\mathcal C_i) \subset B_d(\mathcal C_i)$ for all $1 \leq i \leq N$. 
Then, we define
\begin{align}
    \theta_\eps (x) := \begin{cases}
                \frac{\log(\dist(x, \cup_i \mathcal C_i)) - \log d}{\log (\eps R) - \log d} & \qquad \text{in } \cup_i \left(B_d( \mathcal C_i) \setminus B_{\eps R}(\mathcal C_i) \right) \\
                1 & \qquad \text{in } \cup_i  B_{\eps R}(\mathcal C_i) \\
               0 & \qquad \text{in } \R^3 \setminus \left(\cup_i  B_d( \mathcal C) \right).
                \end{cases}
\end{align}
Note that this cut-off function corresponds to the $2$-dimensional capacitary function of a ball $B_\eps(x)$ within $B_d(x)$.

Now we define
\begin{align}
    \varphi := \theta_\eps w - \Bog(\nabla \theta_\eps \cdot w).
\end{align}
Note that $\supp \nabla \theta_\eps \subset \cup B_d(\mathcal C_i) \setminus B_{\eps R}(\mathcal C_i)$. We may apply Lemma \ref{nioul} to this domain to estimate the Bogovski\u{i} operator since $B_{\eps R}(\mathcal C_i)$ just corresponds to a filament with centerline $\mathcal C_i$ and circular cross section.

Thus, we can estimate
\begin{align} \label{est:varphi.1}
 \| \varphi\|_{\dot H^1} &\leq \|w \nabla \theta_\eps\|_{L^2(\R^3)} + \|\theta_\eps \nabla w\|_{L^2(\R^3)} \\ \nonumber
 &\leq \|w\|_{L^\infty(\R^3)} \| \nabla \theta_\eps\|_{L^2(\R^3)} + \|\theta_\eps\|_{L^r(\R^3)} \|\nabla w\|_{L^{p'}(\R^3)} ,
\end{align}
where $\frac 1 r + \frac 1 {p'} = \frac 1 2$.
To conclude, we assume  that $\eps_0$ is chosen small enough such that $\log d - \log (\eps R) \geq \frac 1 2  {|\log \eps|}$. Then,  
\begin{align} \label{est:nabla.theta}
    \| \nabla \theta_\eps\|^2_{L^2(\R^3)} \leq C |\log \eps|^{-2} \sum_{1 \leq i \leq N} \int_{B_d( \mathcal C_i) \setminus B_{\eps R}(\mathcal C_i)} \frac 1 {\dist^2(x,\mathcal C_i)}  \dd x\leq C |\log \eps|^{-1} ,
\end{align}
and for all $r < \infty$,
\begin{align} \label{est:theta}
    \|\theta_\eps\|^r_{L^r(\R^3)} \leq C |\log \eps|^{-r} \sum_{1 \leq i \leq N} \int_{B_d( \mathcal C_i)}  |\log(\dist(x,\mathcal C_i))|^r +   |\log d|^r \dd x \leq C_r  |\log \eps|^{-r} .
\end{align}
Inserting \eqref{est:nabla.theta} and \eqref{est:theta} in \eqref{est:varphi.1} yields \eqref{est:varphi}.
This concludes the proof.
\end{proof}

\subsection{Proof of Corollary  \ref{prop:asymptotic.resistance}}
\label{sec-prop:asymptotic.resistance}
This subsection is devoted to the proof of Corollary  \ref{prop:asymptotic.resistance}. 
Let us therefore consider a vector field $v $ in $W^{1,\infty}(\cup_{j=1}^N \mathcal S_j)$ and divergence-free. 
	Our aim is to establish the inequality \eqref{eq:Faxen.force}
	 regarding the approximation of
	 \begin{align}
	 \int_{\cup_{j=1}^N \partial \mathcal S_j}  (\Sigma (V_{i,\alpha}, P_{i,\alpha}) n) \cdot v \, d\mathcal{H}^2  \quad \text{by} \quad 
 |\log \eps|^{-1} 
	   I_{\mathcal{C}_i }  [v_{i,\alpha}, v   ],	\end{align}
where the functions $(V_{i,\alpha},P_{i,\alpha})$ are defined in \eqref{sti}.
	Let 
		\begin{align}
		u_{i,\alpha} :=  |\log \eps|^{-1} U_{\mathcal C_i}[v_{i,\alpha}] \quad \text{ and } \quad 
			p_{i,\alpha} :=  |\log \eps|^{-1} P_{\mathcal C_i}[v_{i,\alpha}],
		\end{align}
	where  $U_{\mathcal C_i}$ and $P_{\mathcal C_i}$ are the operators respectively  defined in \eqref{def:U} and \eqref{def:P}. By \eqref{singm}, 
		\begin{align}
		\label{dela}
	-\Delta u_{i,\alpha} + \nabla p_{i,\alpha} = |\log \eps|^{-1} \mu_{\mathcal C_i}[v_{i,\alpha}], \quad  \div  u_{i,\alpha} = 0, 
		\end{align}
			in the sense of distributions in $\R^3$. 
		Let us decompose 
	 \begin{align}
	 \label{decompos}
	 &\int_{\cup_{j=1}^N \partial \mathcal S_j}  (\Sigma (V_{i,\alpha}, P_{i,\alpha}) n) \cdot v \, d\mathcal{H}^2 -  |\log \eps|^{-1} 
	   I_{\mathcal{C}_i }  [v_{i,\alpha}, v   ]
	    \\ 
	   \nonumber &= \int_{\partial \mathcal S_{i}} \Sigma(u_{i,\alpha},p_{i,\alpha}) n \cdot v \, d\mathcal{H}^2
	   -  |\log \eps|^{-1} I_{\mathcal{C}_i } [ v_{i,\alpha},v ] 
	   \\ \nonumber &\quad +  \int_{\cup_{j=1}^N \partial \mathcal S_j}  (\Sigma (V_{i,\alpha}, P_{i,\alpha}) n) \cdot v \, d\mathcal{H}^2 - \int_{\partial \mathcal S_{i}} \Sigma(u_{i,\alpha},p_{i,\alpha}) n \cdot v \, d\mathcal{H}^2.
	\end{align}

By an integration by parts inside the filament $\mathcal S_{i}$ and recalling the definition \eqref{lau}, we deduce from 	\eqref{dela} that 
	\begin{align} \nonumber
	\int_{\partial \mathcal S_{i}} \Sigma(u_{i,\alpha},p_{i,\alpha}) n \cdot v \, d\mathcal{H}^2
	&= \int_{\mathcal S_{i}}  (-\Delta u_{i,\alpha} + \nabla p_{i,\alpha}) \cdot v + \int_{\mathcal S_{i}} \Sigma(u_{i,\alpha},p_{i,\alpha}) : D(v) \\
	\label{24}
	& =  |\log \eps|^{-1} I_{\mathcal{C}_i } [ v_{i,\alpha},v ]
	+ \int_{\mathcal S_{i}} \Sigma(u_{i,\alpha},p_{i,\alpha}) : D(v) .
	\end{align}
By \eqref{eq:decay.nabla.U}, and by observing that the pressure $P_{\mathcal C_i}$ satisfies the same pointwise decay estimates as the velocity gradient $\nabla U_{\mathcal C_i}$
	\begin{align}
	\label{doute}
	    \left|\int_{\mathcal S_{i}}  \Sigma(u_{i,\alpha},p_{i,\alpha}) : D(v)\right| \leq  C \|v\|_{W^{1,\infty}{(\mathcal S_{i})}} \int_{\mathcal S_{i}}  \frac 1 {\dist(x,\mathcal C_i)} \dd x \leq C \eps \|v\|_{W^{1,\infty}{(\mathcal S_{i})}} .
	\end{align}
		Let $w := \sum_{i=1}^N w_i$, where $w_i$ is the solution to \eqref{eq:Stokes.data.filaments}.
Then $w = v$ in $\cup_{j=1}^N \mathcal S_j$ and thus,  
	\begin{align*}
		 \int_{\cup_{j=1}^N \partial \mathcal S_j}  (\Sigma (V_{i,\alpha}, P_{i,\alpha}) n) \cdot v \, d\mathcal{H}^2 
	&=  \int_{\cup_{j=1}^N \partial \mathcal S_j}  (\Sigma (V_{i,\alpha}, P_{i,\alpha}) n) \cdot w \, d\mathcal{H}^2 
	\\ & = 
	 \int_{ \mathcal F} D(V_{i,\alpha}) : D(w) 
	 	\\ & = 
	 \int_{ \R^3 \setminus \mathcal S_i } D(V_{i,\alpha}) : D(w) ,
	\end{align*}
since 
 $D (V_{i, \alpha}) = 0 $ in $\cup_{{j\neq i}} \mathcal S_j$. On the other hand, it follows from \eqref{dela} that 
 	\begin{align*}
\int_{\partial \mathcal S_i} \Sigma (u_{i,\alpha}, p_{i,\alpha})  n  \cdot v {d\mathcal{H}^2} &= \int_{\partial \mathcal S_i} \Sigma (u_{i,\alpha}, p_{i,\alpha})  n  \cdot w  {\
d\mathcal{H}^2}
\\ &=
 \int_{ \R^3 \setminus \mathcal S_i } D(u_{i,\alpha}) : D(w) {\
d\mathcal{H}^2},
	\end{align*}
	so that 
		\begin{align*}
&	 \int_{\cup_{j=1}^N \partial \mathcal S_j}  (\Sigma (V_{i,\alpha}, P_{i,\alpha}) n) \cdot v \, d\mathcal{H}^2 
		-  \int_{\partial \mathcal S_i} \Sigma (u_{i,\alpha}, p_{i,\alpha})  n  \cdot w {d\mathcal{H}^2}
	\\ & \quad	= 
	\int_{  \R^3 \setminus \mathcal S_i } D(V_{i,\alpha} - u_{i,\alpha}) : D(w) .
			\end{align*}
Therefore, by Theorem \ref{lem:slender.several}, we deduce that  
	\begin{align}
	\nonumber
		&\left| \int_{\cup_{j=1}^N \partial \mathcal S_j}  (\Sigma (V_{i,\alpha}, P_{i,\alpha}) n) \cdot h \, d\mathcal{H}^2 
		-  \int_{\partial \mathcal S_i} \Sigma (u_{i,\alpha}, p_{i,\alpha})  n  \cdot v {d\mathcal{H}^2}
\right| \\ 	\nonumber
& \leq 
 \| u_{i,\alpha} - V_{i,\alpha}\|_{\dot H^1(\R^3 \setminus \mathcal S_i)}  \|w\|_{\dot H^1(\R^3)} 
 \\ \label{29}
		& \leq C | \log \eps |^{-3/2} \|v\|_{W^{1,\infty}{(\cup_{j}\mathcal S_j)}}.
	\end{align}
	Gathering 
	 \eqref{decompos}, \eqref{24}, 
\eqref{doute}
and \eqref{29} we arrive at 
 \eqref{eq:Faxen.force} and
this finishes the proof of Corollary  \ref{prop:asymptotic.resistance}.

%%%%%%%%%%%%%%%%%%%%%%%%%%%%%%%%%
 \section{Shape derivatives}
 \label{shape-derivative}

Since the filaments evolve in time, it is necessary to tackle the behaviour of the solutions to the Stokes system with Dirichlet data in  the  filaments under rigid displacements of the filaments. To this end we establish the following bound on the shape derivatives, with respect to rigid motions of the filaments, of the interaction energy of two solutions to the Stokes system with fixed values in the filaments. This bound is uniform with respect to $\eps$ and to the positions for which a positive minimal distance between the centerlines is guaranteed.   
\begin{Proposition} \label{lem:shapeDerivatives}  
For all $d > 0$, there are $C(d)>0$ and $\eps_0(d)>0$ such that for all $\eps$ in $(0, \eps_0)$, for all  $(h,Q)  $ in $ \R^{3N} \times SO(3)^N$ such that the corresponding minimal distance $\dmin$ between the centerlines defined in	\eqref{dmin}, satisfies $\dmin \geq d$,
and for all divergence-free vector fields $\varphi_1$ and  $\varphi_2$
in $W^{2,\infty}(\R^3)$, the following holds true. The corresponding solutions  $\psi_1$ and $\psi_2$   in 	$ \dot H^1(\mathcal F)$
to the Stokes problem
\begin{align} \label{def:psi}
	\begin{aligned}
		- \Delta \psi_i + \nabla p_i &= 0 \quad \text{ and } \quad  \div \psi_i = 0 \quad \text{in } \mathcal F, \\
		\psi_i &= \varphi_i  \quad \text{on } \cup_{j=1}^N \partial \mathcal S_j ,
	\end{aligned}
\end{align}
where the position $\mathcal S_j$ of the filaments are deduced from their original positions by $(h,Q)$ as in \eqref{def:Stourne}, satisfy
	\begin{align} \label{vendred}
	&	| \nabla_{h,Q} \big(D (\psi_1),D (\psi_2)\big)_{L^2(\mathcal F)}| 
\\ \nonumber &\quad 	\leq C |\log \eps|^{-1}\,  \Big( \|\varphi_1\|_{W^{1,\infty} (\R^3) } \|\varphi_2\|_{W^{2,\infty} (\R^3)}
	+ 
	\|\varphi_1\|_{W^{2,\infty} (\R^3) } \|\varphi_2\|_{W^{1,\infty} (\R^3)}
\Big)	.
	\end{align}
\end{Proposition}
Above the notation $\big(\cdot , \cdot\big)_{L^2(\mathcal F)}$ stands for the inner product in $L^2(\mathcal F)$.
\begin{proof}[Proof of Proposition \ref{lem:shapeDerivatives}]
To establish the bound \eqref{vendred} of the shape derivative of the  interaction energy
$$\big(D (\psi_1),D (\psi_2)\big)_{L^2(\mathcal F)} ,$$ we estimate the difference of such interaction energies 
 corresponding to two close configurations of the filaments.
To this end, let $d>0$, let $\varphi_1,\varphi_2 $  divergence-free vector fields in $W^{2,\infty}(\R^3)$. Let  $(h,Q) \in \R^{3N} \times SO(3)^N$ with $\dmin \geq d$
 and let $(\widecheck h,\widecheck Q) \in \R^{3N} \times SO(3)^N$ such that $|(h, Q) - (\widecheck h,\widecheck Q)| \leq \delta$ small enough, to be chosen later.
For $i = 1,2$, let $\psi_i$ and $ \widecheck \psi_i$,  the solutions to \eqref{def:psi} corresponding to the same boundary data $\varphi_i$,  and to the filaments' positions $(h, Q)$ and $(\widecheck h,\widecheck Q)$, respectively. Corresponding pressures are denoted by $p_i$ and $\widecheck p_i$.
Thus, for $i = 1,2$, on the one hand  $( \psi_i ,p_i )$ satisfies \eqref{def:psi}
and on the other hand $(\widecheck  \psi_i,\widecheck  p_i )$ satisfies 
\begin{align} \label{def:psi-bar}
	\begin{aligned}
		- \Delta \widecheck \psi_i + \nabla \widecheck p_i &= 0 \quad \text{ and } \quad  \div \widecheck  \psi_i = 0 \quad \text{in } \widecheck {\mathcal F} , \\
	\widecheck 	\psi_i &=  \varphi_i  \quad \text{on } \cup_{j=1}^N \partial  \widecheck {\mathcal S}_j ,
	\end{aligned}
\end{align}
where $\mathcal F$ and $\widecheck {\mathcal F}$  denote the fluid domain respectively corresponding to   $(h, Q)$ and  $(\widecheck h, \widecheck Q)$, while the sets $ {\mathcal S}_j$ and $\widecheck {\mathcal S}_j$   are the positions respectively occupied by the filaments  in the two configurations.

To prove  Proposition \ref{lem:shapeDerivatives} 
 we are going to prove that  there are $C(d)>0$ and $\eps_0(d) >0$ such that for all $\eps$ in $(0, \eps_0)$, 
\begin{align} \label{est:psi_1.psi_2}
	&	|(D(\psi_1),D(\psi_2))_{L^2(\mathcal F)}  - (D(\widecheck \psi_1) ,D(\widecheck \psi_2))_{L^2(\widecheck{\mathcal F})}|
		\\  \nonumber & \quad \leq C\delta |\log \eps|^{-1}\, 
		 \Big( \|\varphi_1\|_{W^{1,\infty} (\R^3) } \|\varphi_2\|_{W^{2,\infty} (\R^3)}
	+ 
	\|\varphi_1\|_{W^{2,\infty} (\R^3) } \|\varphi_2\|_{W^{1,\infty} (\R^3)}
\Big).
\end{align}

Without loss of generality, we may restrict the proof  to the case where only one filament is displaced, say the first one, so that the positions of the other filaments is the same for the two configurations, that is  $(\widecheck h_j, \widecheck Q_j) = (h_j,Q_j)$ for all $2 \leq j  \leq N$. As a consequence,  
\begin{equation}
    \label{lezautre}
   {\mathcal S}_j = \widecheck {\mathcal S}_j, \text{ for all }  2 \leq j  \leq N. 
\end{equation}
    Moreover, up to a change of frame,  we may also assume without loss of generality that the position of the first filament satisfies  $(h_1,Q_1) = (0,\Id)$, and we recall that the position $(\widecheck h_1, \widecheck Q_1)$ of the first filament in the second configuration is in general different from $(0,\Id)$ but $\delta$-close so that 
  \begin{equation} \label{delta-c}
 |\widecheck h_1 | + |\widecheck Q_1 - \Id | \leq   \delta.
  \end{equation}
\ \par \

\noindent\emph{Step 1: Construction of a suitable deformation.} 
In this first step we introduce an auxiliary vector field associated with $\widecheck  \psi_i$, see \eqref{kejetrans} for the definition, but which solves a Stokes system in $\mathcal F$, see \eqref{subtilde} for the exact system. 

 We choose a neighborhood $\mathcal T_1$ defined by $\mathcal T_1 =  \mathcal S_1^{2\eps}$ (i.e. the filament corresponding to $2 \eps$ instead of $\eps$). Lemma \ref{nioul} ensures that the set 
 \begin{equation} \label{J}
     \mathcal J := B_{d/4}(\mathcal C_1) \setminus \mathcal T_1 , 
 \end{equation}
 is a John domain with a  constant $Z$ independent of $\eps$ for $\eps$ sufficiently small. 
Let $\eta \in C_c^{\infty}(B_{d/4}(\mathcal C_1))$ be a nonnegative  cut-off function such that
$\eta = 1$ in  $\mathcal T_1$.

Let $\phi$ be the function from $\R^3$ to $\R^3$ such that for all $x$ in $\R^3$, 
\begin{equation} \label{def-phi}
	\phi(x) := x + \left((\widecheck Q_1 - \Id)x + \widecheck h_1   \right) \eta(x).
\end{equation}
By construction, 
\begin{equation} \label{dodomain}
\phi(\mathcal F) = \widecheck {\mathcal  F} , \quad \text{ in particular } \quad
\phi(\mathcal S_1) = \widecheck {\mathcal  S}_1 ,
\end{equation}
 and
\begin{equation} \label{elaid}
\phi \vert_{\R^3 \setminus (B_{d/4}(\mathcal C_1))} = \Id_{\R^3 \setminus (B_{d/4}(\mathcal C_1))} .
\end{equation}
Moreover $\phi$ and $\phi^{-1}$ are diffeomorphisms from $\R^3$ to $\R^3$ 
with 
\begin{equation}
    \label{pourcp}
    |\nabla \phi| + |\nabla \phi^{-1}| \leq C ,
\end{equation}
 for a constant $C$ independent of $d$ for $\delta$ sufficiently small. 
To see that $\phi$ is injective for $\delta$ sufficiently small, and to estimate $\nabla \phi^{-1}$, we estimate $|\phi(x_1) - \phi(x_2)|$ from below for all $x_1,x_2 \in \R^3$. Clearly, if $x_i \not \in \supp \eta$ for $i = 1,2$, the estimate is trivial.
Let us assume $x_2 \in \supp \eta$ and note that since $h_1 = 0$ this implies $|x_2| \leq C + d$,
where the constant $C$ depends only on the reference filament $\bar{\mathcal S}_1$.
Thus, using \eqref{def-phi} and \eqref{delta-c}, we have that 
\begin{align*}
    &|\phi(x_1) - \phi(x_2)| \\
    &\geq |x_1 - x_2| - \big|\big((\widecheck Q_1 - \Id)x_1 + \widecheck h_1   \big) \eta(x_1) - \big((\widecheck Q_1 - \Id)x_2 + \widecheck h_1   \big) \eta(x_2)\big| \\
    & \geq |x_1 - x_2| - |\widecheck Q_1 - \Id||x_1 - x_2| \eta(x_1) - \big|(\widecheck Q_1 - \Id)x_2 + \widecheck h_1   \big||\eta(x_1) - \eta(x_2)| \\
    & \geq \frac 1 2 |x_1 - x_2|,
\end{align*}
 for $\delta$ sufficiently small. 

Furthermore, for all $x$ in $\R^3$, we set
\begin{equation} \label{def:Phi}
	\Phi(x) := 	{\nabla\phi(x)}  = \Id +   (\widecheck Q_1^T - \Id) \eta(x) + 
	\left((\widecheck Q_1 - \Id)x + \widecheck h_1   \right)   \otimes \nabla \eta,
\end{equation}
with the convention that $\nabla\phi = (\partial_i \phi_j )_{i,j}$.
From the definition of $\phi$ and $\Phi$ in \eqref{def-phi} and \eqref{def:Phi}
it follows that 
\begin{align} \label{esti-mir}
|(\nabla \phi^T) \nabla \phi - \mathrm{Id}|+
|\nabla \Phi|+ |\nabla^2 \Phi| &\leq C  \delta  \  \mathbf{1}_{\mathcal J},    
\end{align}
where $  \mathbf{1}_\mathcal{J}$ is  the indicator function of the set $\mathcal{J}$  defined in \eqref{J}.

Since $\eta = 1$ in $\mathcal T_1$, 
\begin{equation} \label{soliv-v}
\text{for all} \, x \in \mathcal T_1 , \quad 
(\Phi \widecheck \psi_i \circ \phi)(x) = \widecheck Q_1^T  \psi_i ( \widecheck Q_1 x + \widecheck h_1 ) .
	\end{equation}
In particular the vector field 
$\Phi \widecheck \psi_i \circ \phi$   is divergence-free in   $\mathcal T_1$.
We define
\begin{equation} \label{kejetrans}
	\tilde \psi_i := \Phi \widecheck \psi_i \circ \phi  - \Bog (\dv (\Phi \widecheck \psi_i \circ \phi)) \quad \text{ and } \quad  
	\tilde p_i = \widecheck p_i \circ \phi ,
\end{equation}
where $\Bog$ denotes a Bogovski\u{i} operator provided by Theorem \ref{John} in the domain $\mathcal J$, which satisfies
\begin{equation}\label{support}
\supp(\dv(\Phi \widecheck \psi_i \circ \phi)) \subset \mathcal J .
\end{equation}
Recall that, according to Theorem \ref{John} and the fact that 
$\mathcal J$ is a John domain with a  constant $Z$ independent of $\eps$ for $\eps$ sufficiently small,  the operator $\Bog$ mentioned above satisfies that there exists $C>0$ such that for $\eps$ sufficiently small,
\begin{align} \label{Bog-J}
\text{for all } \, f \in L^2_0(\mathcal J), \quad 	\|\Bog f\|_{H^1 (\mathcal J) } \leq C \|f \|_{L^2 (\mathcal J)}.
\end{align}

\noindent\emph{Step 2: The divergence of $\Phi \widecheck \psi_i \circ \phi$.} In this step, we prove the following identity,  which in combination with \eqref{esti-mir}, is helpful below, see \eqref{giro}, to prove that in $\mathcal J$, 
$\dv ( \Phi \widecheck \psi_i \circ \phi)$ is a $O(\delta)$:
\begin{align}\label{idenden}
\dv (\Phi \widecheck \psi_i \circ \phi) =
(\dv \Phi) \cdot (\psi_i \circ \phi) +
\Big(
  (\nabla \phi)^T  \nabla \phi - Id \Big) :  ((\nabla \widecheck \psi_i) \circ \phi ) .  
\end{align}

Let us first recall that for some regular enough fields  of matrices $A$ and  of vectors $v$, the following identity holds true: 
\begin{align} \label{piola}
\div (Av) = (\dv A) \cdot v + A : \nabla v  ,
		\end{align}
where the operator $\dv$ has to be applied row-wise to $A$. 

In particular, by applying \eqref{piola} to the case where $A= \Phi$ and $v= \widecheck \psi_i \circ \phi$, and recalling that $\Phi = \nabla \phi$, we obtain that %
\begin{align} 
\dv (\Phi \widecheck \psi_i \circ \phi) &=
(\Div \Phi) \cdot (\psi_i \circ \phi) +
\nabla \phi  :  \nabla ( \widecheck \psi_i \circ \phi )\nonumber \\
&= (\Div \Phi) \cdot (\psi_i \circ \phi) +
(\nabla \phi)^T \nabla \phi  :  ((\nabla  \widecheck \psi_i) \circ \phi ), \label{idendown}
\end{align}
where we used in the last identity that $A:BC = B^T A : C$ for any $A,B,C \in R^{3 \times 3}$.

Finally, by definition  $\dv \widecheck  \psi_i = Id : \nabla \widecheck \psi = 0$. Using this in \eqref{idendown}  yields \eqref{idenden}. 
\ \par \

\noindent\emph{Step 3: The Stokes system solved by $\tilde \psi_i$, for $i=1,2$.} 
Observe the following fact: 
\begin{equation} \label{soliv-tout}
\text{for all} \, x \in \mathcal T_1 , \quad 
\tilde \psi_i (x) = \widecheck Q_1^T \widecheck \psi_i ( \widecheck Q_1 x + \widecheck h_1 )
\quad \text{ and } \quad 
	\tilde p_i(x) =  \widecheck p_i  ( \widecheck Q_1 x + \widecheck h_1 ) .
	\end{equation}
Using $\Phi = (\nabla \phi)^T$ and some tensor calculus similar as in the previous step,  we find in $\mathcal F$
\begin{align*}
    - \Delta \tilde \psi_i + \nabla  \tilde p_i &=  - (\Delta \Phi) \widecheck \psi_i \circ \phi + \Delta \Bog (\dv (\Phi \widecheck \psi_i \circ \phi))
    - 2 \nabla \Phi \nabla (\widecheck \psi_i \circ \phi)  \\
    & -\Phi \Big( (\nabla \phi)^T \nabla \phi  : \nabla^2 \widecheck \psi_i \circ \phi \Big) + \Phi \nabla \widecheck p_i \circ \phi \\
    & = -(\Delta \Phi) \widecheck \psi_i \circ \phi + \Delta \Bog (\dv (\Phi \widecheck \psi_i \circ \phi)) - 2 \nabla \Phi \nabla (\widecheck \psi_i \circ \phi) \\
    &+ \Phi (\Id - (\nabla \phi)^T\nabla \phi ) : \nabla^2  \widecheck \psi_i \circ \phi , 
\end{align*}
where we used that $   \nabla \widecheck p_i = \Delta \widecheck \psi_i$  in $\mathcal F$. 
 
Concerning the last term, a further manipulation leads to 
\begin{align*}
\Phi (\Id - (\nabla \phi)^T \nabla \phi ) : \nabla^2 \widecheck \psi_i \circ \phi    &=  \Phi (\Id - (\nabla \phi)^T\nabla \phi ) :(\nabla \phi)^{-T} \nabla( \nabla \widecheck \psi_i \circ \phi) \\
    &= \dv \left( \Phi (\Id - (\nabla \phi)^T \nabla \phi ) (\nabla \phi)^{-T} : \nabla \widecheck \psi_i \circ \phi \right)  \\
    &\quad - \dv\left( \Phi (\Id - (\nabla \phi)^T\nabla \phi ) : (\nabla \phi)^{-T} \right)  \nabla \widecheck \psi_i \circ \phi.
\end{align*}

Therefore, and relying on \eqref{def:psi-bar}, \eqref{def-phi}, \eqref{kejetrans}, \eqref{soliv-tout}  and 
\eqref{dodomain}, 
 we obtain that $\tilde \psi_i$ solves the following Stokes system:
 \begin{subequations}
 \label{subtilde}
\begin{align}
	- \Delta \tilde \psi_i + \nabla \tilde p_i &= \dv g_i + f_i \quad \text{ and } \quad  \dv \tilde \psi_i  = 0 \quad \text{in }  \mathcal F, \\
	\tilde \psi_i (x) &= \widecheck Q_1^T  \varphi_1(\widecheck Q_1 x  + \widecheck h_1 )  \quad \text{in } \mathcal S_1, \\
	\tilde \psi_i &= \varphi_i \quad \text{in } \mathcal S_j \quad \text{for all } j \neq 1.
	\end{align}
	 \end{subequations}
where
\begin{align}
    g_i &:= - \nabla \Bog (\dv (\Phi \widecheck \psi_i \circ \phi)) + \Phi (\Id -  (\nabla \phi)^T\nabla \phi) (\nabla \phi)^{-T}  \nabla \widecheck \psi_i \circ \phi, \label{def:g_i}\\
    f_i &:= -(\Delta \Phi) \widecheck \psi_i \circ \phi - 2 \nabla \Phi \nabla (\widecheck \psi_i \circ \phi) 
    + \dv\left( \Phi (\Id - (\nabla \phi)^T\nabla \phi ) (\nabla \phi)^{-T} \right)  \nabla \widecheck \psi_i \circ \phi. \label{def:f_i}
\end{align}
We observe that $g_i$ and $f_i$ are compactly supported in $ \mathcal J$, see \eqref{esti-mir}. 

 \ \par \

\noindent\emph{Step 4: Estimate of $g_i$ and $f_i$, for $i=1,2$.}
We start with estimating the term of $g_i$  involving the Bogovski\u{i} operator.
First, using that 
$\phi$ is a diffeomorphism  satisfying 
\eqref{pourcp}, 
we obtain that
\begin{align} \label{ctcent}
    \|\nabla \widecheck \psi_i \circ \phi\|_{L^2(\R^3)} \leq C     \|\nabla \widecheck \psi_i\|_{L^2(\R^3)} \leq C |\log \eps|^{-1/2} \|\varphi_i\|_{W^{1,\infty}(\R^3)} ,
\end{align}
by applying Theorem \ref{lem:slender.several}. 
Moreover, by applying \eqref{piola}, it follows 
\begin{align} \label{giro}
|\dv (\Phi \widecheck \psi_i \circ \phi)| 
& \leq C {\delta} \left( |\widecheck \psi_i \circ \phi| 
+ |\nabla \widecheck \psi_i \circ \phi|\right).
\end{align}
Therefore, by using \eqref{Bog-J}
\begin{equation*}
    \|\nabla(\Bog(\dv(\Phi \widecheck \psi_i \circ \phi)))\|_{L^2} \leq C
    \|\dv(\Phi \widecheck \psi_i \circ \phi)\|_{L^2} 
    \leq  C \delta |\log \eps|^{-1/2} \|\varphi_i\|_{W^{1,\infty}(\R^3)},
\end{equation*}
where we used 
\eqref{ctcent} to get the last inequality.

Similarly, we  can estimate all the other terms on the right hand sides of \eqref{def:f_i} and \eqref{def:g_i} by 
\begin{align} \label{midi4}
     \|g_i\|_{L^2(\mathcal F)} + \|f_i\|_{L^{\frac 6 5 }(\mathcal F)} \leq C \delta |\log \eps|^{-1/2} \| \varphi_i\|_{W^{1,\infty}(\R^3)}.
\end{align}
\ \par \

\noindent\emph{Step 5: The interaction energy $2(D(\tilde \psi_1), D(\tilde \psi_2))_{L^2(\mathcal F)}$.}
By \eqref{elaid} and \eqref{lezautre}, for $2\leq j \leq N$, 
\begin{align*}
\int_{\partial \widecheck{\S}_j} 
	\Sigma(\widecheck \psi_1,\widecheck p_1) n \cdot \widecheck \psi_2 =
	 \int_{\partial {\S}_j} 
	\Sigma(\tilde \psi_1,\tilde p_1) n \cdot \tilde \psi_2 .
\end{align*}
On the other hand, by \eqref{soliv-tout}, 
\eqref{def-phi} and 
\eqref{dodomain}, the chain rule and a change of variable (observe that the normal is also rotated), 
\begin{align*}
	\int_{\partial \widecheck{\S}_1} 
	\Sigma(\widecheck \psi_1,\widecheck p_1) n \cdot \widecheck \psi_2 =
	 \int_{\partial {\S}_1} 
	\Sigma(\tilde \psi_1,\tilde p_1) n \cdot \tilde \psi_2 .
\end{align*}
Therefore, by some integrations by parts, 
from the two previous identities, \eqref{def:psi-bar} and \eqref{subtilde}, we arrive at 
\begin{align*}
	2(D(\widecheck \psi_1), D(\widecheck \psi_2))_{L^2(\widecheck{\mathcal F})} 
	&= \sum_{j=1}^N \int_{\partial \widecheck{\S}_j} 
	\Sigma(\widecheck \psi_1,\widecheck p_1) n \cdot \widecheck \psi_2 \\
	&= \sum_{j=1}^N  \int_{\partial {\S}_j} 
	\Sigma(\tilde \psi_1,\tilde p_1) n \cdot \tilde \psi_2 \\
&= 	 2(D(\tilde \psi_1), D(\tilde \psi_2))_{L^2(\mathcal F)} 
+  (g_1,  \nabla \tilde \psi_2)_{L^2(\mathcal F)} +  (f_1,  \tilde \psi_2)_{L^2(\mathcal F)}.
\end{align*}
Therefore, 
\begin{align*}
	&2(D(\widecheck \psi_1), D(\widecheck \psi_2))_{L^2(\widecheck{\mathcal F})} - 2(D(\psi_1), D(\psi_2))_{L^2(\mathcal F)}\\
	&=   2(D(\tilde \psi_1), D(\tilde \psi_2))_{L^2(\mathcal F)} - 2(D(\psi_1), D(\psi_2))_{L^2(\mathcal F)} + (g_1,  \nabla \tilde \psi_2)_{L^2(\mathcal F)} +  (f_1,  \tilde \psi_2)_{L^2(\mathcal F)}\\
	&=   2(D(\tilde \psi_1 - \psi_1), D(\tilde \psi_2))_{L^2(\mathcal F)} - 2(D(\psi_1), D(\psi_2-\tilde \psi_2))_{L^2(\mathcal F)} \\
	&+ (g_1,  \nabla \tilde \psi_2)_{L^2(\mathcal F)} +  (f_1,  \tilde \psi_2)_{L^2(\mathcal F)}.
	\end{align*}
	Thus, by the Cauchy-Schwarz inequality, the H\"older inequality and  the Sobolev embedding of $\dot H^1 (\R^3)$ into $L^6 (\R^3)$, we obtain that 
\begin{align} \label{midi1}
\begin{split}
	&|2(D(\widecheck \psi_1), D(\widecheck \psi_2))_{L^2(\widecheck{\mathcal F})} - 2(D(\psi_1), D(\psi_2))_{L^2(\mathcal F)}|\\ 
	&\leq   \|\tilde \psi_1 -  \psi_1\|_{\dot H^1(\mathcal F)} \|\tilde \psi_2\|_{\dot H^1(\mathcal F)} +  \|\tilde \psi_1\|_{\dot H^1(\mathcal F)} \|\tilde \psi_2 -  \psi_2\|_{\dot H^1(\mathcal F)}\\
	&\quad + \left(\|g_1\|_{L^2(\mathcal F)} + C \|f_1\|_{L^{\frac 6 5 }(\mathcal F)}\right)  \|\tilde \psi_2\|_{\dot H^1 (\R^3)} ,
	\end{split}
\end{align}
 recalling that $g_1$ and $f_1$ are compactly supported in $ \mathcal J \subset \mathcal F$.
\ \par \

\noindent\emph{Step 6: Estimates of $\tilde \psi_i$ and of $\tilde \psi_i - \psi_i$, for $i=1,2$.}
First,  we decompose $\tilde \psi_i$, for $i=1,2$, into
 $$\tilde \psi_i  = w_i^I + w_i^B,$$ where $  w_i^I $ and $ w_i^B$ are the solutions to the following Stokes systems respectively  corresponding to the interior source term and to the boundary data in  \eqref{subtilde}:
\begin{align*}
	- \Delta w_i^I + \nabla p_i^I &= -\dv g_i + f_i \quad \text{ and } \quad \dv w_i^I = 0 \quad \text{in } \mathcal F, \\
	w_i^I &= 0 \quad \text{in } \cup_{j=1}^N \mathcal S_j,
\end{align*} 
and
\begin{align*}
	- \Delta w_i^{B} + \nabla p_i^{B} &= 0 \quad \text{ and } \quad \dv w_i^{B} = 0 \quad \text{in } \mathcal F, \\
	w_i^{B}(x) &=    \widecheck Q_1^T  \varphi_i(\widecheck Q_1   x + \widecheck h_j  )\quad \text{in }  \mathcal S_1, \\
	 w_i^{B}(x)  &= \varphi_i(x) \quad \text{in } \cup_{ 2 \leq j \leq N} \mathcal S_j .
\end{align*}
On the one hand, by a straightforward energy estimate, we have that
\begin{align*}
	\| w^I_i \|_{\dot H^1(\R^3)} \leq \|g_1\|_{L^2 (\mathcal F)} + C \|f_1\|_{L^{\frac 6 5 }(\mathcal F)}.
\end{align*}
On the other hand, by Theorem \ref{lem:slender.several} 
we have that 
\begin{align*}
	\| w_i^{B} \|_{\dot H^1 (\R^3)} \leq C |\log \eps|^{-1/2}\| \varphi_i\|_{W^{1,\infty}(\R^3)}  .
\end{align*}
Thus, we arrive at 
\begin{align}
\label{midi2}
	\|\tilde \psi_i\|_{\dot H^1(\R^3)} \leq C |\log \eps|^{-1/2}\| \varphi_i\|_{W^{1,\infty} (\R^3)} + \| g_i\|_{L^2(\mathcal F)}  + C \|f_i\|_{L^{\frac 6 5 }(\mathcal F)}.
\end{align}

Similarly,  for $i=1,2$, using \eqref{def:psi} and \eqref{subtilde}, we decompose $\tilde \psi_i -  \psi_i$ into 
$$\tilde \psi_i -  \psi_i = w_i^I + w_i^{B,\textrm{diff}} ,$$
with, this time, the boundary term $w_i^{B,\textrm{diff}}$ satisfying the following Stokes system:
\begin{align*}
	- \Delta w_i^{B,\textrm{diff}} + \nabla p_i^{B,\textrm{diff}} &= 0 \quad \text{ and } \quad\dv   w_i^{B,\textrm{diff}} = 0 \quad \text{in } \mathcal F, \\
	w_i^{B,\textrm{diff}}(x) &=    \widecheck Q_1^T  \varphi_i(\widecheck Q_1   x + \widecheck h_1  )  - \varphi_i(x) \quad \text{in }  \mathcal S_1, \\
	 w_i^{B,\textrm{diff}}(x)  &= 0 \quad \text{in } \cup_{ 2 \leq j \leq N} \mathcal S_j .
\end{align*}
By Theorem \ref{lem:slender.several} and \eqref{delta-c}, 
we have that 
\begin{align*}
	\|w_i^{B,\textrm{diff}}\|_{\dot H^1(\R^3)} &\leq C \delta |\log \eps|^{-1/2} \| \varphi_i\|_{W^{2,\infty} (\R^3)}.
\end{align*}
Thus 
we obtain that 
\begin{align}
\label{midi3}
	\|\tilde \psi_i -  \psi_i\|_{\dot H^1(\R^3)} \leq C \delta |\log \eps|^{-1/2}\| \varphi_i\|_{W^{2,\infty} (\R^3)} + C \|g_i\|_{L^2 (\mathcal F)} + C \|f_i\|_{L^{\frac 6 5 } (\mathcal F)}.
\end{align}

\ \par \

\noindent\emph{Step 7: Conclusion.} Gathering \eqref{midi1}, \eqref{midi2}, 
\eqref{midi3} and \eqref{midi4}, we arrive at \eqref{est:psi_1.psi_2}.
This finishes the proof of Proposition \ref{lem:shapeDerivatives}. 
\end{proof}

%%%%%%%%%%%%%%%%%%%%%%%%%%%%%%%%%%%%%%%%%%%%%%%%%%%%%%%%%%%%%%%

\section{Asymptotic behaviour of the filament centerlines}
\label{gathered-section}
This section is devoted to the proof of the part of Theorem \ref{main}  devoted to the asymptotic behaviour of the filament centerlines, that is to the proof of 
\eqref{est:main.positions}, 
\eqref{allezesti} and of 
  Theorem \ref{refinedqs}, which, as mentioned in the comments after Theorem \ref{main}, 
  provides a more precise approximation of the asymptotic behaviour of the filament velocities than the one in Theorem \ref{main}, at the expense of $\eps$-dependent positions. 
  This section is divided into three subsections. 
   
   First Subsection \ref{refor}  is devoted to reformulation of the Newton equations \eqref{intrauN} into 
 a system of second-order quasilinear ODEs on the $6N$ degrees of freedom of the rigid bodies, which does not involve the fluid pressure anymore and reveals the role played by the Stokes resistance matrices. 
 
 In Subsection   \ref{section-modu} we consider the time evolution of a modulated energy which measures, for each positive $\eps$, the difference between the filaments velocities for positive $\eps$ and the so-called   ``Fax\'en''  velocities, which are  given by the quasi-static balance of the Stokes resistance force and torque  
with the force and torque due to the background flow. The latter are a family 
of velocities which depend on the positions of the filaments velocities of $\eps$-thickness. 
  Unlike the total energy of the system  considered in \eqref{nrj}, this modulated energy has the advantage to circumvent the part of the energy corresponding to the motion  of the filaments, under the influence of the fluid.

Finally in Subsection \ref{section-cv-filament} we take advantage of the previous subsections to prove the part of 
Theorem \ref{main} which concerns the filaments. 
  
  %%%%%

  \subsection{Reformulation of the Newton equations}
\label{refor}

This subsection is devoted to reformulation of the Newton equations into a compressed form which does not involve the fluid pressure anymore, and reveals the role played by the Stokes resistance matrices.

We introduce first a few notations. Let us emphasize that all quantities here are defined with respect to the  filaments of $\eps$-thickness at time $t$.
Indeed the result below concerns 
the solutions
 $(h_{i,\eps} , Q_{i,\eps} )_{1 \leq i \leq N}$ to the  Newton-Stokes system  \eqref{N-S} up to the time $T^{\max}_\eps$ as given by Proposition \ref{CL}. 
\begin{itemize}
\item 
Let us first gather all the translation and rotation velocities corresponding to the motions of the $N$ filaments into the following vector of $\R^{6N}$: 
\begin{align} \label{rec2}
\textrm{Y}  :=
 \begin{pmatrix}
  \textrm{v}_i \\ \omega_i
    \end{pmatrix}_{1 \leq i \leq N}
= ( \textrm{Y}_{j,\beta })_{1 \leq j \leq N, \, 1 \leq \beta \leq 6}	 . 
\end{align}
\item 
Similarly let 
\begin{align} \label{beth}
	\textrm{Y}^\flat :=
		\begin{pmatrix}
		 \textrm{v}^\flat_i \\
		 \omega^\flat_i
		\end{pmatrix}_{1 \leq i \leq N} 
			\quad \text{
		 such that } \quad 
\textrm{Y}^\flat 	 = \mathcal K^{-1}\mathfrak{f}^\flat  ,
\end{align}
 recalling that  $\mathcal K$ is the $6N \times 6N$ matrix defined in  \eqref{resis} and that $\mathfrak{f}^\flat$ is defined by \eqref{clour}.
These are the  so-called ``Fax\'en'' velocities. Let us observe that 
 it follows from \eqref{eq:K.coercive} and
\eqref{ofint} that
\begin{align} \label{est:Y}
    |\textrm{Y}^\flat| \leq C
\end{align}
 with a constant depending only on $\dmin$ (see \eqref{dmin}) for all $\eps$ sufficiently small.
 We will also use the following notations, where on the one hand translation velocities are gathered, and on the other hand rotation velocities are gathered: 
\begin{align} \label{fax-defv}
\textrm{v}^\flat := 	(\textrm{v}^\flat_i)_{1 \leq i \leq N}  \quad \text{ and }  \quad  \omega^\flat := (\omega^\flat_i)_{1 \leq i \leq N}.
\end{align}
\item Let 
\begin{align}
\mathfrak{f}^a
	:= 
	\begin{pmatrix}
	F^a \\
	T^a
	\end{pmatrix} ,
\end{align}
where 
\begin{align*} 
 F^a:= 	(F^a_i)_{1 \leq i \leq N}  \quad \text{ and }  \quad 	T^a := (T^a_i)_{1 \leq i \leq N} ,
\end{align*}
with for $ 1 \leq i \leq N$, 
\begin{gather*}
 F^a_i  := \int_{ \partial \mathcal S_i } \Sigma(u^\flat , p^\flat ) n \, d\mathcal{H}^2    \quad \text{
		 and } \quad   T^a_i :=        \int_{ \partial   \mathcal S_i}  (x - h_{i,\eps}) \wedge \Sigma(u^\flat , p^\flat )     n \, d\mathcal{H}^2     .
\end{gather*}
The choice of the index ``$a$" is for Archimedes, because, as one proceeds in the usual computation of gravity buoyancy, see \cite[(4.18)]{Galdi2} or
\cite[p105]{serre-japan}, 
one may use integration by parts inside the filaments to arrive at 
\begin{gather}
\label{archim}
 F^a_i  
 =   \int_{ \mathcal S_i } (\Delta u^\flat + \nabla p^\flat )  \, dx      
  \quad \text{
		 and } \quad   T^a_i
   =   \int_{ \mathcal S_i } (x - h_{i,\eps}) \wedge (\Delta u^\flat + \nabla p^\flat )  \, dx     .
\end{gather}
\item Finally let us gather the inertia of the $N$ filaments into  the $6N \times 6N$ block diagonal matrix $\mathcal{M}$ whose $6 \times 6$ blocks are
\begin{align}
\label{domi}
\mathcal{M} := (m_i \mathrm{Id}_3 , \mathcal{J}_i )_{ 1 \leq i \leq N} . \end{align}
\end{itemize}
We can now state the main result of this subsection. 
 \begin{Proposition}\label{reformul}
 As long as the filaments are separated, 
 the Newton equations  \eqref{eq:v'}-\eqref{eq:omega'} are equivalent to the following compressed form: 
\begin{align}
\label{compressed}
	\eps^2 \frac{d}{dt} ( \mathcal{M} \Y)
	=-	\mathcal K ( \Y - \Y^{\flat} ) + \mathfrak{f}^a
 ,
\end{align}
\end{Proposition}
Let us  highlight that  $\mathcal{M}$, $\mathcal K$, $\textrm{Y}^\flat $ and $\mathfrak{f}^a$
 depend on the positions of the filaments, that is to 
 the solution $\textrm{Y}$  through its time antiderivative, so that the ODE \eqref{compressed} is quasilinear. 

\begin{proof}
Since the left hand sides of \eqref{eq:v'}-\eqref{eq:omega'} clearly correspond to 
the left hand side of \eqref{compressed} according to the definitions \eqref{rec2} and \eqref{domi}, it is sufficient to 
 consider the right hand sides of \eqref{eq:v'}-\eqref{eq:omega'}, which we decompose into
\begin{subequations}
 \label{NN}
\begin{gather}
 \int_{ \partial \mathcal S_i } \Sigma(u,p) n \, d\mathcal{H}^2  
 =  F_i^a +\int_{ \partial \mathcal S_i } \Sigma (u^\mathfrak{p},   p^\mathfrak{p} ) n \, d\mathcal{H}^2 
  ,  \\
   \int_{ \partial   \mathcal S_i}   (x - h_{i,\eps}) \wedge   \Sigma(u,p) n \, d\mathcal{H}^2  
   = T_i^a    +
     \int_{ \partial   \mathcal S_i}   (x - h_{i,\eps}) \wedge  \Sigma (u^\mathfrak{p},   p^\mathfrak{p} )  n \, d\mathcal{H}^2    .
\end{gather}
\end{subequations}
The second terms in the right hand sides of  \eqref{NN} can be computed as follows. 
For $1 \leq i \leq N$, by  \eqref{sti3}, 
\begin{align*}
\begin{pmatrix}
 \int_{ \partial \mathcal S_i } \Sigma(u^\mathfrak{p} , p^\mathfrak{p} ) n \, d\mathcal{H}^2  \\
  \int_{ \partial   \mathcal S_i}  (x - h_{i,\eps}) \wedge  \Sigma(u^\mathfrak{p} , p^\mathfrak{p} )   n  \, d\mathcal{H}^2 
  \end{pmatrix}
  &= 
\Big(
\int_{ \cup_{j=1}^N \partial  \mathcal   S_j  }
\Sigma(u^\mathfrak{p} , p^\mathfrak{p} ) n \cdot V_{i,\alpha} \, d\mathcal{H}^2 
\Big)_{1 \leq \alpha \leq 6}
\\ &= 
\Big(
\int_{ \cup_{j=1}^N \partial  \mathcal   S_j  } \Sigma(V_{i,\alpha}  , P_{i,\alpha} ) n \cdot u^\mathfrak{p} \, d\mathcal{H}^2 
\Big)_{1 \leq \alpha \leq 6} ,
\end{align*}
by Lorentz's reciprocity theorem, using that $u^\mathfrak{p}$ and $V_{i,\alpha}$ are both solutions of the steady Stokes system in $ \mathcal F$, see \eqref{ssys1} and \eqref{sti}.
Then, using the boundary condition \eqref{bc-u}-\eqref{eq:v^S} and the  definition \eqref{t1.6} of $v_{j,\beta}$, we have that in $ \mathcal{S}_j$, 
\begin{equation}  \label{rec1}
u^\mathfrak{p}  = \biggl(\sum_{1\leq \beta \leq 6} \textrm{Y}_{j,\beta }\,  v_{j,\beta}\biggr) -  u^\flat .
\end{equation}
We deduce that, for $1 \leq \alpha \leq 6$, 
\begin{align*}
\int_{ \cup_{j=1}^N \partial  \mathcal   S_j  }
\Sigma(u^\mathfrak{p} , p^\mathfrak{p} ) n \cdot V_{i,\alpha} \, d\mathcal{H}^2 
 &= 
\sum_{1 \leq j \leq N} \sum_{1\leq \beta \leq 6} \textrm{Y}_{j,\beta } \int_{  \partial  \mathcal   S_j  }
\Sigma(V_{i,\alpha}  , P_{i,\alpha} ) n \cdot v_{j,\beta} \, d\mathcal{H}^2 
\\ &\quad-
\int_{ \cup_{j=1}^N \partial  \mathcal   S_j  } \Sigma(V_{i,\alpha}  , P_{i,\alpha} ) n \cdot u^\flat \, d\mathcal{H}^2   .
\end{align*}
Thus, 
\begin{align}
\label{reci}
\biggl(\begin{pmatrix}
 \int_{ \partial \mathcal S_i } \Sigma(u^\mathfrak{p} , p^\mathfrak{p} ) n \, d\mathcal{H}^2  \\
  \int_{ \partial   \mathcal S_i}  (x - h_{i,\eps}) \wedge   \Sigma(u^\mathfrak{p} , p^\mathfrak{p} )  n  \, d\mathcal{H}^2 
  \end{pmatrix}\biggr)_{1 \leq i \leq N}
  = 
 - \mathcal K \textrm{Y}
    + \mathfrak{f}^\flat .
  \end{align}
Thus combining \eqref{reci}, \eqref{beth} and  \eqref{NN}
 we find \eqref{compressed}. 
\end{proof}

One difficulty associated with the equation \eqref{compressed}
is  the  factor $\eps^2$ in front of the left hand side which makes the asymptotic analysis of this ordinary differential system belong to the class of singular perturbations, i.e. degeneracy at the main order. 
However the matrix $\mathcal K$ is positive definite  symmetric which guarantees that the effect of the associated term is to damp the velocities when time proceeds, or more exactly that they relax to the Fax\'en velocities. 
Indeed to tackle the asymptotic behaviour of the solutions to \eqref{compressed}
 one key point is the behaviour of Stokes' resistance matrix $\mathcal K$ with respect to $\eps$, that is to quantify the damping effect in the limit of zero thickness.

  %%%%

%%%%%%%%%%%%%%%%%%%%%%%%%%%%%%%%%%%%

\subsection{Modulated  energy and lifetime}
\label{section-modu}

To estimate the relaxation of the exact solution $Y$ of \eqref{compressed}
to the time-dependent vector $\textrm{Y}^\flat$ 
 for small $\eps$, we consider the modulated energy:
\begin{align}
 E := \frac 1 2 ( \textrm{Y} - \textrm{Y}^\flat ) \cdot \mathcal{M}  (\textrm{Y} - \textrm{Y}^\flat ) .
\end{align}
Thanks to the assumptions on the inertia of the filaments in 
Section
\ref{sec-setting}, the matrix $\mathcal{M}$ defined in \eqref{domi} is symmetric positive definite uniformly in $\eps$, and it is also 
uniformly bounded. Thus the modulated energy $E$ is $\eps$-uniformly equivalent to $\vert \textrm{Y} - \textrm{Y}^\flat \vert^2$.

As mentioned in Proposition \ref{CL}, for each $\eps$ there is a positive time interval during which the filaments remain separated. 
Below we will perform some computations
which are valid until the time 
 that the filaments remain well separated uniformly with respect to $\eps$.  By a bootstrap argument in Section \ref{section-cv-filament} we then derive uniform estimates of this time with respect to  $\eps$ and show that it extends until $\hat T$ in the sense of Theorem \ref{main}.
 
More precisely, for $d > 0$, we define
\begin{align}\label{def.Tepsd}
    T_{\eps,d} := \inf \left\{ t \geq 0 : \dmin(t) > d, Z \leq \frac {c_d}{C_d \eps^2 |\log \eps|} \right\},
\end{align}
where $Z:=\sqrt E$ and $\dmin$ is the minimal distance between the centerlines as defined in \eqref{dmin}.
Since $\min_{i \neq j} \dist(\mathcal S_i,\mathcal S_j) \geq \dmin - C \eps$, Proposition \ref{CL} 
implies that for $\eps \leq \eps_0(d)$ we have $T_{\eps,d} < T_\eps^{\max}$ and thus that the dynamics is well-posed on $(0,T_{\eps,d})$.  
Note that $Z = \sqrt E \leq C | \textrm{Y} - \textrm{Y}^\flat  |$.
Thus,  since $Z$ is continuous, decreasing the value of $\eps_0(d)$ if necessary, we have for all $\eps\in(0,\eps_0)$ that $T_{\eps,d} > 0$. 
In the following we consider only $t < T_{\eps,d}$.

 \begin{Proposition}\label{p}
For all $d>0$, there exists $C(d)>0$ and $\eps_0(d)>0$ independent of $\eps$ such that  for all $\eps\in(0,\eps_0)$,
 for all $t \in (0,T_{\eps,d})$,  
\begin{align} \label{est:Z}
 | \Y	(t) - {\Y^\flat} 	(t)|  \leq | \Y (0) - {\Y^\flat} (0)|  e^{-  \frac{c_d t}{\eps^2 |\log \eps|}} + C_d \eps^2 |\log \eps|.
\end{align}
\end{Proposition}
To prove Proposition \ref{p} we will use the following lemma. 
 \begin{Lemma}\label{lemma-drift}
For all $d>0$, there exists $C(d)>0$ and $\eps_0(d)>0$ independent of $\eps$ such that if the minimal distance $d_{min}$ between the centerlines satisfies $d_{min}\geq d$ and $\eps\in(0,\eps_0)$, we have the following estimate:
\begin{align} \label{est:s..2}
|	(\Y^\flat)'| &\leq C(1 + |\Y|)  .
\end{align}
\end{Lemma}
\begin{proof}[Proof of Lemma \ref{lemma-drift}]
First, by \eqref{beth}, 
\begin{align}
(\textrm{Y}^\flat)'
 &=  - \mathcal K^{-1} \mathcal K' \mathcal K^{-1} \, \mathfrak{f}^\flat 	 + \mathcal K^{-1} \,  (\mathfrak{f}^\flat )' .
\end{align}
Let $\bar u^\flat$  the solution to 
\begin{align*}
		- \Delta \bar u^\flat + \nabla \bar p^\flat &= 0  \quad \text{ and }\quad \div {\bar u^\flat} = 0 \quad \text{in } \mathcal F \\
		\bar u^\flat &= u^\flat \quad \text{ in }  \cup_{1 \leq i \leq N} \mathcal S_i,
\end{align*}
and let 
$$\bv := (V_{i,\alpha})_{1\leq i \leq N, 1 \leq \alpha \leq 6},$$
where  we recall that the $ V_{i,\alpha}$  are the  unique solutions to the steady Stokes equations associated with the rigid velocities in the filaments, see  \eqref{sti}. 
Recalling the definition of $ \mathfrak{f}^\flat $ in  \eqref{clour}-\eqref{fab3}-\eqref{fab1}-\eqref{fab2} 
we obtain by an integration by parts that 
\begin{equation*} 
 \mathfrak{f}^\flat  =  \Big( (D (V_{i,\alpha} ),D (\bar u^\flat))_{L^2(\mathcal F)} \Big)_{1\leq i \leq N, 1 \leq \alpha \leq 6} ,
\end{equation*}
where we recall that  the notation $\big(\cdot , \cdot\big)_{L^2(\mathcal F)}$ stands for the inner product in $L^2(\mathcal F)$.
By \eqref{eq:K.coercive} and \eqref{clour}, 
we deduce that 
\begin{align} 
\begin{aligned}
	|(\textrm{Y}^\flat)'| &\leq C |\log \eps|^2 |\mathcal K'| \| \bar u^\flat\|_{\dot H^1(\mathcal F)} \|\bv\|_{\dot H^1(\mathcal F)}
	+ C |\log \eps|  |(D (\bv), D (\bar u^\flat))'_{L^2(\mathcal F)}|.
\end{aligned}
\end{align}
Moreover, by \eqref{resis} and an integration by parts, we have that 
\begin{equation*} 
 \mathcal K =  \Big( (D (V_{i,\alpha} ),D (V_{j,\beta} ))_{L^2(\mathcal F)} \Big)_{1\leq i,j \leq N, 1 \leq \alpha,\beta \leq 6} .
\end{equation*}
Thus, with the convention that the terms containing $\bv$ are the sum for $1\leq i \leq N$ and $ 1 \leq \alpha \leq 6$, of  corresponding terms  for $V_{i,\alpha}$, 
and by Theorem \ref{lem:slender.several}, we have  
$$\|\bv \|_{\dot H^1(\mathcal F)}
+ \| \bar u^\flat \|_{\dot H^1(\mathcal F)}
\leq C |\log \eps|^{-\frac12},$$
so that we arrive at 
\begin{align} \label{est:s..}
	|(\textrm{Y}^\flat)'| &\leq C |\log \eps| \left(   {\big|\big(D (\bv),  
D (\bv)\big)_{L^2(\mathcal F)}'\big|}
+  |\big(D( \bv), D (\bar u^\flat)\big)'_{L^2(\mathcal F)}|
\right).
\end{align}

We decompose the time derivative in the terms on the right hand side of \eqref{est:s..} into several contributions.
To this end, we introduce the operator 
$$G \colon  W^{2,\infty}_\sigma(\R^3) \times W^{2,\infty}_\sigma(\R^3) \to \R,$$
defined by 
\begin{align*}
    G(\varphi_1,\varphi_2) :=  (D (\psi_1),D (\psi_2))_{L^2(\mathcal F)},
\end{align*}
	where  $\psi_k$, for $k=1,2$, is the solution in 	$ \dot H^1(\mathcal F)$ to
the problem
\begin{align*} 
	\begin{aligned}
		- \Delta \psi_k + \nabla p_k &= 0 , \quad \div \psi_k = 0 \quad \text{in } \mathcal F , \\
		\psi_k &= \varphi_i  \quad \text{in } \cup_i \mathcal S_i . 
	\end{aligned}
\end{align*}
Then, for any $ 1\leq i \leq N$, for any $ 1 \leq \alpha \leq 6$,
$$(D (V_{i,\alpha}), D( \bar u^\flat))_{L^2(\mathcal F)} = G(v_{i,\alpha},u^\flat).$$
Note that the operator $G$ as well as $v_{i,\alpha}$ implicitly depends on the positions and orientations of the particles. Consequently,
\begin{align*}
    (D (V_{i,\alpha}), D( \bar u^\flat))'_{L^2(\mathcal F)} = \textrm{Y} \cdot \nabla_{h,Q} G(v_{i,\alpha},u^\flat) + G(\textrm{Y} \cdot \nabla_{h,Q} v_{i,\alpha}, u^\flat) + G(v_{i,\alpha}, \partial_t u^\flat).
\end{align*}

Theorem \ref{lem:slender.several} implies 
\begin{align*}
    |G(\varphi_1,\varphi_2)| \leq C |\log \eps|^{-1} \|\varphi_1\|_{W^{1,\infty} } \|\varphi_2\|_{W^{1,\infty}}.
\end{align*}
Thus, taking into account the regularity assumptions on the background flow $u^\flat$,  see \eqref{regover}, we arrive at
\begin{align*}
     |(D (V_{i,\alpha}), D (\bar u^\flat))'_{L^2(\mathcal F)}| \leq  |\textrm{Y}| |\nabla_{h,Q} G(v_{i,\alpha},u^\flat)| +  C |\log \eps|^{-1} (1+ |\textrm{Y}|).
\end{align*}
Analogous considerations hold for the term ${\big|\big(D (\bv),  
D (\bv)\big)_{L^2(\mathcal F)}'\big|}$. Proposition \ref{lem:shapeDerivatives} applied  to both
$(V_{i,\alpha},\bar u^\flat)$ and to $(V_{i,\alpha},V_{j,\beta})$ leads to the result. 
\end{proof}

With the result of  Lemma \ref{lemma-drift} in hands we can now start the proof of Proposition \ref{p}.
%%%%
\begin{proof}[Proof of  Proposition \ref{p}]
We first recast \eqref{compressed} as 
\begin{align*}
	\eps^2  ( \mathcal{M} (\textrm{Y}- \textrm{Y}^\flat ))'
	= 
-	\mathcal K ( \textrm{Y} - \textrm{Y}^\flat ) + \mathfrak{f}^a
- \eps^2  ( \mathcal{M}  \textrm{Y}^\flat )' ,
\end{align*}
and then take the inner product with $\textrm{Y}- \textrm{Y}^\flat $, with the observation that 
\begin{align*}
   ( \textrm{Y} - \textrm{Y}^\flat ) \cdot \big( \mathcal{M}  (\textrm{Y} - \textrm{Y}^\flat )  \big)' =  E'
 +  \frac 1 2 ( \textrm{Y} - \textrm{Y}^\flat ) \cdot \mathcal{M}'  (\textrm{Y} - \textrm{Y}^\flat ) ,
\end{align*}
so that 
\begin{align*}
	 E' 
	&= 
		-{\eps^{-2}\, } 
	(\textrm{Y} - \textrm{Y}^\flat )
	\cdot \mathcal K (\textrm{Y} - \textrm{Y}^\flat  )
	+ {\eps^{-2}\, } 
	(\textrm{Y} - \textrm{Y}^\flat  )
	 \cdot  \mathfrak{f}^a 
\\ 
	&\quad-  ( \mathcal{M}  \textrm{Y}^\flat )' \cdot  (\textrm{Y} - \textrm{Y}^\flat ) 
-  \frac 1 2 ( \textrm{Y} - \textrm{Y}^\flat ) \cdot \mathcal{M}'  (\textrm{Y} - \textrm{Y}^\flat )  .
\end{align*}
Recalling the definition of $\mathcal{M}$ in \eqref{domi}, 
we arrive at 
 the following formula for the time derivative $E'$ of the  modulated energy: 
\begin{align}
\label{timevol}
\begin{split}
E' &= 	-{\eps^{-2}\, } 
	(\textrm{Y} - \textrm{Y}^\flat )
	\cdot \mathcal K (\textrm{Y} - \textrm{Y}^\flat  )
	+ {\eps^{-2}\, } 
	(\textrm{Y} - \textrm{Y}^\flat  )
	 \cdot  \mathfrak{f}^a 
	\\ 
	&\quad- (\textrm{Y} - \textrm{Y}^\flat  ) \cdot
\mathcal{M} (\textrm{Y}^\flat)' -  (\omega - \omega^\flat) \cdot \mathcal J' \omega^\flat 
	- \frac 1 2 (\omega - \omega^\flat) \cdot \mathcal J' (\omega - \omega^\flat) 	,
\end{split}
\end{align}
where $\mathcal{J}$ is the $3N \times 3N$ block diagonal matrix whose $3 \times 3$ blocks are
 $ \mathcal{J}_i $ for $1 \leq i \leq N$.

By Corollary \ref{prop:asymptotic.resistance}, 
the first term of the right hand side of \eqref{timevol} can be bounded by $ - c {\eps^{-2}\, } |\log \eps|^{-1}\, E $ for some constant $c$ which is positive and uniform with respect to $\eps$, and will possibly change from line to line, while still satisfying  these properties.

By 
\eqref{archim}, since the background flow is assumed to be smooth, the term $\mathfrak{f}^a $ can be bounded by ${\eps^{2} }$. Therefore the second term of the right hand side  of \eqref{timevol} can be bounded by 
$C  \sqrt E $, where the constant $C$  is also positive and uniform with respect to $\eps$, and will also possibly change from line to line, while still satisfying  these properties.

Similarly the last three terms of the right hand side of \eqref{timevol} can be respectively bounded by 
 $C \sqrt E  |	(\textrm{Y}^\flat)' |$, $C \sqrt E |\mathcal J'|$ and  $ C E |\mathcal J'|$.

Thus 
\begin{align} \label{est:E}
	E'  + c {\eps^{-2}} {|\log \eps|^{-1}\,} E \leq C \sqrt E \left(1 + |	(\textrm{Y}^\flat)' | + |\mathcal J'|
	\right) + C  E  |\mathcal J'| .	
\end{align}
Regarding $\mathcal J'$,  we use \eqref{eq:J(t)} and $|\mathcal J_{0,i}| \leq C$ to deduce
\begin{align}
	|\mathcal J'| \leq C |Q'| \leq C | \omega| \leq C	\left| \textrm{Y}  \right| .
\end{align}
Combining this with the estimate for  $(\textrm{Y}^\flat)'$ from Lemma \ref{lemma-drift}  in the energy estimate \eqref{est:E} yields
\begin{align}
	E' \leq - \frac{c}{\eps^2 |\log \eps|} E + C \sqrt E (1 + \left| 
	\textrm{Y} 
	\right|)  + C E \left| \textrm{Y}  \right| .
\end{align}
Since 
$	\left| \textrm{Y}  \right| \leq  C (\sqrt E + 1)$, 
and using the uniform bound on $\textrm{Y}^\flat$ from \eqref{est:Y},
we arrive at 
\begin{align}
	E' \leq - \frac{c}{\eps^2 |\log \eps|} E + C \sqrt E (1 + E)
\end{align}
and for $Z = \sqrt E$
\begin{align}
	Z' \leq - \frac{c}{\eps^2 |\log \eps|} Z + C (1 + Z^2).
\end{align}
Recall that the constants depend on the minimal distance between the particles $\dmin(t)$ (see \eqref{dmin}).
More precisely, if $\dmin(t) \geq d$, then
\begin{align}
	Z' \leq - \frac{c_d}{\eps^2 |\log \eps|} Z + C_d (1 + Z^2),
\end{align}
for all $\eps \leq \eps_0(d)$.

By definition of $T_{\eps,d}$ in \eqref{def.Tepsd}, we find that on $(0,T_{\eps,d})$
\begin{align}
	Z' \leq - \frac{c_d}{2 \eps^2 |\log \eps|} Z + C_d.
\end{align}
By Gronwall's inequality, and recalling that the modulated energy $E$ is $\eps$-uniformly equivalent to $\vert \textrm{Y} - \textrm{Y}^\flat \vert^2$, 
we obtain that \eqref{est:Z} holds 
for all $t \in (0,T_{\eps,d})$, up to an adaptation of the constants $c_d$ and $ C_d$. 
\end{proof}

Below, in Section \ref{section-cv-filament}, we will prove the following result on the asymptotic behaviour of $ T_{\eps,d}$ as $\eps$ converges to $0$. 

To this end, let us recall
the definition  $\hat d_{\min} := \inf_{i \neq j } \dist( \hat {\mathcal C }_i,  \hat {\mathcal C }_j)$ from \eqref{def:d.min.hat}
and that $\hat T$ from \eqref{def:hat.T} is the maximal time for which $\hat d_{\min}$ stays positive.

\begin{Proposition}\label{prop-lifetime}
There is $\eps_0 >0 $ small enough which depends only on the reference filaments, $u^\flat$ and $\min_{t \leq T} \hat d_{\min}(t)$, such that {for all $T $ in $(0, \hat T)$,} for $d = \frac 1 4 \min_{t \leq T} \hat d_{\min}(t)$  and for all $\eps < \eps_0$
\begin{align}
    T_{\eps,d} > T.
\end{align}
\end{Proposition}

Let us already observe that combining 
 Proposition \ref{prop-lifetime}
and Proposition \ref{p} we obtain the following result. 

\begin{thm}\label{refinedqs}
Under the same assumptions as in Theorem 
\ref{main}
we have on the one hand the estimate 
\eqref{temps} on the lifetime 
and on the other hand, for all $T < \hat T$ there exists $C$  depending only on $u^\flat$, the reference filaments $\bar {\mathcal S}_i$,  $\inf_{t \in [0,T]} \hat d_{\min}(t)$ and the initial velocities, and there exists $\eps_0 > 0$ depending in addition on $T$ 
 such that for all $\eps $ in $(0, \eps_0)$ and all $t$ in $[0,T]$
	the difference between the solution 
	$(h' , \omega)$ to \eqref{N-S} and the  ``Fax\'en's'' velocities $(v^\flat,\omega^\flat)$ defined in \eqref{beth} satisfies 
\begin{align} \label{eq:asymptotic.velocities}
 \left| (h', \omega)(t)  - (v^\flat , \omega^\flat)(t) \right| 	 \leq   \left| (h'  , \omega) (0) - ( v^\flat ,\omega^\flat) (0) \right| \, e^{-  \frac{C t}{\eps^2 |\log \eps|}} + C \eps^2 |\log \eps|.
\end{align}
	\end{thm}
	Indeed, as we emphasized in Section \ref{sec-cv}, this theorem provides a more precise approximation than Theorem \ref{main}, at the expense of $\eps$-dependent positions and implicit forces.

%%%%%%%%
\subsection{Proof of the part of Theorem \ref{main} which concerns the filaments}
\label{section-cv-filament}

We now turn to the proofs of \eqref{est:main.positions} and of 
\eqref{allezesti}. 
   In particular we are going to prove the convergence of the filament positions given by the time dependent vector 
  \textrm{Y}  defined in 
  \eqref{rec2} to  the limit dynamics for which we use the notation 
\begin{align}\label{def.hatY}
    \hat {\textrm Y}(t) := \hat {\mathcal K}^{-1}(\hat h(t) , \hat Q(t))  \hat {\mathfrak{f}}^\flat(t, \hat h(t) , \hat Q(t)) ,
\end{align}
where   
$\hat {\mathcal K}$ is 
  the $6N \times 6N$ matrix whose $6 \times 6$  diagonal blocks are the $\hat{\mathcal K}_{i,i}$, for $1 \leq i \leq N$, are defined in 
  \eqref{hatKblock}, and proved to be invertible in Lemma \ref{hat-inverse}, and $\hat {\mathfrak{f}}^\flat$ is the  vector  in $\R^{6N}$ which gathers the vectors $\hat{\mathfrak{f}}_i^\flat$, 
   for $1 \leq i \leq N$, defined in 
  \eqref{def:f.hat.flat.frac}.
  It follows from  \eqref{debaze} that 

  \begin{gather} \label{debaze-hat}
\hat {\textrm Y}(t) :=   (\hat{\textrm{v}}_i(t) ,{\hat \omega}_i (t))_{1 \leq i \leq N}  . 
      \end{gather}

From the estimates in the previous subsection
 we already know that the velocities $\textrm Y$ and $\textrm Y^\flat$ are close as long as the filaments are well separated. 
 We now introduce 
\begin{align}\label{def.tildeY}
    \tilde { \textrm Y}(t) := \hat {\mathcal K}^{-1}( h_\eps(t) ,Q_\eps(t)) \hat {\mathfrak{f}}^\flat(t, h_\eps(t) ,Q_\eps(t) )  .
\end{align}
The velocities $\tilde {\textrm Y}$ correspond to the limit dynamics but with the positions of the filaments given by the $\eps$-dynamics rather than the limit dynamics. In this sense, $\tilde {\textrm Y}$ can be seen as intermediate between $\textrm Y^\flat$ and $\hat {\textrm Y}$.
Next lemma takes benefit from the previous estimates of $\textrm Y - \textrm Y^\flat$ to establish 
 some estimates of  $\tilde {\textrm Y} - {\textrm Y^\flat}$  as long as the filaments are well separated. 
\begin{Lemma}
\label{itscor}
	For all $d > 0$ there exists a constant $C(d) > 0$ and $\eps_0(d) > 0$ such that for all $\eps \in(0,\eps_0)$ and $\dmin \geq d$,
	\begin{align}
		  |\tilde {\Y} - {\Y^\flat}| &\leq C |\log \eps|^{-1/2} \|u^\flat\|_{W^{1,\infty}(\R^3)}. \label{eq:Faxen.velocites}
	\end{align}
\end{Lemma}
\begin{proof}
Recalling the definition of ${\textrm Y^\flat}$ in 
\eqref{beth} and the one of $\tilde {\textrm Y}$ above, we observe that 
		\begin{align*}
	      {\textrm Y^\flat} - \tilde {\textrm Y} &=
	   \mathcal K^{-1} (\mathfrak{f}^\flat -| \log \eps |^{-1} \hat {\mathfrak{f}}) -
	   \mathcal K^{-1} (\mathcal K - | \log \eps |^{-1}  \mathcal {\hat K})  \mathcal {\hat K}^{-1} \hat {\mathfrak{f}}	   ,
	   	   	\end{align*}
	 where $\hat K$ and $\hat f$ should be understood as being  evaluated at $(h_{\eps},Q_\eps)$. 
		Combining \eqref{eq:approx.K}, \eqref{eq:K.coercive}, \eqref{ofint} and observing that 
		$
	   | \mathcal {\hat K}^{-1} | + |\hat {\mathfrak{f}} | 
	    \leq C $, 
		we conclude the   proof of Lemma \ref{itscor}. 
	\end{proof}
We now turn to the proof that \eqref{est:main.positions} and 
 \eqref{allezesti} holds on $[0,T_{\eps,d}]$, where $T_{\eps,d}$ is defined by \eqref{def.Tepsd},
 that is in particular to to the 
estimate of $ \textrm Y - \hat {\textrm Y}$. 
 \begin{Proposition}\label{pbis}
For all $d>0$, there exists $C(d)>0$ and $\eps_0(d)>0$ independent of $\eps$ such that 
for all $\eps\in(0,\eps_0)$, the estimates 
\eqref{est:main.positions} and 
 \eqref{allezesti} hold on $[0,T_{\eps,d}]$. 
\end{Proposition}

	\begin{proof}
First we recall that the coefficients of \eqref{fav} are smooth and globally Lipschitz, so that recalling \eqref{def.hatY} and \eqref{debaze-hat}, we infer that 
\begin{align} \label{est:Y.hat.Y.tilde}
    |\tilde {\textrm Y} - \hat {\textrm Y}| \leq C |(h_\eps,Q_\eps) -  (\hat h, \hat Q)|.
\end{align}
Recalling  \eqref{eq:initial.positions} and \eqref{eq:center.mass.small}, we also have that 
 the initial data $(h(0),Q(0))$ and $(\hat h(0),\hat Q(0))$ satisfy
\begin{align}
    |(h(0),Q(0)) - (\hat h(0),\hat Q(0))| \leq C \eps ,
\end{align}
Thus, using \eqref{e.hi}, \eqref{e.Qi} and 
\eqref{debaze}, by a combination of \eqref{est:Z}, \eqref{est:Y.hat.Y.tilde} and Lemma \ref{itscor}
 we obtain that  for all $t \leq T_{\eps,d}$, 
\begin{align*}
    &| (h_\eps(t),Q_\eps(t)) - (\hat h(t),\hat Q(t)) |\\
    &\leq C \eps +  \int_0^t |{\textrm Y} - {\textrm Y^\flat}| + |{\textrm Y^\flat} - \tilde {\textrm Y}| + |\tilde {\textrm Y} - \hat {\textrm Y}| \dd s\\
    & \leq C \eps + \int_0^t \big( |{\textrm Y}(0) - {\textrm Y^\flat}(0)| e^{-  \frac{c_d s}{\eps^2 |\log \eps|}} + C_d \eps^2 |\log \eps| + C_d |\log \eps|^{-1/2} \big) \dd s\\
    &  + C \int_0^t  |(h_\eps(s),Q_\eps(s)) -  (\hat h(s), \hat Q(s))| \dd s  \\
    & \leq  C \eps 
    +   C_d |\log \eps|^{-1/2} t  + C \int_0^t  |(h_\eps(s),Q_\eps(s)) -  (\hat h(s), \hat Q(s))| \dd s.
\end{align*}
Here,  we used the bound  \eqref{est:Y}  on $ {\textrm Y^\flat}$ and   the constant in the last line depends on the initial velocities $\textrm Y(0)$.
By Gronwall's estimate, we deduce that  \eqref{est:main.positions}  holds for all $t \leq T_{\eps,d}$.
 We also note that, bookkeeping the computations above, this allows to prove the following bound on the velocities: 
\begin{align} \label{est:main.velocities}
|{\textrm Y} - \hat {\textrm Y}| \leq C_d e^{- \frac{c_dt}{\eps^2 |\log\eps|}}{|{\textrm Y}(0) - {\textrm Y^\flat}(0)|} + C_d |\log \eps|^{-1/2} + C_d |\log \eps|^{-1/2} t e^{Ct}. 
\end{align}
{Thus, estimating 
\begin{align}
|{\textrm Y}(0) - {\textrm Y^\flat}(0)| \leq |{\textrm Y}(0) - \hat {\textrm Y}(0)| + |\hat {\textrm Y}(0) - \tilde {\textrm Y}(0)| + |\tilde {\textrm Y}(0) - {\textrm Y^\flat}(0)|, 
\end{align}
and applying again  \eqref{est:Y.hat.Y.tilde} and Lemma \ref{itscor} yields  \eqref{allezesti} on $[0,T_{\eps,d}]$.}
	\end{proof}

We now turn to the proof of Proposition \ref{prop-lifetime}. 

\begin{proof}[Proof of Proposition \ref{prop-lifetime}] 
Let  $T < \hat T$ 
and $d = \frac 1 4 \min_{t \leq T} \hat d_{\min}(t)$.
We first observe that \eqref{est:Z} implies $T_{\eps,d} = \infty$ or $\dmin(T_{\eps,d} ) = d$.
Moreover, by \eqref{est:main.positions},
 we have  on $(0,T_{\eps,d})$, 
\begin{align} \label{est:d_min}
\dmin(t) \geq  \hat d_{\min}(t) - C \eps -  C_d(\eps +|\log \eps|^{- \frac 1 2} t) e^{C t},
\end{align} 
where the term $C \eps$ accounts for the filaments' thickness.
 Thus, the choice $ d = \frac 1 4 \min_{t \leq T} \hat d_{\min}(t)$ implies for $\eps$ sufficiently small (depending on $d$, $T$, $u^\flat$ and the reference filaments), $\dmin(t) \geq 2 d$ for all
$t \leq \min\{T_{\eps,d}, T\}$. Since $T_{\eps,d} = \infty$ or $\dmin(T_{\eps,d} ) = d$ this implies $T_{\eps,d} > T$. This concludes the proof of \eqref{est:main.positions} and \eqref{allezesti} on $[0,T]$. 
This completes the proof of 
Proposition \ref{prop-lifetime}. 
\end{proof}

\section{Asymptotic behaviour of the fluid}
\label{sectiondassaut}
This section is devoted to the proof of the part of Theorem \ref{main}  devoted to the asymptotic behaviour of the fluid, that is to the proof of \eqref{conv-flui-precised}, together with the proof of \eqref{conv-flui-precised.comment}. 
To this aim we first decompose   
$u^\mathfrak{p} $ 
into 
\begin{equation}
\label{trf}
u^\mathfrak{p} =  \sum_{1 \leq i \leq N} u^\mathfrak{p} _i
 \quad 
 \text{ and }   \quad
 p^\mathfrak{p} =  \sum_{1 \leq i \leq N} p^\mathfrak{p} _i,
 \end{equation}
 where   
\begin{subequations}
\label{fintrau-tot}
\begin{gather}
\label{fintrau1}
\displaystyle  - \Delta u^\mathfrak{p}_i  + \nabla p^\mathfrak{p}_i =0 \quad  \text{ and } \quad \operatorname{div} u^\mathfrak{p}_i   = 0  \quad \text{in }\mathcal{F}(t), \\
\label{fintrau1bis}
u^\mathfrak{p}_i = v^{{\mathcal S}_i}  - u^\flat \quad \text{for}\ \  x\in  \mathcal{S}_i  (t) , \\
\label{fintrau1bisbis}
u^\mathfrak{p}_i = 0 \quad \text{for}\ \  x\in  \mathcal{S}_j (t) , \quad \text{ for} ~ j \neq i .
\end{gather}
\end{subequations}
Then, we apply Theorem \ref{lem:slender.several} to $ u^\mathfrak{p}_i $, for each $i$,  recalling that in $ \mathcal{S}_i$, 
\begin{equation*} 
 v^{{\mathcal S}_i} =  \sum_{1\leq \beta \leq 6} \textrm{Y}_{i,\beta }\,  v_{i,\beta}   , 
\end{equation*}
to  obtain the following proposition.

\begin{Proposition} \label{pro:u^p.U}
Let $d>0$. Then there exists $\eps_0(d) > 0$ and $C(d) < \infty$ such that for all $\eps \in (0,\eps_0)$ and all $t \in [0,T_\eps^{\max}]$ 
with $d_{\min}(t) \geq d$

\begin{align} \label{est:fluid.velocities.-1} 
\|u^\mathfrak{p}(t,\cdot) - |\log \eps|^{-1} \sum_{1 \leq i \leq N} U_{\mathcal C_i(t)}[ v^{{\mathcal S}_i(t)}  - u^\flat(t,\cdot)  ]\|_{\dot H^{1} (\R^3 \setminus \cup_i {\mathcal S}_i(t)) }
\\ \nonumber \quad \leq C | \log \eps |^{-1} \big( |\Y(t)|_{} + \|u^\flat(t,\cdot)\|_{W^{1,\infty}} \big).
\end{align}
Moreover, for  $1 \leq q < 3/2$,
\begin{align} \label{est:fluid.velocities.0} 
\|u^\mathfrak{p}(t,\cdot) - |\log \eps|^{-1} \sum_{1 \leq i \leq N} U_{\mathcal C_i(t)}[ v^{{\mathcal S}_i(t)}  - u^\flat(t,\cdot)  ]\|_{W^{1,q}_{\loc}} \leq C | \log \eps |^{-3/2} \big( |\Y(t)|_{} + \|u^\flat(t,\cdot)\|_{W^{1,\infty}} \big),
\end{align}
Furthermore, for $1 \leq p< 3$, 
\begin{align} \label{est:fluid.velocities.1} 
\|u^\mathfrak{p} - |\log \eps|^{-1} \sum_{1 \leq i \leq N} U_{\mathcal C_i}[ v^{{\mathcal S}_i}  - u^\flat  ]\|_{L^p_{\loc}} \leq C | \log \eps |^{-3/2} \big( |\Y|_{} + \|u^\flat\|_{W^{1,\infty}} \big),
\end{align}
and, for $3 \leq p < 6$
\begin{align} \label{est:fluid.velocities.comment} 
\|u^\mathfrak{p} - |\log \eps|^{-1} \sum_{1 \leq i \leq N} U_{\mathcal C_i}[ v^{{\mathcal S}_i}  - u^\flat  ]\|_{L^p_{\loc}} 
\leq C | \log \eps |^{-1 + {\frac 3 p - \frac 1 2 - \delta}} \big( |\Y(t)|_{} + \|u^\flat\|_{W^{1,\infty}} \big).
\end{align}
\end{Proposition}
\begin{proof}
Estimates \eqref{est:fluid.velocities.-1} and \eqref{est:fluid.velocities.0} follow immediately from Theorem \ref{lem:slender.several}. Moreover, \eqref{est:fluid.velocities.1} follows from \eqref{est:fluid.velocities.0} and Sobolev embedding.

Concerning \eqref{est:fluid.velocities.comment} for $3 \leq p < 6$,
we  combine \eqref{est:u-U_i} and \eqref{est:u-U_i.L^p} instead of just relying on \eqref{est:u-U_i.L^p}.
More precisely, combining \eqref{est:u-U_i} and \eqref{est:u-U_i.L^p} with Hölder' inequality yields  for all $3/2 \leq q < 2$ and all $\delta > 0$
\begin{align} \label{est:fluid.velocities.comment.1} 
&\|u^\mathfrak{p} - |\log \eps|^{-1} \sum_{1 \leq i \leq N} U_{\mathcal C_i}[ v^{{\mathcal S}_i}  - u^\flat  ]\|_{W^{1,q}_{\loc}(\R^3 \setminus \cup_i \mathcal S_i)}\\
&\leq C | \log \eps |^{-1 - \frac 3 p + \frac 1 2 + \delta} \big( |\Y(t)|_{} + \|u^\flat\|_{W^{1,\infty}} \big).
\end{align}

Since the $H^1$-estimate in \eqref{est:u-U_i} excludes the sets occupied by the filaments, we had to exclude it in the above estimate. However, the pointwise estimates \eqref{eq:decay.nabla.U}
imply that for all   $q < 2$
\begin{align} \label{est:fluid.velocities.comment.2}
    &\|u^\mathfrak{p} - |\log \eps|^{-1} \sum_{1 \leq i \leq N} U_{\mathcal C_i}[ v^{{\mathcal S}_i}  - u^\flat  ]\|_{L^p(\cup_i \mathcal S_i)}\\
    &\leq  |\log \eps|^{-1} \eps^{2-q} \big( |\Y(t)|_{} + \|u^\flat\|_{W^{1,\infty}} \big).
\end{align}
Combining \eqref{est:fluid.velocities.comment.1} and \eqref{est:fluid.velocities.comment.2} yields \eqref{est:fluid.velocities.comment}.
\end{proof}
In order to conclude that  \eqref{conv-flui-precised} and \eqref{conv-flui-precised.comment} hold true,  we show the following result.
\begin{Lemma} \label{lemma-doubleu}
Let $1 \leq i \leq N$ and denote 
\begin{align}
w := U_{\mathcal C_i}[ v^{{\mathcal S}_i} - u^\flat] - U_{\hat{\mathcal C}_i}[   {\hat v}^{{\mathcal S}_i}(t,\cdot)  - u^\flat (t,\cdot)  ].
\end{align}
Then, there exists $\eps_0 > 0$  such that for all $\eps $ in $(0,\eps_0)$,  for all $t \in [0,T_\eps^{\max}]$,
\begin{itemize}
    \item for
all $1 \leq p < 2$, and all  compact subset $K$ of $\R^3$, there exists  $C $ in $(0,+\infty)$  such that
\begin{equation} \label{omeg-s}
\begin{aligned}
 \|w\|_{L^p (K)}  \leq C |(Q_{\eps,i},h_{i,\eps} + \bar h_{i,\eps}) - (\hat Q_i,\hat h_i)| \big( |\Y|_{} + \|u^\flat\|_{W^{1,\infty}} \big) + C |(v,\omega) - (\hat v,\hat \omega)|.
\end{aligned}
\end{equation}
  \item  for $2 \leq p < 6$, and all  compact subset $K$ of $\R^3$, there exists  $C $ in $(0,+\infty)$  such that
\begin{align}  \label{omeg-s.2}
    \|w\|_{L^p (K)} & \leq C |(Q_{\eps,i},h_{i,\eps} + \bar h_{i,\eps}) - (\hat Q_i,\hat h_i)|^{\frac 3 p - \frac 1 2 - \delta} \big( |\Y|_{} + |\hat \Y|_{} + \|u^\flat\|_{W^{1,\infty}} \big)
  \\ \nonumber  &\quad + C |(v,\omega) - (\hat v,\hat \omega)|
\end{align}
\end{itemize}
\end{Lemma}

Before proving Lemma \ref{lemma-doubleu}, we show how to deduce 
\eqref{conv-flui-precised} and \eqref{conv-flui-precised.comment}.

\begin{proof}[Proof of  \eqref{conv-flui-precised} and \eqref{conv-flui-precised.comment}]

Let $T < \hat T$ and $d = \frac 1 4 \min_{t \in [0,T]} \hat d_{\min}(t)$. Then, by Proposition \ref{prop-lifetime}, for $\eps$ sufficiently small, we have $T_\eps^{\max} \geq T$ and $\dmin \geq d$ on $[0,T]$. Therefore, Proposition \ref{pro:u^p.U} and 
 Lemma \ref{lemma-doubleu}, as well as \eqref{est:main.positions} and 
 \eqref{allezesti} yields \eqref{conv-flui-precised} and \eqref{conv-flui-precised.comment}.
\end{proof}

\begin{proof}[Proof of Lemma \ref{lemma-doubleu}]

We first observe that $w$ satisfies
\begin{align*}
    w(x) &= \frac12\int_{\mathcal C_i} S(x-y) k(y) (v^{{\mathcal S}_i} - u^\flat)(y) \dd \mathcal H^1(y) 
    - \frac12 \int_{\hat {\mathcal C}_i} S(x-y) k(y) (\hat v^{{\mathcal S}_i} - u^\flat)(y) \dd \mathcal H^1(y) \\
      &=: w_1 + w_2,
\end{align*}
where 
\begin{align*}
    w_1 (x) 
    &:= \frac12 \int_{\mathcal C_i} S(x-y) k(y) (v^{{\mathcal S}_i} - u^\flat)(y) \dd \mathcal H^1(y) \\
    &-  \frac12 \int_{\hat {\mathcal C}_i} S(x-y) Q k(\phi(y)) (v^{{\mathcal S}_i} - u^\flat)(\phi(y)) \dd \mathcal H^1(y) ,
    \end{align*}
    and
   \begin{align*}
    w_2 (x)  :=& \frac12 \int_{\hat {\mathcal C}_i} S(x-y) \left(Q k(\phi(y)) (v^{{\mathcal S}_i} - u^\flat)(\phi(y)) - k(y)(\hat v^{{\mathcal S}_i} - u^\flat)(y) \right) \dd \mathcal H^1(y),
\end{align*}
where 
$$Q:=Q_{i,\eps} \hat Q_{i}^T
\text{ and } \phi(x) := Q (x - \hat h_i ) + h_{i,\eps} + \bar h_{i,\eps},$$
is the rigid body motion that transforms $\hat {\mathcal C}_i$ to $ {\mathcal C}_i$.
Then, using $S(Q^Tx) = Q^T S(x) Q$, we find
\begin{align*}
    w_1(x) &= \frac12 \int_{\mathcal C_i} S(x-y) k(y) (v^{{\mathcal S}_i} - u^\flat)(y) \dd \mathcal H^1(y) \\
    &\quad - \frac12 Q^T \int_{\mathcal C_i} S(\phi(x)-y) k(y) (v^{{\mathcal S}_i} - u^\flat)(y) \dd \mathcal H^1(y) \\
    &=  U_{\mathcal C_i}[ v^{{\mathcal S}_i} - u^\flat](x)-  Q^T U_{\mathcal C_i}[ v^{{\mathcal S}_i} - u^\flat](\phi(x)).
\end{align*}
Using the fundamental theorem of calculus, we observe that for any $\psi \in W^{1,s}_\loc${, $s\in[1,\infty]$,}
\begin{align} \label{est:Poincare.psi}
\|\psi - \psi \circ \phi\|_{L^s_\loc} \leq C |(Q_{\eps,i},h_{i,\eps} + \bar h_{i,\eps}) - (\hat Q_i,\hat h_i)| \|\nabla \psi\|_{L^s_\loc}.
\end{align}
Hence, for $p<2$, by recalling \eqref{eq:norm.U_i.subcritical}, we infer that 
\begin{align} \label{est:w_1}
    \|w_1\|_{L^p_{\loc}} \leq C |(Q_{\eps,i},h_{i,\eps} + \bar h_{i,\eps}) - (\hat Q_i,\hat h_i)| \| v^{{\mathcal S}_i} - u^\flat\|_{L^\infty(\mathcal C_i)}.
\end{align}
On the other hand, pointwise bounds  analogous to \eqref{eq:decay.U}
imply for all $p < \infty$
\begin{align} \label{est:w_2}
\begin{aligned}
    \|w_2\|_{L^p_{\loc}} &\leq C |(Q_{\eps,i},h_{i,\eps} + \bar h_{i,\eps}) - (\hat Q_i,\hat h_i)| \left(\| v^{{\mathcal S}_i} - u^\flat\|_{L^{\infty}(\mathcal C_i)} + \| \nabla(v^{{\mathcal S}_i} - u^\flat)\|_{L^{\infty}(\R^3)} \right) \\
    &+ C |(v,\omega) - (\hat v,\hat \omega)|.  
    \end{aligned}
\end{align}
Combining \eqref{est:w_1} and \eqref{est:w_2} yields \eqref{omeg-s}.

\medskip

Finally, we give the  proof of \eqref{omeg-s.2}. Notice that in the case $|(Q_{\eps,i},h_{i,\eps} + \bar h_{i,\eps}) - (\hat Q_i,\hat h_i)| \geq 1$, the estimate follows from \eqref{eq:norm.U_i.subcritical} and the  Sobolev embedding by estimating the two terms separately. Thus, since \eqref{est:w_2} holds for all $p < \infty$, it suffices to show  for $2 \leq p < 6$ and  for $|(Q_{\eps,i},h_{i,\eps} + \bar h_{i,\eps}) - (\hat Q_i,\hat h_i)| \leq 1$,
\begin{align} 
    \|w_1\|_{L^p_{\loc}} \leq C |(Q_{\eps,i},h_{i,\eps} + \bar h_{i,\eps}) - (\hat Q_i,\hat h_i)|^{\frac 3 p - \frac 1 2 - \delta}  \| v^{{\mathcal S}_i} - u^\flat\|_{L^\infty(\mathcal C_i)}
\end{align}
With $\psi$ as above, the critical Sobolev inequality yields for $s < 3$ and $s^\ast = (3s)/(3-s)$
\begin{align} \label{est:Gagliardo}
    \|\psi - \psi \circ \phi\|_{L^{s^\ast}_\loc} \leq \|\nabla \psi\|_{L^s_\loc}.
\end{align}
Therefore,
for any $p < 6$, $s < 2$, $\theta \in [0,1]$ such that
\begin{align} \label{eq:numerology}
    \frac 1 p = \frac \theta s + \frac{1 - \theta} {s^*},
\end{align} 
Hölder's inequality, \eqref{est:Poincare.psi} and \eqref{est:Gagliardo} yield 
\begin{align}
    \|\psi - \psi \circ \phi\|_{L^{p}_\loc} \leq |(Q_{\eps,i},h_{i,\eps} + \bar h_{i,\eps}) - (\hat Q_i,\hat h_i)|^{\theta}   \|\nabla \psi\|_{L^s_\loc}
\end{align}
Elementary calculations show that for all $2 \leq p<6$ and any $\delta > 0$  we can choose   $s < 2$, such that \eqref{eq:numerology} holds with 
\begin{align}
    \theta = \frac 3 p - \frac 1 2 - \delta.
\end{align}
Hence, 
\begin{align} \label{est:w_1.comment}
    \|w_1\|_{L^p_{\loc}} \leq C |(Q_{\eps,i},h_{i,\eps} + \bar h_{i,\eps}) - (\hat Q_i,\hat h_i)|^{\frac 3 p - \frac 1 2 - \delta}  \| v^{{\mathcal S}_i} - u^\flat\|_{L^\infty(\mathcal C_i)}.
\end{align}
This concludes the proof of Lemma \ref{lemma-doubleu}.
\end{proof}

\smallskip 
\smallskip 
\smallskip 
\noindent
\subsection*{Acknowledgment}

 The three authors are partially supported by the Agence Nationale de la Recherche,  Project BORDS, grant ANR-16-CE40-0027-01 and by the IDEX of the University of Bordeaux, project BOLIDE. F.S. and C.P. also acknowledge support by the Agence Nationale de la Recherche, Project SINGFLOWS, grant ANR-18-CE40-0027-01. 
  R.H. is supported by the German National Academy of Science Leopoldina, grant LPDS 2020-10. C.P. is also partially supported by the project CRISIS, grant ANR-20-CE40-0020-01. 
  In addition, F.S. is supported by the Agence Nationale de la Recherche, Project IFSMACS, grant ANR-15-CE40-0010.
 F.S. is also partially supported by the H2020-MSCA-ITN-2017 program, Project ConFlex, Grant ETN-765579. This work was partly accomplished while F.S. was participating in a program hosted by the Mathematical Sciences Research Institute in Berkeley, California, during the Spring 2021 semester, and supported by the National Science Foundation under Grant No. DMS-1928930.

%%%%%%%%%%%%%%%%%%%%%%%%%%%%%%%%%%%%
%

\Addresses


\begin{thebibliography}{10}
\bibliographystyle{plain}

\bibitem{AcostaDuranMuschietti06} Acosta, G., Durán, R.G., Muschietti, M.A. (2006). Solutions of the divergence operator on John domains. Advances in Mathematics, 206(2), pp. 373-401.

\bibitem{BCD} 
Badra, M., Caubet, F.,  Dambrine, M. (2011). Detecting an obstacle immersed in a fluid by shape optimization methods. Mathematical Models and Methods in Applied Sciences, 21(10), 2069-2101.



\bibitem{BV} 
Banica, V.,  Vega, L. (2011).  Self-similar solutions of the binormal flow and their stability,  
Panoramas et Synthèses 38 (2012), 1-35. 

 \bibitem{Batchelor} 
  Batchelor, G., Slender-body theory for particles of arbitrary cross-section in Stokes flow. J. Fluid Mech., 44(3):419-440, 1970.
 


\bibitem{BG}
Berkowitz, J., Gardner, C.S.: On the asymptotic series expansion of the motion of a charged particle in slowly varying fields. Commun. Pure Appl. Math. 12, 501-512 (1959)


\bibitem{BNDRDM1} 
Bonnaillie-No\"el, V., Dalla Riva, M., Dambrine, M.,  Musolino, P. (2018). A Dirichlet problem for the Laplace operator in a domain with a small hole close to the boundary. Journal de Mathématiques Pures et Appliquées, 116, 211-267.

\bibitem{BNDRDM2} 
Bonnaillie-No\"el, V., Dalla Riva, M., Dambrine, M.,  Musolino, P. (2020). Global representation and multiscale expansion for the Dirichlet problem in a domain with a small hole close to the boundary. Communications in Partial Differential Equations, 1-28.

 \bibitem{Cox} 
Cox., R., The motion of long slender bodies in a viscous fluid part 1. general theory. J. Fluid Mech., 44(4):791-810, 1970.

\bibitem{Fax} 
Fax\'en, H. (1922). Der Widerstand gegen die Bewegung einer starren Kugel in einer z\"ahen Fl\"ussigkeit, die zwischen zwei parallelen ebenen W\"anden eingeschlossen ist. Annalen der Physik, 373(10), 89-119.

\bibitem{FNN} 
Feireisl, E., Namlyeyeva, Y.,  Necasov\'a, S. (2016). Homogenization of the evolutionary Navier-Stokes system. Manuscripta Mathematica, 149(1-2), 251-274.

\bibitem{Galdi} 
  Galdi, G. (2011). An introduction to the mathematical theory of the Navier-Stokes equations: Steady-state problems. Springer Science $\&$ Business Media.
 
  \bibitem{Galdi2} 
    Galdi, G.P. On the motion of a rigid body in a viscous liquid: a mathematical analysis with applications. In: Handbook of Mathematical Fluid Dynamics, vol. I, pp. 653-791. North-Holland, Amsterdam (2002).
  
  \bibitem{GS}
  Glass, O.,  Sueur, F. (2019). Dynamics of several rigid bodies in a two-dimensional ideal fluid and convergence to vortex systems. arXiv preprint arXiv:1910.03158.
  
  \bibitem{Goetz} Götz, T. (2000). Interactions  of  fibers  and  flow:   asymptotics,  theory  and  numerics.   Doctoral dissertation, University of Kaiserslautern.
  
  \bibitem{Gonzalez}
  Gonzalez, O. (2021). Theorems on the Stokesian Hydrodynamics of a Rigid Filament in the Limit of Vanishing Radius. SIAM Journal on Applied Mathematics, 81(2), 551-573.
  
 \bibitem{Hancock} 
 Hancock,  G., The self-propulsion of microscopic organisms through liquids. Proc. R. Soc. Lond. A, 217(1128):96-121, 1953.
 
 \bibitem{HP20} 
 Higaki, M., Prange, C. (2020). Regularity for the stationary Navier--Stokes equations over bumpy boundaries and a local wall law. Calculus of Variations and Partial Differential Equations, 59(4), 1-46.
 

 
  \bibitem{HMS} 
 Hillairet, M., Moussa, A.,  Sueur, F. (2019). On the effect of polydispersity and rotation on the Brinkman force induced by a cloud of particles on a viscous incompressible flow. Kinetic $\&$ Related Models, 12(4).
 
  \bibitem{HW} 
 Hsiao, G. C.,  Wendland, W. L. (2008). Boundary integral equations (pp. 25-94). Springer Berlin Heidelberg.
 
   \bibitem{Jeffery}  Jeffery, G. B. (1922). The motion of ellipsoidal particles immersed in a viscous fluid. In Proc. R. Soc. Lond. A, Vol. 102, No. 715, pp. 161-179. The Royal Society.
   
   \bibitem{JS}
   Jerrard, R. L.,  Seis, C. (2017). On the vortex filament conjecture for Euler flows. Archive for Rational Mechanics and Analysis, 224(1), 135-172.
   
   
   \bibitem{Johnson} Johnson, R. E. (1980). An improved slender-body theory for Stokes flow. Journal of Fluid Mechanics, 99(2), 411-431.
      
 \bibitem{Junk-Illner}      Junk, M.,  Illner, R. (2007). A new derivation of Jeffery's equation. Journal of Mathematical Fluid Mechanics, 9(4), 455-488.

\bibitem{Kim-Karilla}    
       Kim, S.,  Karrila, S. J. (2013). Microhydrodynamics: principles and selected applications. Courier Corporation.
  
 

 
 \bibitem{KR} 
 Keller, J. B. and Rubinow, S. I.. Slender-body theory for slow viscous flow. J. Fluid Mech, 75(4):705-714, 1976.
 
    \bibitem{Lady} 
  Ladyzhenskaya, O. A. (1969). The mathematical theory of viscous incompressible flow (Vol. 12, No. 3). New York: Gordon $\&$ Breach.
  
  
\bibitem{MNP}  Maz'Ya, V., Nazarov, S.,  Plamenevskij, B. (2012). Asymptotic theory of elliptic boundary value problems in singularly perturbed domains (Vol. II). Birkh\"auser.

\bibitem{Miot}
Miot, E. (2017). Le flot binormal, l?équation de Schrödinger et les tourbillons filamentaires. Asterisque, 2015(390), 427-451.

\bibitem{Mori-Ohm} 
Mori, Y., Ohm, L., (2020). An error bound for the slender body approximation of a thin, rigid fiber sedimenting in Stokes flow. Research in the Mathematical Sciences, vol. 7, no 2, p. 1-27.

\bibitem{Mori-Ohm-free-ends} 
Mori, Y., Ohm, L., Spirn, D., (2020). Theoretical justification and error analysis for slender body theory with free ends. Archive for Rational Mechanics and Analysis, 235(3), 1905-1978.

\bibitem{Mori-Ohm-Accuracy} 
Mori, Y.,  Ohm, L. (2020). Accuracy of slender body theory in approximating force exerted by thin fiber on viscous fluid. arXiv preprint arXiv:2008.06829.

\bibitem{Mori-Ohm-Spirn} 
Mori, Y., Ohm, L.,  Spirn, D. (2018). Theoretical justification and error analysis for slender body theory. Communications on Pure and Applied Mathematics.

\bibitem{Ohm-numerics} 
Ohm, L., Tapley, B. K., Andersson, H. I., Celledoni, E.,  Owren, B. (2019). A slender body model for thin rigid fibers: validation and comparisons. arXiv preprint arXiv:1906.00253.

\bibitem{Peskin}
 Peskin, C. S. (2002). The immersed boundary method. Acta numerica, 11:479-517.

\bibitem{Pozrikidis}
Pozrikidis, C. (1992). Boundary integral and singularity methods for linearized viscous flow. Cambridge university press.

\bibitem{serre-japan}
Serre, D. (1987). Chute libre d?un solide dans un fluide visqueux incompressible. Existence. Japan Journal of Applied Mathematics, 4(1), 99-110.

\bibitem{simon}
Simon, J. (1991). Domain variation for drag in Stokes flow. In Control Theory of Distributed Parameter Systems and Applications (pp. 28-42). Springer, Berlin, Heidelberg.

\bibitem{Tornberg}
Tornberg, A-K. (2020). Accurate evaluation of integrals in slender-body formulations for fibers in viscous flow. arXiv preprint arXiv:2012.12585


\bibitem{VM}
Vibe, A.,  Marheineke, N. (2018). Modeling of macroscopic stresses in a dilute suspension of small weakly inertial particles. Kinetic \& Related Models, 11(6), 1443.

\end{thebibliography}
\end{document}